\documentclass[a4paper]{article}

\usepackage[a4paper,width=150mm,top=25mm,bottom=35mm]{geometry}

\usepackage{amsmath}
\usepackage{amssymb}
\usepackage{amsthm}
\usepackage{bbm}
\usepackage{color}

\numberwithin{equation}{section}

\author{Edgar Assing}
\title{On sup-norm bounds part II: $GL(2)$ Eisenstein series}

\theoremstyle{plain}
\newtheorem{lemma}{Lemma}[section]
\newtheorem{prop}{Proposition}[section]
\newtheorem{theorem}{Theorem}[section]
\newtheorem{cor}{Corollary}[section]
\newtheorem{rem}{Remark}[section]

\DeclareMathOperator{\R}{\mathbb{R}}
\DeclareMathOperator{\C}{\mathbb{C}}
\DeclareMathOperator{\A}{\mathbb{A}}
\DeclareMathOperator{\Q}{\mathbb{Q}}
\DeclareMathOperator{\Z}{\mathbb{Z}}
\DeclareMathOperator{\N}{\mathbb{N}}
\DeclareMathOperator{\No}{\mathcal{N}}

\DeclareMathOperator{\n}{\mathfrak{n}}
\DeclareMathOperator{\p}{\mathfrak{p}}

\DeclareMathOperator{\op}{\mathfrak{o}}

\DeclareMathOperator{\vol}{\text{Vol}}

\DeclareMathOperator{\sgn}{\text{sgn}}

\newcommand{\red}[1]{{#1}}
\newcommand{\redd}[1]{{#1}}
\newcommand{\abs}[1]{\left\vert #1 \right\vert}
\newcommand{\du}[1]{\underline{\underline{#1}}}
\newcommand{\nilp}[1]{\left(\begin{matrix} 1&#1 \\ 0&1\end{matrix}\right)}
\newcommand{\cent}[1]{\left(\begin{matrix} #1&0 \\ 0&#1\end{matrix}\right)}
\newcommand{\am}[1]{\left(\begin{matrix} #1&0 \\ 0&1\end{matrix}\right)}

\begin{document}

\maketitle

\begin{abstract}
In this paper we consider the sup-norm problem in the context of analytic Eisenstein series for \red{$GL_2$} over number fields. We prove a hybrid bound which is sharper than the corresponding bound for Maa\ss\  forms. Our results generalize those of Huang and Xu where the case of Eisenstein series of square-free levels over the base field $\Q$ had been considered.
\end{abstract}

\tableofcontents

\section{Introduction}

The general notion of an automorphic form includes the Eisenstein series. Many of the problems in
analytic number theory that have been studied for cusp forms can be studied for Eisenstein series as well. One example is Quantum unique ergodicity. With slight modification it carries over to Eisenstein series, see \cite{LS95}, \cite{Tr11}, and \cite{Zh16}. Another measure of equidistribution is the $L^{\infty}$-norm. The sup-norm problem is very popular for cusp forms in many different settings. See \cite{BHMM16}, \cite{Sa15}, and the references within for more information on this subject. But also this problem carries over to Eisenstein series. This was done \red{in} \cite{Yo15} for the unitary Eisenstein series on $SL_2(\Z)\setminus \mathbb{H}$ and for congruence subgroups of square free level in \cite{HX16}. In this paper we will generalize their work to the number field setting. We go even further and allow for arbitrary level and central character.

The pure fact that Eisenstein series link the continuous spectrum of a reductive group $G$ to the spectrum of lower rank groups gives rise to \red{an} intriguing interplay between say $GL_2$ and $GL_1$ theories. An example of this is the Burgess type subconvexity bound for $GL_1$ which is proved in \red{\cite{Wu16_2}} by applying $GL_2$ results to Eisenstein series. Another nice example is \cite{Sa04}, where the theory of Eisenstein series is used to give lower bounds for $\zeta(1+it)$. This last method has vast generalizations. See for example \cite{GL06}.
Our case is reverse. We exploit that the theory of Hecke $L$-functions over number fields is \red{well-developed} and use this to produce an improved amplifier. This leads to the remarkable fact that (as in \cite{Yo15}) we get sup-norm bounds which are in some aspects better than the state of the art results for cusp forms. It was already observed in \cite{IS95} that one can improve the sup-norm bound in the spectral aspect if one has a better understanding of the Hecke eigenvalues. This is exactly the ingredient we gain from the classically well\red{-}developed $GL_1$ theory. More precisely, we use highly non-trivial \red{zero-free} regions for Hecke $L$-functions to derive an asymptotic expansion for generalized divisor sums. This will serve as a lower bound for the amplifier.

Note that recently in \red{\cite{Ju16}, as well as in \cite{RP17},} it was shown that the quality of the sup-norm bound that we achieve for Eisenstein series holds for $SL_2(\Z)$-Maa\ss\  forms on average.
 
Before we can give a precise statement of our theorems we have to introduce some notation. We assume that the reader is familiar with the results and the notation in \cite{As17_2} and \red{mainly} focus on the aspects where Eisenstein series differ from cusp forms. 

\redd{\textbf{Acknowledgements.} I would like to thank the referee for carefully reading the manuscript, which has led to a significant improvement of this work. }

\subsection{Set-up and basic definitions} \label{sec:set_uP}

\red{Let $F$ be a number field of degree $n$, class number $h_F$, discriminant $d_F$, different ideal $\mathfrak{d}$ and ring of integers $\mathcal{O}_F$. The field $F$ comes with $r_1$ real embeddings and $r_2$ pairs of complex embeddings. These make up the archimedean places of $F$ which are denoted by $\nu$. At the same time we use $\nu$ for the corresponding embedding $\nu\colon F\to F_{\nu}$, where $F_{\nu}$ is either $\R$ or $\C$ depending on the type of $\nu$. We equip the local fields $F_{\nu}$ with the Lebesgue measure $\mu_{\nu}$ coming  from either $\R$ or $\C$. Further we let $\abs{\cdot}$ be the standard absolute value on $F,\R \subset \C$ and set $\abs{\cdot}_{\nu} = \abs{\cdot}^{[F_{\nu}\colon \R]}$. The non-archimedean places of $F$ are associated to prime ideals $\p$ of $\mathcal{O}_F$. They come with corresponding (canonically normalized) valuation $v_{\p}(\cdot)$ and absolute value $\abs{\cdot}_{\p} = q_{\p}^{-v_{\p}(\cdot)}$, where $q_{\p}=\mathcal{N}(\p)$. The local fields $F_{\p}$ have ring of integers $\mathfrak{o}_{\p}$, uniformizer $\varpi_{\p}$ and unique prime ideal $\p = (\varpi_{\p})\subset \mathfrak{o}_{\p}$. Let $\mu_{\p}$ (respectively $\mu_{\p}^{\times}$) be the Haar measure on $(F_{\p},+)$ (respectively $(F^{\times},\cdot)$) normalized by $\mu_{\p}(\mathfrak{o}_{\p}) = 1$ (respectively $\mu_{\p}^{\times}(\mathfrak{o}^{\times}_{\p}) = 1$).

Define $F_{\infty} = \prod_{\nu} F_{\nu}$ and $\abs{\cdot}_{\infty} = \prod_{v}\abs{\cdot}_{\nu}$. We use $\abs{\cdot}_{\R}$ (respectively $\abs{\cdot}_{\C}$) to denote the part of $\abs{\cdot}_{\infty}$ coming from the real (respectively complex) embeddings only. Let $\A_{\text{fin}}$ denote the finite ad\'eles equipped with the absolute value $\abs{\cdot}_{\text{fin}}$ being the product of all the local absolute values. The full ad\'ele ring is given by $\A_F = F_{\infty}\times \A_{\text{fin}}$ and equipped with $\abs{\cdot}_{\A}$ and $\mu$ in the usual manner. We also define the set of totally positive field elements $F^+ $ to contain all $x\in F$ such that $x_{\nu} >0$ for all real $\nu$. Furthermore, put $F^0(\A_F)=\{ a\in \A_F\colon \abs{a}_{\A_F}=1 \}$ and $F_{\infty}^+= \R^+ \subset F_{\infty}$ diagonally. Finally, we choose ideal representatives $\theta_1, \dots, \theta_{h_F}\in \hat{\mathcal{O}}_{F}$.

Write $[\mathfrak{m}]_{\n}= \frac{\mathfrak{m}}{(\mathfrak{m},\n^{\infty})}$ for the coprime-to-$\n$ part of an ideal $\mathfrak{m}$. 

To a Hecke character $\chi\colon F^{\times}\setminus \A_F^{\times} \to \C$ we associate the corresponding completed $L$-function
\begin{equation}
	\Lambda(s,\chi) = \underbrace{\prod_{\nu} L_{\nu}(s,\chi_{\nu})}_{= \gamma_{\infty}(s,\chi)} \underbrace{\prod_{\p} L_{\p}(s,\chi_{\p}) }_{=L(s,\chi)}. \nonumber
\end{equation}
For the trivial character $\chi=1$ the local factors are $ L_{\p}(1,s) = \zeta_{\p}(s) =  (1-q_{\p}^{-s})^{-1}$, $L_{\nu}(1,s) = \Gamma_{\R}(s) = \pi^{-\frac{s}{2}}\Gamma(\frac{s}{2})$ if $\nu$ is real, and $L_{\nu}(1,s) =\Gamma_{\C}(s) = 2(2\pi)^{-s}\Gamma(s)$ otherwise. Further, put $\zeta_{\n}(s) = \prod_{\p\mid\n} \zeta_{\p}(s)$. 

Let $R$ be a commutative ring with 1. In practice $R$ will be one of the objects introduced above. We set $G(R) = GL_2(R)$ and define the subgroups
\begin{eqnarray}
	Z(R) = \left\{  z(r)=\cent{r} \colon r\in R^{\times} \right\},& \quad A(R) = \left\{  a(r)=\am{r} \colon r\in R^{\times} \right\},  \nonumber\\
	N(R) = \left\{  n(x)=\nilp{x} \colon x\in R \right\}& \text{ and } \quad B(R) = Z(R)A(R)N(R). \nonumber
\end{eqnarray}
Choose the maximal compact subgroups
\begin{eqnarray}
	K_{\nu} &=& \begin{cases} U_2(\C) & \text{ if $\nu$ is complex,} \\ O_2(\R) & \text{if $\nu$ is real} \end{cases} \subset G(F_{\nu}), \nonumber\\
	K_{\p} &=& GL_2(\op_{\p}) \subset G(F_{\p}) \text{ and } \nonumber \\
	K_{\infty} &=& \prod_{\nu} K_{\nu} \subset G(F_{\infty}). \nonumber 
\end{eqnarray}
For $R=F_{\p}$ we further define
\begin{eqnarray}
	K_{\p}^0(n) &=& K_{\p} \cap \left[ \begin{matrix} \op_{\p}&\varpi_{\p}^n\op_{\p} \\ \op_{\p} &\op_{\p}\end{matrix} \right], K_{0,\p}(n) = K_{\p} \cap \left[ \begin{matrix} \op_{\p}&\op_{\p} \\ \varpi_{\p}^n\op_{\p} &\op_{\p}\end{matrix} \right] \text{ and }\nonumber\\
	K_{1,\p}(n) &=& K_{\p} \cap \left[ \begin{matrix} 1+\varpi_{\p}^n\op_{\p}&\op_{\p} \\ \varpi_{\p}^n\op_{\p} &\op_{\p}\end{matrix} \right], K_{2,\p}(n) = K_{\p} \cap \left[ \begin{matrix} \op_{\p}&\op_{\p} \\ \varpi_{\p}^n\op_{\p} &1+ \varpi_{\p}^n\op_{\p}\end{matrix} \right]. \nonumber
\end{eqnarray}
Globally we put
\begin{equation}
	K_1(\n) = K_{\infty} \prod_{\p} K_{1,\p}(v_{\p}(\n)) \text{ and } K=K_{\infty} \prod_{\p}K_{\p}. \nonumber
\end{equation}
The long Weyl element is given by
\begin{equation}
	\omega=\left(\begin{matrix} 0&1 \\ -1&0\end{matrix}\right). \nonumber
\end{equation}

These groups are equipped with the following measures. Locally we use the identifications $N(R) = (R,+)$, $A(R) = R^{\times}$, and $Z(R) = R^{\times}$ to transport the measures defined on the local fields to the corresponding groups. The compact groups $K_{\p}$ and $K_{\nu}$ are equipped with the unique probability Haar measure on $\mu_{K_{\p}}$ and $\mu_{K_{\nu}}$.
Globally we use the following product measure on $K$, $A(\A_F)$ and $N(\A_F)$:
\begin{equation}
	\mu_{K} = \prod_{\nu} \mu_{K_{\nu}} \prod_{\p} \mu_{K_{\p}}, \quad \mu_{A(\A_F)} = \prod_{\nu} \mu_{\nu}^{\times} \prod_{\p} \mu_{\p}^{\times} \text{ and } \mu_{N(\A_F)} = \frac{2^{r_2}}{\sqrt{\abs{d_F}}} \prod_{\nu} \mu_{\nu} \prod_{\p} \mu_{\p}. \nonumber
\end{equation}
Finally, we define 
\begin{equation}
	\int_{Z(\A_F)\setminus G(\A_F)} f(g) d\mu(g) = \int_K \int_{\A_F^{\times}} \int_{N(\A_F)} f(na(y) k) d\mu_{N(\A_F)}(n) \frac{d\mu^{\times}_{\A_F^{\times}}(y)}{\abs{y}} d\mu_K(k) \label{eq:integral_convention}
\end{equation}
as in \cite{GJ79}.}

We define the function $H\colon G(\A_F) \to \R_{+}$ via the Iwasawa decomposition as follows
\begin{equation}
	H\left( \left( \begin{matrix} 1&x \\ 0&1 \end{matrix} \right) \left( \begin{matrix} a&0 \\ 0&b \end{matrix} \right) k \right) = \abs{\frac{a}{b}}_{\A_F} \text{ for all } k\in K. \nonumber
\end{equation} 
$H$ factors in the obvious way. We have $H=\prod_{\nu}H_{\nu}\prod_{\p}H_{\p}$.

Fix two unitary characters $\chi_1,\chi_2\colon F^{\times} \setminus \A_{F}^{\times}\to \C$ such that $\chi_1\chi_2^{-1}\vert_{F_{\infty}^+} = 1$. We associate a family of admissible $G(\A_F)$-representations by defining the twisted versions of these characters 
\begin{equation}
	\chi_1(s) = \abs{\cdot}_{\A_F}^s\chi_1 \text{ and } \chi_2(-s) = \abs{\cdot}_{\A_F}^{-s} \chi_2 \nonumber
\end{equation}
and introducing the global principal series representations
\begin{equation}
	(\pi(s), \mathbf{H}(s)) = (\chi_1(s)\boxplus \chi_2(-s), \red{\mathcal{B}}(\chi_1(s),\chi_2(-s))). \nonumber
\end{equation} 

It is well\red{-}known that if $s\neq \pm\frac{1}{2}$ these representations are irreducible. As usual we have the factorization
\begin{equation}
	\pi(s) = \bigotimes_{\nu}\big(\chi_{1,\nu}(s)\boxplus \chi_{2,\nu}(-s) \big)\bigotimes_{\p}\big(\chi_{1,\p}(s)\boxplus \chi_{2,\p}(-s)\big). \nonumber
\end{equation}

We can view $\red{\mathbf{H}(s)}$ as a trivial holomorphic fibre bundle over $\red{\mathbf{H}}=\red{\mathbf{H}}(0)$. Thus, $v\in \red{\mathbf{H}}$ gives rise to a section
\begin{equation}
	[v(s)](g) = v(g)\cdot H(g)^s \in \red{\mathbf{H}}(s). \nonumber
\end{equation}
To such a section we associate the Eisenstein series
\begin{equation}
	E_v(s,g) = \sum_{\gamma\in B(F)\setminus G(F)} [v(s)](\red{\gamma}g). \label{eq:averaging_def} 
\end{equation}
This is well\red{-}defined for $\Re(s)>\frac{1}{2}$ but can be meromorphically continued to all $\C$. For more details on this see \cite[Section~5]{GJ79}. 

\red{We assume that at} the archimedean places the characters $\chi_1$ and $\chi_2$ are given by
\begin{eqnarray}
	\chi_j(y) &=&  \abs{y}_{\nu}^{ it_{\nu,j}}\sgn(y)^{m_{\nu}} \text{ if $\nu$ is real,} \nonumber \\
	\chi_j(re^{i\theta})  &=& r^{ i2t_{\nu,j}} e^{i m_{\nu} \theta} \text{ else.} \nonumber
\end{eqnarray}  
Analogously to the Maa\ss\  form situation considered in \cite{As17_2} the parameters \red{$t_{\nu,1}$ and $t_{\nu,2}$} describe the spectral properties of $E_v(s,g)$. In view of this we define 
\begin{equation}
	\red{t_{\nu} = (t_{\nu,1}-t_{\nu,2})/2 \text{ and } s_{\nu} = t_{\nu,1}+t_{\nu,2}.}\nonumber 
\end{equation}
Then the correct spectral parameter for the Eisenstein series $E_v(s,\cdot)$ is $(\lambda_{\nu}(s))_{\nu}$, where
\begin{equation}
	\lambda_{\nu}(s) = \begin{cases}
		\frac{1}{4}+(t_{\nu}+s)^{2} &\text{ if $\nu$ is real,} \\
		1+4(t_{\nu}+s)^2 &\text{ if $\nu$ is complex.}	
	\end{cases}
\end{equation}

By \cite{Sc02} the log-conductor of $\pi_{\p}(s) =\chi_{1,\p}(s)\boxplus \chi_{2,\p}(-s)$ is given by $n_{\p} = a(\chi_{1,\red{\p}})+a(\chi_{2,\red{\p}})$. Therefore, the conductor of $\pi(s)$ is $\n = \prod_{\p} \p^{n_{\p}}$. \red{Furthermore, the central character of $\pi_{\p}(s)$ is independent of $s$ and given by $\omega_{\pi(0),\p} =\chi_{1,\p} \chi_{2,\p}$. Its log-conductor is denoted by $m_{p}$. We write $\omega_{\pi(0)}=\chi_1\chi_2$ and $\mathfrak{m}$ for the corresponding global objects.} Next, we want to fix a new vector for $\pi(s)$. Due to the explicit construction of $\pi(s)$ as principal series, we can be very precise. This is important because it normalizes the associated Eisenstein series. Before we continue let us define the character $\chi\colon \red{B(\A_F)}\to \C^{\times}$ by
\begin{equation}
	\chi\left( \left(\begin{matrix} a& b \\ 0 &d\end{matrix} \right)\right) = \chi_1(a)\chi_2(d). \nonumber
\end{equation}
The new vector $v^{\circ}(s)\in \mathbf{H}(s)$ is defined locally by
\begin{eqnarray}
	v^{\circ}_{\nu}(s)(bk) &=& \chi_{\nu}(b)H_{\nu}(\red{b})^{\frac{1}{2}+s} \text{ for all } \nu, \label{eq:arch_new_vec}\\
	v^{\circ}_{\p}(s)(g) &=&  \begin{cases} \chi_{\p}(b)H_{\p}(g)^{\frac{1}{2}+s} &\text{ if } \p\nmid \n \text{ and } g=bk, \\  \omega_{\red{\pi(0),\p}}(\det(g))\red{q_{\p}^{-\frac{a(\chi_{1,\p})}{2}}}f_s(g) &\text{ else.}  \end{cases}  \label{eq:non_arch_new_vec}
\end{eqnarray}
\red{Here} we take $f$ to be the function
\begin{equation}
	f_s(g) =\begin{cases} \chi_{1,\p}(s)(\varpi_{\p}^{-a(\chi_{2\red{,\p}})})\chi_{1,\p}(a)\chi_{2,\p}(d) \abs{\frac{a}{d}}^{\frac{1}{2}+s}_{\p} &\text{ if } g\in \left(\begin{matrix} a& * \\ 0 & d \end{matrix} \right)\left(\begin{matrix} 1&0 \\ \varpi_{\p}^{a(\chi_{2\red{,\p}})} & 1 \end{matrix} \right) K_{2,\p}(n_{\p}), \\ 0 &\text{ else.} \end{cases} \label{eq:def_fs}
\end{equation}
This definition is taken from \cite[Proposition~2.1.2]{Sc02}. Note that we \red{rescaled} it by $\red{q_{\p}^{-\frac{a(\chi_{1,\p})}{2}}}$, and we twisted it by $\omega_{\red{\pi(0),\p}}$ to make it $K_{1,\p}(\red{n}_{\p})$ invariant. Swapping the roles of $\chi_1$ and $\chi_2$ \red{leads to} a completely analogous situation. The corresponding new vector defined as above will be denoted by $\hat{v}^{\circ}$.

Throughout this work we are mainly concerned with Eisenstein series attached to new vectors. Thus, for sake of notation, we set
\begin{equation}
	E(s,g) = E_{v^{\circ}}(s,g). \nonumber
\end{equation}
The dual Eisenstein series 
\begin{equation}
	\hat{E}(s,g) = E_{\hat{v}^{\circ}}(s,g) \nonumber
\end{equation}
will also play an important role. \red{Furthermore, to each $\mathfrak{L}\mid \n$ we associate the new vector $v^{\circ}_{\mathfrak{L}}$ in the (twisted) principal series representation
\begin{equation}
	 \bigg(\pi^{\mathfrak{L}}(s), \mathbf{H}_{\mathfrak{L}}(s)\bigg) =  \bigg(\omega_{\pi}^{-1}\omega_{\pi}^{\mathfrak{L}}\chi_1(s)\boxplus \omega_{\pi}^{-1}\omega_{\pi}^{\mathfrak{L}}\chi_2(-s), \red{\mathcal{B}}(\omega_{\pi}^{-1}\omega_{\pi}^{\mathfrak{L}}\chi_1(s),\omega_{\pi}^{-1}\omega_{\pi}^{\mathfrak{L}}\chi_2(-s))\bigg) \nonumber
\end{equation}
defined locally as in \eqref{eq:arch_new_vec} and \eqref{eq:non_arch_new_vec} above. As in \cite[Section~2.3]{As17_2} the character $\omega_{\pi}^{\mathfrak{L}}$ is the $[\mathfrak{m}]_{\mathfrak{L}}$-part of $\omega_{\pi}$ and as such it is in particular independent of $s$. Important for us is the relation 
\begin{equation} 
	\omega_{\pi,\p}^{-1}\omega_{\pi,\p}^{\mathfrak{L}}\vert_{\op_{\p}^{\times}} = \begin{cases} \omega_{\pi,\p}^{-1}\vert_{\op_{\p}^{\times}} &\text{ if } \p\mid\mathfrak{L},\\ 1 & \text{ if } \p \nmid \mathfrak{L}. \end{cases}
\end{equation}
given in \cite[(2.2)]{As17_1}, which defines the local \redd{constituents} of $\omega_{\pi}^{\mathfrak{L}}$ up to (unitary) unramified twist. These twists can be specified, however this is not  important for the following argument. We associate the Eisenstein series
\begin{equation}
	E^{\mathfrak{L}}(s,g) = E_{v^{\circ}_{\mathfrak{L}}}(s,g). \nonumber
\end{equation}
}

\red{Finally, recall the sets $\mathcal{J}_{\n}\subset G(\A_{\text{fin}})$ and $\mathcal{F}_{\n_2}\subset  \bigsqcup_{1\leq i\leq h_F} a(\theta_i)B(F_{\infty})$ as defined in \cite[Section~2.2]{As17_2}. Recall also that
\begin{equation}
\eta_{\mathfrak{L}} =\prod_{\p\mid \mathfrak{L}}\left(\begin{matrix} 0 & 1 \\ \varpi_{\p}^{n_{\p}} & 0 \end{matrix} \right)\prod_{\p\nmid\mathfrak{L}} 1 \in G(\A_{\text{fin}}). \nonumber 
\end{equation}}

Note that for $s\in i\R$, when the representation underlying $E(s,\cdot)$ is unitary, the Eisenstein series is not an element of $L^2(G(F)\setminus G(\A_F),\red{\omega_{\pi(0)}})$. \red{Nonetheless}, we will be able to use the spectral theory of this space (and some additional tricks) to prove upper bounds for this function.

\subsection{Statement of results}

We are ready to state the main theorem. In contrast to the cusp form case (see \cite{As17_2}) one can not expect uniform bounds which are valid on the whole space. This is due to the presence of a constant term in the Whittaker expansion. However, the asymptotic size of this constant term is well\red{-}understood. Therefore, what we really prove is a bound for $\abs{E(s,g)}$ in terms of the constant term with a uniform error bound. 

\begin{theorem} \label{th:main_th_3}
Let $E(s,\cdot)$ be the Eisenstein series with underlying \red{unitary} characters $\chi_1$ and $\chi_2$. \red{Write $\n=\n_2\n_0^2$ where $\n_2$ is square free and set $\mathfrak{m}_1 = \frac{\mathfrak{m}}{(\mathfrak{m},\n_2\n_0)}$.} Fix $T_0\in \redd{[1,\infty)}$, define $(T_{\nu})_{\nu}=(\max(\frac{1}{2},\abs{t_{\nu}+T_0}))_{\nu}$ and let $\mathfrak{l}$ denote the conductor of $\chi_1\chi_2^{-1}$. \red{We have
\begin{equation}
	\abs{E(iT_0,g)} =  \abs{E^{\mathfrak{L}}(iT_0,a(\theta_i)g')}, \nonumber
\end{equation}
 for some $\mathfrak{L}\mid\n_2$, $1\leq i \leq h$, and some $g'\in \mathcal{J}_{\n}\times\mathcal{F}_{\n_2}$.} If
\begin{equation}
	\log(\No(\n)) \ll \log(\abs{T}_{\infty}) \text{ and } \log(\No(\mathfrak{l}))\ll \log(\abs{T}_{\infty})^{1-\delta} \text{ for some } \delta>0,\nonumber
\end{equation} 
then we have
\begin{eqnarray}
	E(iT_0,\red{a(\theta_i)g'}) &=& [v^{\circ}(iT_0)](\red{a(\theta_i)g'}) + c(iT_0)[\hat{v}^{\circ}(-iT_0)](\red{a(\theta_i)g'}) \nonumber \\
	&&\qquad\qquad +O\bigg(\No(\n)^{\epsilon}\abs{T}_{\infty}^{\epsilon} \No(\n_0\mathfrak{m}_1)^{\frac{1}{2}}\No(\n_2)^{\frac{1}{4}}\left(\abs{T}_{\infty}^{\frac{3}{8}}  +\abs{T}_{\R}^{\frac{3}{16}}\abs{T}_{\C}^{\frac{7}{16}}\right) \nonumber \\
	&&\qquad\qquad\qquad\qquad+ \No(\n)^{\epsilon}\abs{T}_{\infty}^{\epsilon}\No(\n_0)\bigg)\red{,} \nonumber
\end{eqnarray}
\red{for
\begin{equation}
	 c(s) = \No(\mathfrak{d})^{-\frac{1}{2}}c_r(s)\frac{\Lambda(2s,\chi_1\chi_2^{-1})}{\Lambda(2s+1,\chi_1\chi_2^{-1})}. \nonumber
\end{equation}
Here $c_r(s)$, the correction term coming from the ramified places, is given in \eqref{eq:def_of_c_ram} below.}
\end{theorem}

\begin{rem}
\begin{itemize}
\item This theorem generalizes the methods from \cite{Yo15} and \cite{HX16} to number fields. We also implement some ideas from \cite{Sa15_2} in order to deal with arbitrary level and central character.
\item The error term $\No(\n)^{\epsilon}\abs{T}_{\infty}^{\epsilon}\No(\n_0)$ has its origin in Lemma~\ref{lm:Bound_way_to_bad_in_n_0} and can probably removed with some computational effort. However, we did not attempt this here.
\item Suppose $F=\Q$, the central character is trivial, and the level $N$ is square free, then one quickly checks that the conductor of induced representations (which might contribute to the continuous spectrum) $\chi\boxplus \chi^{-1}$ is a perfect square. Thus, in the square free case there is (up to scaling) exactly one Eisenstein series $E_0(s,g)$ induced from a new vector. Applying our theorem to this Eisenstein series recovers the result from \cite{Yo15}. However, any other Eisenstein series transforming with respect to $K_1(N)$ can be produced by linear combinations of translates of $E_0$. Thus, we can also deal with the situation in \cite{HX16}. Since in this case we have $\No(\n_0)=1= \No(\mathfrak{l})$, there are no additional conditions. 
\end{itemize}
\end{rem}

The driving force behind this theorem is the improved amplifier. In contrast to the cusp form case, where the lower bound for the amplifier relies on combinatorial identities between Hecke eigenvalues, we use analytic tools to control the amplifier. Indeed, under some technical assumptions we prove
\begin{equation}
	\sum_{\substack{L\leq \No(\p)\leq 2L,\\ \p\in\mathcal{P}_{\mathfrak{q}}}} \log(\No(\p)) \eta_{\chi_1,\chi_2,it}(\p)\eta_{\chi_1^{-1},\chi_2^{-1},\red{-}ir}(\p) \asymp \frac{L}{\sharp\text{Cl}_F^{\mathfrak{q}}}, \label{eq:rough_statement_div_s}  
\end{equation}
where $\mathcal{P}_{\mathfrak{q}}$ is the set of principal prime \red{ideals with a generator congruent to} 1 modulo $\mathfrak{q}$ and
\begin{equation}
	\eta_{\chi_1,\chi_2,it}(\mathfrak{m}) = \sum_{\mathfrak{ab}=\mathfrak{m}} \chi_1(it)(\mathfrak{a})\chi_2(-it)(\mathfrak{b}) \nonumber
\end{equation}
is a generalized divisor sum. The precise statement can be found in Lemma~\ref{lm:low_amp_eis} below. The upshot of this result is that it enables us to choose a significantly shorter amplifier. However, it relies on extended zero\red{-}free regions for Hecke $L$-functions. To the best of our knowledge such zero\red{-}free regions do not yet exists in full uniformity. This is the reason for several technical assumptions on $L$, $\chi_1$, $\chi_2$, $r$ and $t$ in \eqref{eq:rough_statement_div_s} which ultimately \red{lead} to the caveat 
\begin{equation}
	\log(\No(\n)) \ll \log(\abs{T}_{\infty}) \text{ and } \log(\No(\mathfrak{l}))\ll \log(\abs{T}_{\infty})^{1-\delta} \text{ for some } \delta>0.\nonumber
\end{equation} 
In several special cases these assumptions may be relaxed, but we do not know how to get rid of them in general.

If one wants to avoid dealing with an explicit normalization one, can pose the sup-norm problem for Eisenstein series in a slightly different form. Indeed one can fix a \red{Jordan measurable} compact set $K\subset Z(\A_F)\setminus G(\A_F)$ \red{with non-empty interior} and study the quotients
\begin{equation}
	\frac{\sup_{g\in K}\abs{E(iT_0,g)}}{\left(\int_K \abs{E(iT_0,g)}^2 dg\right){\frac{1}{2}}}. \nonumber
\end{equation}
This is very similar in spirit to the way quantum unique ergodicity is studied for Eisenstein series. Indeed Quantum unique ergodicity for the Eisenstein case usually takes the form
\begin{equation}
	\lim_{T_0\to \infty}\frac{\int_{K_1} \abs{E(iT_0,g)}^2 dg}{\int_{K_2} \abs{E(iT_0,g)}^2 dg} = \frac{\vol(K_1)}{\vol(K_2)}. \label{eq:QUE_eisenstein}
\end{equation}
This is not yet known for ramified Eisenstein series over number fields. To the best of our knowledge \cite{Tr11} is the most general result to date. Let us state a nice corollary of our main theorem assuming quantum unique ergodicity in our setting.
 
\begin{cor}
If we assume \eqref{eq:QUE_eisenstein} for the Eisenstein series $E$\red{,} then we have
\begin{eqnarray}
	&&\frac{\sup_{g\in K}\abs{E(iT_0,g)}}{\left(\int_K \abs{E(iT_0,g)}^2 dg\right){\frac{1}{2}}} \ll \vol(K)^{-\frac{1}{2}}\sup_{g\in \red{G(F)KK_1(\n)}}\left(\abs{[v^{\circ}(iT_0)](g)} + \abs{c(iT_0)[\hat{v}^{\circ}(-iT_0)](g)}\right)  \nonumber \\
	&&+ \vol(K)^{-\frac{1}{2}}\left(\No(\n)^{\epsilon}\abs{T}_{\infty}^{\epsilon} \No(\n_0\mathfrak{m}_1)^{\frac{1}{2}}\No(\n_2)^{\frac{1}{4}}\left(\abs{T}_{\infty}^{\frac{3}{8}}+\abs{T}_{\R}^{\frac{3}{16}}\abs{T}_{\C}^{\frac{7}{16}}\right)+ \No(\n)^{\epsilon}\abs{T}_{\infty}^{\epsilon}\No(\n_0) \right), \nonumber 
\end{eqnarray}
for any \red{Jordan measurable} compact set $K\subset Z(\A_F)\setminus G(\A_F)$ \red{with non-empty interior}. \red{In particular, for $T_0\gg_K 1$ we have
\begin{equation}
	\frac{\sup_{g\in K}\abs{E(iT_0,g)}}{\left(\int_K \abs{E(iT_0,g)}^2 dg\right){\frac{1}{2}}} \ll_K \No(\n)^{\epsilon}\abs{T}_{\infty}^{\epsilon}\left( \No(\n_0\mathfrak{m}_1)^{\frac{1}{2}}\No(\n_2)^{\frac{1}{4}}\left(\abs{T}_{\infty}^{\frac{3}{8}}+\abs{T}_{\R}^{\frac{3}{16}}\abs{T}_{\C}^{\frac{7}{16}}\right)+ \No(\n_0)\right). \nonumber
\end{equation}}
\end{cor}

\section{The reduction step} \label{sec:reduction_step}

In \cite[Proposition~2.1]{As17_2} we established the following generating domain for $G(\A_F)$.

\begin{prop}  \label{pr:generating_domain_from_part_1}
For $g\in G(\A_F)$ we find $\mathfrak{L}\vert \n_2$ and $1\leq i\leq h_F$ such that
\begin{equation}
	g\in Z(\A)G(F) \left( a(\theta_i) \mathcal{J}_{\n} \times \mathcal{F}_{\n_2} \right)\eta_{\mathfrak{L}}K_1(\n). \nonumber
\end{equation}
\end{prop}

Furthermore, in \cite[Section~2.3]{As17_2}, we investigated the action of $\eta_{\mathfrak{L}}$ on cuspidal newforms. In the case of Eisenstein series we will use a similar argument. However, me must ensure that our explicit choice of new vector is preserved. \red{We will now proof the following reduction result, which settles the first part of the theorem.}

\begin{lemma} \label{lm:Eisenstein_reduction}
\red{Let $s\in i\R$ and $g\in G(\A_F)$. Then there is $\mathfrak{L}\mid\n_2$, $1\leq i\leq h_{\redd{F}}$, and $g'\in \mathcal{J}_{\n}\times \mathcal{F}_{\n_2}$ such that
\begin{equation}
	\abs{E(s,g)} =  \abs{E^{\mathfrak{L}}(s,a(\theta_i)g')}. \nonumber
\end{equation}}
\end{lemma}
This is the correct analogue of \cite[Corollary~2.2]{As17_2} taking our normalization of Eisenstein series into account. 
\begin{proof}
There is a canonical isomorphism 
\begin{equation}
	\bigg((\omega_{\pi}^{-1}\omega_{\pi}^{\mathfrak{L}})\pi(s),\mathbf{H}(s)\bigg) \to \bigg(\pi^{\mathfrak{L}}(s), \mathbf{H}_{\mathfrak{L}}(s)\bigg) \nonumber
\end{equation}
given by
\begin{equation}
	\mathbf{H}(s) \ni h \mapsto \tilde{h} = [\omega_{\pi}^{-1}\omega_{\pi}^{\mathfrak{L}}](\det(\cdot))h \in \mathbf{H}_{\mathfrak{L}}(s). \nonumber
\end{equation}
By \cite[Lemma~2.4]{As17_2} the vector $\pi(s)(\eta_{\mathfrak{L}})\tilde{v}^{\circ}$ is new. Thus, by multiplicity one, we have to check that \red{
\begin{equation}
	\abs{[\pi(s)(\eta_{\mathfrak{L}})v^{\circ}(s)](g)}=\abs{[\pi^{\mathfrak{L}}(s)(\eta_{\mathfrak{L}})\tilde{v}^{\circ}(s)](g)} = \abs{[v^{\circ}_{\mathfrak{L}}(s)](g)} \text{ for all $g\in G(\A_F)$}. \nonumber
\end{equation}
The first equality follows straight from the definition and unitarity of the characters involved. To show the second equality we have to work harder.} We will do so by checking this place by place. 

First, note that we only have to consider $\p\mid \mathfrak{L}$, since otherwise we only deal with unitary, unramified twists. Therefore, let $\p\mid \mathfrak{L}$. \red{Because we already know that $\pi_{\red{\p}}^{\red{\mathfrak{L}}}(s)(\eta_{\mathfrak{L},\red{\p}})\tilde{v}_{\red{\p}}^{\circ}(s)$ is a multiple of $v^{\circ}_{\mathfrak{L},\p}(s)$, we only have to compare their absolute value at some specific matrix.} We calculate
\begin{eqnarray}
	[\pi_{\red{\p}}^{\red{\mathfrak{L}}}(s)(\eta_{\mathfrak{L},\red{\p}})\tilde{v}_{\red{\p}}^{\circ}(s)]\left(\left(\begin{matrix} 1&0\\ \varpi_{\p}^{a(\omega_{\pi,\p}^{-1}\omega_{\pi,\p}^{\mathfrak{L}}\chi_{2,\p})}&1\end{matrix}\right)\right) &=& \red{\omega_{\pi,\p}^{\mathfrak{L}}(\varpi_{\p}^{n_{\p}})}\red{q_{\p}^{-\frac{a(\chi_{1,\p})}{2}}}f_s\left(\left(\begin{matrix} 0&1\\ \varpi_{\p}^{n_{\p}} & \varpi_{\p}^{a(\chi_{1,\p})} \end{matrix}\right)\right). \nonumber
\end{eqnarray}
One can write
\begin{equation}
	\left(\begin{matrix} 0&1\\ \varpi_{\p}^{n_{\p}} & \varpi_{\p}^{a(\chi_{1,\p})} \end{matrix}\right) = \left(\begin{matrix} -\varpi_{\p}^{n_{\p}-a(\chi_{1,\p})} & 1 \\ 0 & \varpi_{\p}^{a(\chi_{1,\p})} \end{matrix}\right)\left( \begin{matrix}1&0\\ \varpi_{\p}^{n_{\p}-a(\chi_{1,\p})}&1 \end{matrix}\right). \nonumber
\end{equation}
With this decomposition \red{and \eqref{eq:def_fs}} we evaluate
\red{\begin{equation}
	f_s\left(\left(\begin{matrix} 0&1\\ \varpi_{\p}^{n_{\p}} & \varpi_{\p}^{a(\chi_{1,\p})} \end{matrix}\right)\right) = \chi_{1,\p}(s)(-1)\chi_{2,\p}(-s)(\varpi_{\p}^{a(\chi_{1,\p})})q_{\p}^{\frac{1}{2}(a(\chi_{1,\p})-a(\chi_{2,\p}))}. \nonumber
\end{equation}}
\red{\redd{Altogether} this yields
\begin{multline}
	[\pi_{{\p}}^{{\mathfrak{L}}}(s)(\eta_{\mathfrak{L},{\p}})\tilde{v}_{{\p}}^{\circ}(s)]\left(\left(\begin{matrix} 1&0\\ \varpi_{\p}^{a(\omega_{\pi,\p}^{-1}\omega_{\pi,\p}^{\mathfrak{L}}\chi_{2,\p})}&1\end{matrix}\right)\right) \\ =  {\omega_{\pi,\p}^{\mathfrak{L}}(\varpi_{\p}^{n_{\p}})}\chi_{1,\p}(-1)\chi_{2,\p}(-s)(\varpi_{\p}^{a(\chi_{1,\p})})q_{\p}^{-\frac{a(\chi_{2,\p})}{2}}. \nonumber
\end{multline}
}

\red{The claim $\abs{\pi^{\mathfrak{L}}(s)(\eta_{\mathfrak{L}})\tilde{v}^{\circ}(s)} = \abs{v^{\circ}_{\mathfrak{L}}(s)}$ now follows from
\begin{equation}
	[v^{\circ}_{\mathfrak{L},\p}(s)]\left(\left(\begin{matrix} 1&0\\ \varpi_{\p}^{a(\omega_{\pi,\p}^{-1}\omega_{\pi,\p}^{\mathfrak{L}}\chi_{2,\p})}&1\end{matrix}\right)\right) = [\omega_{\pi,\p}^{-1}\omega_{\pi,\p}^{\mathfrak{L}}\chi_{1,\p}(s)](\varpi_{\p}^{-a(\chi_{1,\p})})q_{\p}^{-\frac{a(\chi_{2,\p})}{2}}, \nonumber
\end{equation}
 together with $s\in i\R$ and the fact that $\omega_{\pi,\p}^{-1}\omega_{\pi,\p}^{\mathfrak{L}}\chi_{1,\p}(s)$ is a unitary unramified twist of $\chi_{2,\p}^{-1}(s)$.

Finally, we observe that due to $F^{\times}$-invariance of $\omega_{\pi}^{-1}\omega_{\pi}^{\mathfrak{L}}$ and \eqref{eq:averaging_def} we have
\begin{equation}
	E_{\tilde{v}}(s,g) = [\omega_{\pi}^{-1}\omega_{\pi}^{\mathfrak{L}}](\det(g))E_{v}(s,g). \nonumber
\end{equation}
Thus, our computations imply that $\abs{E(s,g\eta_{\mathfrak{L}})} = \abs{E^{\mathfrak{L}}(s,g)}$. With this at hand the claimed result follows directly from the generating domain given in Proposition~\ref{pr:generating_domain_from_part_1}.}
\end{proof}

\section{Bounds via Whittaker expansions} \label{sec:nounds_via_Whittaker}

The Whittaker expansion of $E(s,g)$ is given by
\begin{equation}
	E(s,g) = v^{\circ}(s)(g) + [M(s)v^{\circ}(s)](g)+ \sum_{q\in F^{\times}} W_s(a(q)g), \nonumber
\end{equation}
where
\begin{eqnarray}
	[M(s)v^{\circ}(s)](g) &=& \int_{N(\A_F)} v^{\circ}(s)(\omega ng)dn \text{ and }\nonumber \\
	 W_s(g) &=& \frac{2^{r_2}}{\sqrt{\abs{d_F}}} \int_{\A_F} v^{\circ}(s)(\omega n(x) g) \psi(-x) d\mu_{\A_F}(x). \nonumber
\end{eqnarray}
If the integral representation\red{s} for $M(s)$ \red{and $W_s$ do} not converge\red{,} we understand \red{them} by \red{their} analytic continuation. 

We will start by making this expansion as explicit as possible. After doing so we put it to use and derive several useful bounds.

\subsection{The constant term of $E(s,g)$} \label{app:const_term}

The goal of this section is to evaluate $[M(s)v^{\circ}(s)](g)$. First, we observe that $[M(s)v^{\circ}(s)]\in B(\chi_2(-s),\chi_1(s))$. Second, $M(s)v^{\circ}(s)$ is $K_1(\n)$ invariant. Therefore, we have
\begin{equation}
	[M(s)v^{\circ}(s)](g) = c(s) \hat{v}^{\circ}(-s)(g). \label{eq:everything_about_cs}
\end{equation}
We will now continue to calculate the constant $c(s)$. To do so we exploit that $M(s)$ factorizes into $p$-adic integrals. These integrals can be evaluated locally.

While the operator $M(s)$ is defined globally by the integral representation only for $\Re(s)>\frac{1}{2}$, the local integrals converge as long as $\Re(s)>0$. Thus, throughout the rest of this subsection we assume $\Re(s)>0$.

Let us start with archimedean $\nu$. Here we use the decomposition
\begin{equation}
	\left( \begin{matrix} 0 & -1 \\ 1 & x \end{matrix} \right)  = \left( \begin{matrix} \frac{1}{\sqrt{\abs{x}^2+1}} & \frac{-\overline{x}}{\sqrt{\abs{x}^2+1}} \\ 0 &\sqrt{\abs{x}^2+1}  \end{matrix} \right)\left( \begin{matrix} \frac{\overline{x}}{\sqrt{\abs{x}^2+1}}& \frac{-1}{\sqrt{\abs{x}^2+1}} \\ \frac{1}{\sqrt{\abs{x}^2+1}} & \frac{x}{\sqrt{\abs{x}^2+1}} \end{matrix} \right). \label{eq:special_iwasa_acrch}
\end{equation} 
This holds for real as well as complex $x$. By  \eqref{eq:arch_new_vec} we have $\hat{v}^{\circ}_{\nu}(-s)(1)=1$. Therefore, the local contribution to $c(s)$ coming from $\nu$ is simply given by $M(s)v_{\nu}^{\circ}(s)(1)$.
For real $\nu$ we compute
\begin{eqnarray}
	c_{\nu}(s) &=& \int_{\R} v^{\circ}_{\nu}(s)(-\omega n(x))d\mu_{\nu}(x) = \int_{\R} \frac{dx}{(x^2+1)^{s+it_{\nu}+\frac{1}{2}}} = \sqrt{\pi}\frac{\Gamma(s+it_{\nu})}{\Gamma(s+it_{\nu}+\frac{1}{2})} \nonumber \\
	&=& \frac{\Gamma_{\R}(2(s+it_{\nu}))}{\Gamma_{\R}(2(s+it_{\nu})+1)} = \frac{L_{\nu}(2s,\chi_{1,\nu}\chi_{2,\nu}^{-1})}{L_{\nu}(2s+1,\chi_{1,\nu}\chi_{2,\nu}^{-1})}. \nonumber
\end{eqnarray}
If $\nu$ is complex, we argue similarly. One checks
\begin{eqnarray}
	c_{\nu}(s) &=& \int_{\C} v^{\circ}_{\nu}(s)(-\omega n(z))d\mu_{\nu}(z) = \int_{\C} \frac{dz}{(\abs{z}^2+1)^{2s+2it_{\nu}+1}} = 2\pi  \int_0^{\infty} \frac{r}{(r^2+1)^{2s+2it_{\nu}+1}}dr \nonumber \\
	&=& \frac{2\pi}{4s+4it_{\nu}} = \frac{\Gamma_{\C}(2s+2it_{\nu})}{2\Gamma_{\C}(2s+2it_{\nu}+1)} = \frac{L_{\nu}(2s,\chi_{1,\nu}\chi_{2,\nu}^{-1})}{2L_{\nu}(2s+1,\chi_{1,\nu}\chi_{2,\nu}^{-1})} . \nonumber
\end{eqnarray}

Now we turn to the non-archimedean places. Note that since we assumed $\Re(s)>0$ we have $\abs{\chi_{1,\p}(s)(\varpi_{\p})} < \abs{\chi_{2,\p}(-s)(\varpi_{\p})}$. By \cite[Proposition~4.5.6]{Bu96} the integral representation for $M_{\p}$ is defined. 

First, we consider $\p \nmid \n$. In this situation we have $\chi_{1,\p}(s)=\abs{\cdot}_{\p}^{s+it_{\p,1}}$ and $\chi_{2,\p}(s)=\abs{\cdot}_{\p}^{-s+it_{\p,2}}$ for some $t_{\p,i}\in \R$. Then \cite[Proposition~4.6.7]{Bu96} yields
\begin{equation}
	M_{\p}v_{\p}^{\circ}(s) = \underbrace{\frac{L_{\p}(2s,\chi_{1,\p}\chi_{2,\p}^{-1})}{L_{\p}(2s+1,\chi_{1,\p}\chi_{2,\p}^{-1})}}_{=c_{\p}(s)}\hat{v}_{\p}^{\circ}(-s). \nonumber
\end{equation}

At last, we deal with the places $\p\mid \n$. Observe that 
\begin{equation}
	\red{\hat{v}_{\p}^{\circ}(-s) \left(\left( \begin{matrix} 1 & 0 \\ \varpi_{\p}^{a(\chi_{1,\p})} & 1 \end{matrix}\right)\right) = q_{\p}^{-\frac{a(\chi_{2,\p})}{2}}\chi_{2,\p}(-s)(\varpi_{\p}^{-a(\chi_{1,\p})})}.\nonumber
\end{equation}
Therefore, we need to evaluate
\begin{equation}
	c_{\p}(s) = M(s)v^{\circ}_{\p}(s)\left(\left( \begin{matrix} 1 & 0 \\ \varpi_{\p}^{a(\chi_{1,\p})} & 1 \end{matrix}\right)\right)\red{q_{\p}^{\frac{a(\chi_{2,\p})}{2}}\chi_{2,\p}(-s)(\varpi_{\p}^{a(\chi_{1,\p})})}. \nonumber
\end{equation}
Using the identity
\begin{equation}
	-\omega n(x) = \left(\begin{matrix}x^{-1}&-1 \\ 0 &x \end{matrix}\right)\left(\begin{matrix} 1&0\\ x^{-1} & 1\end{matrix}\right),\label{eq:non_arch_iwa}
\end{equation}
for $x\neq 0$, we can calculate
\begin{align}
	c_{\p}(s) &= \red{q_{\p}^{\frac{a(\chi_{2,\p})}{2}}\chi_{2,\p}(-s)(\varpi_{\p}^{a(\chi_{1,\p})})}\int_{F_{\p}}  v_{\p}^{\circ}(s)\left( -\left( \begin{matrix} 0&-1 \\ 1&x \end{matrix}\right) \left(\begin{matrix} 1 & 0 \\ \varpi_{\p}^{a(\chi_{1,\p})} & 1 \end{matrix}\right)\right) d\mu_{\p}(x)   \nonumber \\
	 &=  \red{q_{\p}^{\frac{a(\chi_{2,\p})}{2}}\chi_{2,\p}(-s)(\varpi_{\p}^{a(\chi_{1,\p})})}\int_{F_{\p}^{\times}}  v_{\p}^{\circ}(s)\left( -\left( \begin{matrix}x^{-1}&-1 \\ 0&x \end{matrix}\right) \left(\begin{matrix} 1 & 0 \\ x^{-1}+\varpi_{\p}^{a(\chi_{1,\p})} & 1 \end{matrix}\right)\right) d\mu_{\p}(x)  \nonumber \\
	&=\omega_{\pi(\red{0}),\p}(-1)\red{q_{\p}^{\frac{a(\chi_{2,\p})-a(\chi{1,\p})}{2}}\chi_{2,\p}(-s)(\varpi_{\p}^{a(\chi_{1,\p})})} \nonumber \\
	&\qquad\qquad\qquad\qquad\qquad \cdot \int_{F_{\p}^{\times}}\frac{\chi_{2,\p}(x)}{\chi_{1,\p}(x) }\abs{x}_{\p}^{-2s} f_s \left(\left(  \begin{matrix} 1&0 \\ x^{-1}+\varpi_{\p}^{a(\chi_{1,\p})}& 1 \end{matrix} \right)\right)\frac{d\mu_{\p}(x)}{\abs{x}_{\p}}.\label{eq:starting_poinmt_for_zeta} 
\end{align}
Fairly standard manipulations yield
\begin{eqnarray}
	c_{\p}(s) &=& \zeta_{\p}(1)^{-1}\red{q_{\p}^{\frac{a(\chi_{2,\p})-a(\chi{1,\p})}{2}}}\omega_{\pi(\red{0}),\p}(-1) \sum_{m\in \Z} q_{\p}^{-2sm\red{+}a(\chi_{\red{1},\p})s}\chi_{1,\p}(\varpi_{\p}^{m})\chi_{2,\p}(\varpi_{\p}^{\red{a(\chi_{1,\p})}-m}) \nonumber\\
	&&\qquad\qquad \cdot \int_{\op_{\p}^{\times}} \frac{\chi_{1,\p}(x)}{\chi_{2,\p}(x)} f_s\left( \left( \begin{matrix} 1&0 \\ \varpi_{\p}^{m}x+\varpi_{\p}^{a(\chi_{1,\p})} & 1 \end{matrix} \right) \right) d\mu_{\p}^{\times}(x). \label{eq:sum_we_arrive_at}
\end{eqnarray}
We will use explicit values for $f_{\red{s}}$ given in \cite{Sc02} to evaluate the remaining integrals. The argument splits into several cases.

First, we consider $\chi_{1,\p}$ to be unramified, in particular $a(\chi_{2,\p})>0$. One observes
\begin{equation}
	v_{\p}(\varpi_{\p}^{m}x+1) = \begin{cases} 0 &\text{ if } m>0, \\ v_{\p}(x+1) &\text{ if } m=0, \\ m &\text{ if } m<0. \end{cases} \nonumber
\end{equation}
Inserting \cite[(22)]{Sc02} in \eqref{eq:sum_we_arrive_at} yields
\begin{eqnarray}
	c_{\p}(s) &=& \zeta_{\p}(1)^{-1}\red{q_{\p}^{\frac{a(\chi_{2,\p})}{2}}\chi_1(s)(\varpi_{\p}^{-a(\chi_{2,\p})})}\int_{-1+\varpi_{\p}^{a(\chi_{2,\p})}\op_{\p}}\chi_{2,\p}(-x)^{-1}d\mu_{\p}^{\times}(x) \nonumber \\ 
	&=& \red{\chi_{1,\p}(\varpi_{\p}^{-a(\chi_{2,\p})})} q_{\p}^{\red{(s-\frac{1}{2})n_{\p}}}. \nonumber
\end{eqnarray}

Second, we consider $\chi_{2,\p}$ to be unramified. In this case we have 
\begin{equation}
	v_{\p}(\varpi_{\p}^{m}x+\varpi_{\p}^{a(\chi_{1,\p})}) = \begin{cases} m &\text{ if } m<a(\chi_{1,\p}), \\ a(\chi_{1,\p}) + v_{\p}(x+1) &\text{ if } m=a(\chi_{1,\p}), \\ a(\chi_{1,\p}) &\text{ if } m>a(\chi_{1,\p}). \end{cases} \label{eq:valuation_of_dirt}
\end{equation}
Thus, after applying \cite[(23)]{Sc02} we obtain
\begin{align}
	c_{\p}(s) &= \zeta_{\p}(1)^{-1} \red{q_{\p}^{-\frac{a(\chi_{1,\p})}{2}}}\sum_{m\leq 0} q_{\p}^{\red{a(\chi_{1,\p})s}-2sm}\chi_{1,\p}(\varpi_{\p}^{m})\chi_{2,\p}(\varpi_{\p}^{\red{a(\chi_{1,\p})}-m}) \int_{\op_{\p}^{\times}} \chi_{1,\p}(-x) \nonumber \\
	&\quad \cdot  \chi_{1,\p}(s)(\varpi_{\p}^{m}x+\varpi_{\p}^{a(\chi_{1,\p})})^{-1} \chi_{2,\p}(\red{-}s)(\varpi_{\p}^{m}x+\varpi_{\p}^{a(\chi_{1,\p})})\abs{\varpi_{\p}^{m}x+\varpi_{\p}^{a(\chi_{1,\p})}}^{-1}_{\p} d\mu_{\p}^{\times}(x) \nonumber \\
	&= \zeta_{\p}(1)^{-1} \red{\chi_{2,\p}(\varpi_{\p}^{a(\chi_{1,\p})})q_{\p}^{-\frac{a(\chi_{1,\p})}{2}}}\sum_{m\leq 0} q_{\p}^{m+\red{a(\chi_{1,\p})s}} \int_{\op_{\p}^{\times}} \chi_{1,\p}(-x) \chi_{1,\p}(x+\varpi_{\p}^{a(\chi_{1,\p})-m})^{-1} d\mu_{\p}^{\times}(x) \nonumber \\
	&= \chi_{1,\p}(-1)\red{\chi_{2,\p}(\varpi_{\p}^{a(\chi_{1,\p})})q_{\p}^{(s-\frac{1}{2})n_{\p}}.} \nonumber
\end{align}

Summarizing the first two cases gives
\begin{equation}
	c_{\p}(s) = \chi_{1,\p}(-\red{\varpi_{\p}^{-a(\chi_{2,\p})}})\red{\chi_{2,\p}(\varpi_{\p}^{a(\chi_{1,\p})})q_{\p}^{(s-\frac{1}{2})n_{\p}}} \label{eq:summary_degen_cases}
\end{equation}
if either $a(\chi_{1,\p})=0$ or $a(\chi_{2,\p})=0$.

Finally, we need to consider the situation where $\chi_{1,\p}$ and $\chi_{2,\p}$ are both ramified. Define 
\begin{eqnarray}
	\chi &=& \chi_{1,\p}\chi_{2,\p}^{-1},\text{ and} \nonumber \\
	h(x) &=& f_s\left( \left( \begin{matrix} 1&0 \\ x+\varpi_{\p}^{a(\chi_{1,\p})} & 1 \end{matrix} \right) \right). \nonumber
\end{eqnarray}
We will exploit the local functional equation. Recall
\begin{eqnarray}
	Z(s,h,\chi) &=& \zeta_{\p}(1)^{-1}\int_{F_{\p}^{\times}} \abs{x}^{s}\chi(x) h(x) d\mu_{\p}^{\times}(x), \nonumber\\
	\hat{h}(y) &=& \int_{F_{\p}} h(x) \psi_{\p}'(xy) d\mu_{\p}(x), \nonumber
\end{eqnarray}
where $\psi_{\p}' = \psi_{\p}(\cdot \varpi_{\p}^{-d_{\p}})$ is an additive character with $n(\psi_{\p}')=0$. Thus, the measure $\mu_{\p}$  is the self-dual measure for the Fourier transform defined above. Therefore, we have the local functional equation
\begin{equation}
	\frac{Z(1-s, \hat{h},\chi^{-1})}{Z(s,h,\chi)} = \epsilon_{\p}(s,\chi,\psi_{\red{\p}}') \frac{L_{\p}(1-s,\chi^{-1})}{L_{\p}(s,\chi)}. \nonumber
\end{equation}
Having a close look at \eqref{eq:starting_poinmt_for_zeta} we observe that
\begin{align}
	&c_{\p}(s) =\omega_{\pi(s),\p}(-1)\red{ q_{\p}^{\frac{a(\chi_{2,\p})-a(\chi_{1,\p})}{2}}\chi_{2,\p}(-s)(\varpi_{\p}^{a(\chi_{1,\p})})} Z(2s,h,\chi)  \nonumber \\
	&= \omega_{\pi(s),\p}(-1) \red{ q_{\p}^{\frac{a(\chi_{2,\p})-a(\chi_{1,\p})}{2}}\chi_{2,\p}(-s)(\varpi_{\p}^{a(\chi_{1,\p})})} \epsilon_{\p}(2s,\chi,\psi_{\p}')^{-1} \frac{L_{\red{\p}}(2s,\chi)}{L_{\red{\p}}(1-2s,\chi^{-1})} Z(1-2s,\hat{h},\chi^{-1}). \label{eq:cs_via_zeta_int}
\end{align}

We start by computing the Fourier transform of $h$. By \cite[(21)]{Sc02} we have
\begin{eqnarray}
	\hat{h}(y) &=& \int_{F_{\p}}f_{\red{s}}\left( \left( \begin{matrix} 1&0\\ x+\varpi_{\p}^{a(\chi_{1,\p})} & 1 \end{matrix} \right) \right) \psi_{\p}'(xy) d\mu_{\p}(x) \nonumber \\
	&=& \psi_{\p}'(-\varpi_{\p}^{a(\chi_{1,\p})}y) \int_{F_{\p}}f_{\red{s}}\left( \left( \begin{matrix} 1& 0 \\ x& 1\end{matrix} \right) \right) \psi'_{\p}(xy) d\mu_{\p}(x) \nonumber \\
	&=&  \psi_{\p}'(-\varpi_{\p}^{a(\chi_{1,\p})}y) \int_{\varpi_{\p}^{a(\chi_{2,\p})}\op_{\p}^{\times}}[\chi_{1,\p}(s)]^{-1}(x) \psi'_{\p}(xy) d\mu_{\p}(x) \nonumber \\
	&=& \zeta_{\p}(1)^{-1} \psi_{\p}'(-\varpi_{\p}^{a(\chi_{1,\p})}y)q_{\p}^{(s-1)a(\chi_{2,\p})}\chi_{1,\p}(\varpi_{\p}^{-a(\chi_{2,\p})})G(\varpi_{\p}^{a(\chi_{2,\p})}y,\chi_{1,\p}^{-1}).  \nonumber 
\end{eqnarray}
Here we used the definition of the Gau\ss\  sum given in \cite{Sa15_2}. We evaluate this explicitly in terms of $\epsilon$-factors using \cite[(6)]{Sa15_2}. This yields
\begin{equation}
	\hat{h}(y)  = \begin{cases}
		\psi_{\p}'(-\varpi_{\p}^{a(\chi_{1,\p})}y) \chi_{1,\p}(y) \epsilon_{\p}(\frac{1}{2},\chi_{1,\p},\psi_{\p}') q_{\p}^{(s-1)a(\chi_{2,\p})-\frac{1}{2}a(\chi_{1,\p})} 	&\text{ if } v_{\p}(y) = -n_{\p}, \\
		0 &\text{ else.}
	\end{cases} \nonumber
\end{equation}
The $Z$-integral boils down to another Gau\ss\  sum. We have
\begin{eqnarray}
	Z(1-2s,\hat{h},\chi^{-1}) &=&\zeta_{\p}(1)^{-1} q_{\p}^{(s-1)a(\chi_{2,\p})-\frac{1}{2}a(\chi_{1,\p})}\epsilon_{\p}(\frac{1}{2},\chi_{1,\p},\psi'_{\p}) \nonumber\\
	&&\qquad\quad\cdot \int_{\varpi_{\p}^{-a(\chi_{2,\p})-a(\chi_{1,\p})}\op_{\p}^{\times}}\abs{x}_{\p}^{1-2s} \chi_{2,\p}(x) \psi'_{\p}(-\varpi_{\p}^{a(\chi_{1,\p})}y)d\mu_{\p}^{\times}(x) \nonumber \\
	&=& \zeta_{\p}(1)^{-1} q_{\p}^{-sa(\chi_{2,\p})+(\frac{1}{2}-2s)a(\chi_{1,\p})}\chi_{2,\p}(\varpi_{\p}^{-a(\chi_{2,\p})-a(\chi_{1,\p})}) \nonumber \\
	&&\qquad\quad \cdot\epsilon_{\p}(\frac{1}{2},\chi_{1,\p},\psi'_{\p}) G(-\varpi_{\p}^{-a(\chi_{2,\p})},\chi_{2,\p}) \nonumber \\
	&=&  q_{\p}^{(-\frac{1}{2}-s)a(\chi_{2,\p})+(\frac{1}{2}-2s)a(\chi_{1,\p})} \chi_{2,\p}(-\varpi_{\p}^{-a(\chi_{1,\p})})\epsilon_{\p}(\frac{1}{2},\chi_{1,\p},\psi'_{\p})\epsilon_{\p}(\frac{1}{2},\chi_{2,\p}^{-1},\psi'_{\p}). \nonumber
\end{eqnarray}
Thus, by \eqref{eq:cs_via_zeta_int} we have the formula
\begin{eqnarray}
	c_{\p}(s) &=& \chi_{1,\p}(-1)q_{\p}^{\red{-sn_{\p}}} \frac{ \epsilon_{\p}(\frac{1}{2},\chi_{1,\p},\psi'_{\p})\epsilon_{\p}(\frac{1}{2},\chi_{2,\p}^{-1},\psi'_{\p})}{\epsilon_{\p}(2s,\chi_{1,\p}\chi_{2,\p}^{-1},\psi'_{\p})}\frac{L_{\p}(2s,\chi_{1,\p}\chi_{2,\p}^{-1})}{L_{\p}(1-2s,\chi_{1,\p}^{-1}\chi_{2,\p})}. \nonumber
\end{eqnarray}
\red{Note that this expression is valid for all $\p\mid \n$. Indeed, if exactly one of the characters $\chi_{1,\p}$ is ramified, then the $L$- and $\epsilon$-factor contribution can be evaluated explicitly. One observes that the result fits precisely to \eqref{eq:summary_degen_cases}.}

To gather the global expression of $c(s)$ we first recall that $\No(\mathfrak{d}) = \abs{d_F}$. We define 
\begin{align}
 	\mathfrak{l} &= \prod_{\p}\p^{a(\chi_{1,\p}\chi_{2,\p}^{-1})}, \nonumber \\
	c_{r,\p}\red{(s)} &= \chi_{1,\p}(-1)q_{\p}^{\red{-sn_{p}}-(\frac{1}{2}-2s)a(\chi_{1,\p}\chi_{2,\p}^{-1})}\frac{ \epsilon_{\p}(\frac{1}{2},\chi_{1,\p},\psi'_{\p})\epsilon_{\p}(\frac{1}{2},\chi_{2,\p}^{-1},\psi'_{\p})}{\epsilon_{\p}(\frac{1}{2},\chi_{1,\p}\chi_{2,\p}^{-1},\psi'_{\p})} \frac{L_{\p}(2s+1,\chi_{1,\p}\chi_{2,\p}^{-1})}{L_{\p}(1-2s,\chi_{1,\p}^{-1}\chi_{2,\p})},  \nonumber \\ 
	c_r\red{(s)} &= \prod_{\p\mid \n} c_{r,\p}\red{(s)} \red{= \mathcal{N}(\n)^{-s}\mathcal{N}(\mathfrak{l})^{2s-\frac{1}{2}}\frac{ \epsilon(\frac{1}{2},\chi_{1},\psi)\epsilon(\frac{1}{2},\chi_{2}^{-1},\psi)}{\epsilon(\frac{1}{2},\chi_{1}\chi_{2}^{-1},\psi)} \prod_{\p\mid\n} \chi_{1,\p}(-1) \frac{L_{\p}(2s+1,\chi_{1,\p}\chi_{2,\p}^{-1})}{L_{\p}(1-2s,\chi_{1,\p}^{-1}\chi_{2,\p})}    }. \label{eq:def_of_c_ram}
\end{align}
The constant term of $E(s,g)$ is given  by
\begin{equation}
	v^{\circ}(s)(g) + \No(\mathfrak{d})^{-\frac{1}{2}}c_r(s)\frac{\Lambda(2s,\chi_1\chi_2^{-1})}{\Lambda(2s+1,\chi_1\chi_2^{-1})}\hat{v}^{\circ}(-s)(g). \nonumber
\end{equation}

\subsection{The Whittaker coefficients of $E(s,g)$} \label{app:whitt_f}

In this subsection we compute
\begin{equation}
	W_s(g)=\frac{2^{r_2}}{\sqrt{\abs{d_F}}}\int_{\A_F}v^{\circ}(s)(\omega n(x)g)\psi(-x)d\mu_{\A_f}(x). \nonumber
\end{equation}
These functions will give rise to the coefficients in the Whittaker expansion of $E(s,\cdot)$.

We start with an archimedean place $\nu$. In this case we can clearly restrict our attention to $g=a(y)$ for some $y>0$. One checks
\begin{eqnarray}
	W_{s,\nu}(a(y)) &=& \int_{F_{\nu}} v_{\nu}^{\circ}(s)(-\omega n(x) a(y))\psi_{\nu}(\red{-x})d\mu_{\nu}(x) \nonumber \\
	&=&  \abs{y}_{\nu}^{\frac{1}{2}+it_{\red{\nu,2}}-s} \int_{F_{\nu}} v_{\nu}^{\circ}(s)(-\omega n(x)) \psi_{\nu}(-xy) d\mu_{\nu}(x).\nonumber
\end{eqnarray}
Applying \eqref{eq:arch_new_vec} and \eqref{eq:special_iwasa_acrch} gives
\begin{equation}
	W_{s,\nu}(a(y)) = \abs{y}_{\nu}^{\frac{1}{2}+it_{\red{\nu,2}}-s} \int_{F_{\nu}} \frac{\psi_{\nu}(xy)}{\abs{\abs{x}^2+1}_{\nu}^{s+it_{\nu}+\frac{1}{2}}}d\mu_{\nu}(x). \nonumber
\end{equation}

For $\nu$ \red{real we use \cite[3.771~(2)]{GR07} to compute}
\begin{eqnarray}
	W_{s,\nu}(a(y)) &=& \abs{y}_{\nu}^{\frac{1}{2}+it_{\red{\nu,2}}-s}  \int_{\R} \frac{e^{\red{2\pi i xy}}}{(\abs{x}^2+1)^{s+it_{\nu}+\frac{1}{2}}}d\mu_{\nu}(x) \nonumber \\
	&=&  2\abs{y}_{\nu}^{\frac{1}{2}+it_{\red{\nu,2}}-s}  \int_{0}^{\infty} \frac{\cos(2\pi  xy)}{(\abs{x}^2+1)^{s+it_{\nu}+\frac{1}{2}}}d\mu_{\nu}(x) \nonumber \\ 
	&=&  \red{2}\abs{y}_{\nu}^{\frac{1}{2}+\frac{1}{2}is_{\nu}}\Gamma_{\R}(2(s+it_{\nu})+1)^{-1} K_{s+it_{\nu}}(2\pi y).\nonumber 
\end{eqnarray}

For $\nu$ complex we have
\begin{eqnarray}
	W_{s,\nu}(a(y)) &=& \abs{y}_{\nu}^{\frac{1}{2}+it_{\red{\nu,2}}-s} \int_{\C} \frac{e^{4\pi i y \Re(x)}}{(\abs{x}^2+1)^{2s+2it_{\nu}+1}}d\mu_{\nu}(x) \nonumber \\
	&=& \abs{y}_{\nu}^{\frac{1}{2}+it_{\red{\nu,2}}-s}\int_0^{\infty} \frac{r}{(r^2+1)^{2s+2it_{\nu}+1}} \int_{\red{0}}^{\red{2\pi}} e^{\red{4}\pi ryi\red{\cos(\theta)}}d\theta dr \nonumber \\
	&=& \red{2}\abs{y}_{\nu}^{\frac{1}{2}+it_{\red{\nu,2}}-s}\int_0^{\infty} \frac{r}{(r^2+1)^{2s+2it_{\nu}+1}}\int_{\red{0}}^{\red{\pi}}\red{\cos(4\pi y r\cos(\theta))}d\theta dr \nonumber \\
	&=& 2\pi \abs{y}_{\nu}^{\frac{1}{2}+it_{\red{\nu,2}}-s}\int_0^{\infty} \frac{rJ_0(4\pi y r)}{(r^2+1)^{2s+2it_{\nu}+1}} dr \nonumber \\
	&=&  \abs{y}_{\nu}^{\frac{1}{2}+\frac{1}{2}is_{\nu}}\frac{(2\pi)^{2s+2it_{\nu}+1}}{\Gamma(2s+2it_{\nu}+1)}K_{2s+2it_{\nu}}(4\pi y)\nonumber \\
	&=& 2 \abs{y}_{\nu}^{\frac{1}{2}+\frac{1}{2}is_{\nu}} \Gamma_{\C}(2s+2it_{\nu}+1)^{-1} K_{2s+2it_{\nu}}(4\pi y). \nonumber
\end{eqnarray}
Here we used \red{\cite[3.715~(18)]{GR07} to compute the $\theta$-integral, and} \cite[6.565 (4)]{GR07} to compute the $r$-integral. 

Next we turn to the non-archimedean places. We have to be careful because we are not working with an unramified additive character. Indeed, $n(\psi_{\p})=-v_{\p}(\mathfrak{d})=-d_{\p}$. Therefore, we use $\psi'_{\p}$ as defined above instead. We check
\begin{eqnarray}
	W_{s,\p}(g) &=& \int_{F_{\p}}v_{\p}^{\circ}(s)(\omega n(x)g) \psi_{\p}'(-\varpi_{\p}^{d_{\p}}x)d\mu_{F_{\p}}(x) \nonumber \\
	&=& \abs{\varpi_{\p}^{-d_{\p}}}_{\p}^{\frac{1}{2}-s}\chi_{2,\p}( \varpi_{\p}^{-d_{\p}}) \int_{F_{\p}}v_{\p}^{\circ}(s)(\omega n(x) a(\varpi_{\p}^{d_{\p}})g)\psi_{\p}'(-x)d\mu_{F_{\p}}(x).\nonumber
\end{eqnarray}
The last integral defines a Whittaker new vector with respect to the unramified character $\psi_{\p}'$. We set
\begin{equation}
	\tilde{W}_{s,\p}(a(\varpi_{\p}^{d_{\p}})g)= \int_{F_{\p}}v_{\p}^{\circ}(s)(\omega n(x) a(\varpi_{\p}^{d_{\p}})g)\psi_{\p}'(-x)d\mu_{F_{\p}}(x).\nonumber
\end{equation}

We start by considering $\p\nmid \n$. In this case the evaluation of $\tilde{W}_{s,\p}$ is quite standard. By \red{\cite[(6.11) and Theorem~4.6.5]{Bu96}} we have
\begin{equation}
	\tilde{W}_{s,\p}(1) = L_{\p}(2s+1,\chi_{1,\p}\chi_{2,\p}^{-1})^{-1} \nonumber
\end{equation}
\red{and}
\begin{equation}
	\tilde{W}_{s,\p}(a(\varpi_{\p}^m))=\begin{cases} \tilde{W}_{s,\p}(1) q_{\p}^{-\frac{m}{2}}\frac{\alpha_1^{m+1}-\alpha_2^{m+1}}{\alpha_1-\alpha_2} &\text{ if $m\geq 0$,} \\ 0 &\text{ else} \end{cases} \nonumber
\end{equation}
with $\alpha_1=\chi_{1,\p}(s)(\varpi_{\p})$ and $\alpha_2=\chi_{2,\p}(-s)(\varpi_{\p})$. This can be rewritten in terms of arithmetic functions. Indeed, for a fractional ideal $\mathfrak{a}$ we write 
\begin{equation}
	\chi_i(\mathfrak{a}) = \prod_{\p} \chi_{i,\p}(\varpi_{\p}^{v_{\p}(\mathfrak{a})}). \nonumber
\end{equation} 
and define the generalized divisor function
\begin{equation}
	\eta_{\chi_1,\chi_2,s}(\mathfrak{m}) = \sum_{\mathfrak{a}\mathfrak{b} = \mathfrak{m}} \chi_{1,\p}(s)(\mathfrak{a})\chi_{2,\p}(-s)(\mathfrak{b}).\label{eq:def_of_gen_div_sum}
\end{equation}
One checks that
\begin{equation}
	\frac{\alpha_1^{m+1}-\alpha_2^{m+1}}{\alpha_1-\alpha_2} = \eta_{\chi_1,\chi_2,s}(\p^m). \nonumber
\end{equation}
\red{Altogether} we have
\begin{equation}
	W_{s,\p}(a(\varpi_{\p}^m)) = \delta_{\red{m}\geq -d_{\p}} \frac{\abs{\varpi_{\p}^{-d_{\p}}}_{\p}^{\frac{1}{2}-s}\chi_{2,\p}( \varpi_{\p}^{-d_{\p}})}{L_{\p}(2s+1,\chi_{1,\p}\chi_{2,\p}^{-1})} \abs{\varpi_{\p}^{m\red{+d_{\p}}}}_{\p}^{\frac{1}{2}} \eta_{\chi_1,\chi_2,s}(\p^{m+d_{\p}}).\nonumber
\end{equation}

If $\p\mid\n$ we can not give an explicit formula for $W_{s,\p}$ in general. However, we will relate it to the normalized Whittaker function studied in \cite{Sa15_2}. Let $W_{\pi(s),\p}$ be the Whittaker new vector for $\pi(s)$ with respect to $\psi_{\p}'$ normalized by $W_{\pi(s),\p}(1)=1$. Then
\begin{equation}
	W_{s,\p}(g) = \abs{\varpi_{\p}^{-d_{\p}}}_{\p}^{\frac{1}{2}-s}\chi_{2,\p}( \varpi_{\p}^{-d_{\p}}) \tilde{W}_{s,\p}(1) W_{\pi(s),\p}(a(\varpi_{\p}^{d_{\p}})g). \nonumber
\end{equation}
Using \red{\cite[Lemma~2.2.1]{Sc02}} we have
\begin{equation}
	\tilde{W}_{s,\p}(1) = \red{q_{\p}^{-\frac{a(\chi_{1,\p})}{2}}}\Lambda_{\psi_{\p}'}f_s =  \red{q_{\p}^{-\frac{a(\chi_{1,\p})}{2}}} \epsilon_{\p}(s+1,\chi_{2,\p}^{-1},\red{\psi}_{\p}'^{-1}).\nonumber
\end{equation}
This implies
\begin{equation}
	W_{s,\p}(g) = q_{\p}^{(\frac{1}{2}-s)d_{\p}-\redd{\frac{a(\chi_{1,\p})}{2}-(s+\frac{1}{2})}a(\chi_{2,\p})}\chi_{2,\p}(\varpi_{\p}^{-d_{\p}})\epsilon(\frac{1}{2},\chi_{2,\p}^{-1},\red{\psi}_{\p}'^{-1})W_{\pi(s),\p}(a(\varpi_{\p}^{d_{\p}})g). \nonumber
\end{equation}

In order to summari\red{z}e this and write down the Whittaker expansion in \red{a} uniform manner we introduce some more notation. First, for $y\in \R_+^n$ and $\mathfrak{a}  = \prod_{\p\nmid \n}\p^{m_{\p}}$, we introduce 
\begin{eqnarray}
	W_{\infty,s}(a(y)) &=& \prod_{\nu}\red{\abs{y_{\redd{\nu}}}_{\nu}^{\frac{1}{2}+\frac{1}{2}s_{\nu}}}\frac{K_{[F_{\nu}:\R](s+\red{i}t_{\nu})}([F_{\nu}:\R]2\pi\abs{y_{\nu}})}{\Gamma_{\nu}(2(s+t_{\nu})+1)} \text{ and } \label{eq:arch_whitt_eis} \\
	W_{\text{ur},s}(a((\varpi_{\p}^{m_{\p}})_{\p\nmid \n})) &=& \frac{\delta_{\mathfrak{a}\subset \mathfrak{d}^{-1}[\mathfrak{d}]_{\n}}}{\sqrt{\No(\mathfrak{a\red{\frac{\mathfrak{d}}{[\mathfrak{d}]_{\n}}}})}}\eta_{\chi_{1},\chi_2,s}(\mathfrak{a}\frac{\mathfrak{d}}{[\mathfrak{d}]_{\n}}). \label{eq:spherical_whitt_eis}
\end{eqnarray}
At the places $\nu$ and $\p\nmid\n$ the new vector is spherical so that we can define $W_{\infty,s}(g)$ and $W_{\text{ur},s}(g)$ in the obvious way. At the remaining places we put
\begin{eqnarray}
	W_{\n,s}(g) &=& \prod_{\p\mid\n}W_{\pi(s),\p}(a(\varpi_{\p}^{d_{\p}})g_{\p}), \nonumber \\
	b_{\p}(s) &=& q_{\red{\p}}^{-\redd{\frac{a(\chi_{1,\p})}{2}-(\frac{1}{2}+s)}a(\chi_{2,\p})}\epsilon(\frac{1}{2},\chi_{2,\p}^{-1},\red{\psi}_{\p}'^{-1})L_{\p}(2s+1,\chi_{1,\p}\chi_{2,\p}^{-1}) \nonumber \\
	b_r(s) &=&\prod_{\p\mid\n}b_{\p}(s). \label{eq:def_of_b_ram}
\end{eqnarray}
We have shown that
\begin{equation}
	W_s(g) = \frac{\red{2^{r_1+2r_2}}\No(\mathfrak{d})^{-s}\chi_2^{-1}(\mathfrak{d})}{L(2s+1,\chi_1\chi_2^{-1})}b_r(s)W_{\text{ur},s}(g_{ur})W_{\n,s}(g_{\n})W_{\infty\red{,s}}(g_{\infty}), \nonumber
\end{equation}
for $g=g_{\text{ur}} g_{\n} g_{\infty}$.

\subsection{The Whittaker expansion of $E(s,g)$}

Summarizing the computations from the previous two subsections we obtain
\begin{multline}
	E(s,g) = v^{\circ}(s)(g) + \frac{c_r(s)}{\sqrt{\No(\mathfrak{d})}} \frac{\Lambda(2s,\chi_1\chi_2^{-1})}{\Lambda(2s+1,\chi_1\chi_2^{-1})}\hat{v}^{\circ}(-s)(g)\\
	+ \frac{\red{2^{r_1+2r_2}}\No(\mathfrak{d})^{-s}\chi_2^{-1}(\mathfrak{d})}{L(2s+1,\chi_1\chi_2^{-1})}b_r(s)\sum_{q\in F^{\times}} W_{ur,s}(a(q) g_{ur})W_{\n,s}(a(q) g_{\n})W_{\infty,s}(a(q)g_{\infty}). \label{eq:Whittaker_exp_E}
\end{multline}
Here the constants $c_{r}(s)$ and $b_r(s)$ come from the choice of $v^{\circ}$ at the ramified places and are explicitly given in \eqref{eq:def_of_c_ram} and \eqref{eq:def_of_b_ram}. 

We define the truncated Eisenstein series
\begin{eqnarray}
	F(s,g) &=& E(s,g) -  v^{\circ}(s)(g) - \frac{c_r(s)}{\sqrt{\No(\mathfrak{d})}} \frac{\Lambda(2s,\chi_1\chi_2^{-1})}{\Lambda(2s+1,\chi_1\chi_2^{-1})}\hat{v}^{\circ}(-s)(g). \nonumber
\end{eqnarray}
The upshot is that the Whittaker expansion of $F$ has no constant term and many estimates will carry over from the cusp form case considered in \cite{As17_2}. In the following we will bring $F$ in the necessary shape.

For now let us fix $1 \leq i\leq h_F$, $g\in \mathcal{J}_{\n}$ and  $n(x) a(y)\in \mathcal{F}_{\n_2}$. Then \eqref{eq:arch_whitt_eis} implies that
\begin{equation}
	\abs{W_{\infty,s}(a(q)n(x)a(y))} = \abs{qy}_{\infty}^{\frac{1}{2}} \prod_{\nu} \frac{\abs{K_{[F_{\nu}\colon \R](s+it_{\nu})}([F_{\nu}\colon \R] 2\pi \abs{q_{\nu}y_{\nu}})}}{\abs{\Gamma_{\nu}(2(s+it_{\nu})+1)}}. \label{eq:arch_set_up_eisenstein}
\end{equation}
Due to \eqref{eq:spherical_whitt_eis} it is easy to control the unramified coefficients. Indeed, we have
\begin{equation}
	\abs{W_{\text{ur},s}(a(q \theta_i))} = \begin{cases}  
		\abs{q\theta_i\red{\mathfrak{d}}}_{ur}^{\frac{1}{2}}\abs{\eta_{\chi_1,\chi_2,s}\left( \frac{(q)\red{\theta_i\mathfrak{d}}}{[(q)\red{\theta_i\mathfrak{d}}]_{\n}} \right)} &\text{ if } v_{\p}(q\theta_i) \geq -v_{\p}(\mathfrak{d}) \text{ for all} \p\nmid \n, \ \\
		0 &\text{ else.}
	\end{cases}\nonumber
\end{equation} 
The generalized divisor sum $\eta_{\chi_1,\chi_2,s}$ was defined in \eqref{eq:def_of_gen_div_sum} \red{and $\abs{\cdot}_{ur} = \prod_{\p\nmid \n} \abs{\cdot}_{\p}$}. \red{Recall the definitions of $l(g_{\p})$, $t(g_{\p})$, $n_{1,\p}(g_{\p})$, $n_{0,\p}=\min(l(g_{\p}),n_{\p}-l(g_{\p}))$, and $m_{1,\p}(g_{\p}\redd{)}=\max(0,n_{0,\p}(g_{\p})-n_{\p}+m_{\p})$ from \cite[Section~2.1]{As17_2}.} As \red{in \cite[(3.2)]{As17_2}} we define the ideals
\begin{equation}
	\red{\n_1 = \n_0\n_2 = \prod_{\p} \p^{n_{1,\p}}, \quad \mathfrak{m}_1(g) = \prod_{\p} \p^{m_{1,\p}(g_{\p})} \text{ and }} \imath= \n_0\mathfrak{m}_1(g)\red{\mathfrak{d}} \prod_{\p}\p^{v_{\p}(\theta_i)}.\label{eq:def_of_several_ideals}
\end{equation}
\red{Note that $\mathcal{N}(\mathfrak{m}_1(g))\leq \mathcal{N}(\mathfrak{m}_1)$, where  $\mathfrak{m}_1 = \frac{\mathfrak{m}}{(\mathfrak{m},\n_2\n_0)}$ as in Theorem~\ref{th:main_th_3}.}

We write 
\begin{equation}
	\lambda_{\n,s}(q) = W_{\n,s}(a(q\theta_i)g_{\red{\n}}) =  \prod_{\p\mid\n} W_{\pi(s),\p}(a(\varpi_{\p}^{v_{\p}(\mathfrak{d})}\theta_i q)g_{\p}). \nonumber
\end{equation}

Towards the support of $\lambda_{\n,s}(q)$ we can prove the following lemma.
\begin{lemma} 
If $\lambda_{\n,s}(q)\neq 0$ then $v_{\p}(q) \geq  -v_{\p}(\theta_i)-v_{\p}(\mathfrak{d})-n_{0,\p}-m_{1,\p}(g_p)$ for all $\p\mid \n$.
\end{lemma}
This is essentially \cite[Lemma~3.11]{Sa15}. The only difference is that if $\Im(s) \neq 0$, we are not dealing with unitary representations. 
\begin{proof}
If $\lambda_{\n,s}(q)\neq 0$ then $W_{\pi(s),\p}(a(\varpi_{\p}^{v_{\p}(\mathfrak{d})}\theta_i q)g_{\p})\neq 0$ for \red{all} $\p\mid \n$ so that Proposition~\ref{pr:supp_non_unitary_W} implies
\begin{equation}
	v_{\p}(\theta_i q)+v_{\p}(\mathfrak{d})+ t(g_{\p}) \geq -\max(2l(g_{\p}),l(g_{\p})+m_{\p},n_{\p}) \text{ for all } \p\mid \n. \nonumber
\end{equation}
Since $g\in \mathcal{J}_{\n}$ we have $g_{\p} \in K_{\p}a(\varpi_{\p}^{n_{1,\p}})$ \red{and $n_{1,\p}(g_{\p})=n_{0,\p}$.} \red{Thus, if $n_{1,\p}>n_{0,\p}$, then} $l(g_{\p})\leq n_{0,\p}$. Furthermore, by \cite[Lemma~2.4]{Sa15} we have
\begin{equation}
	t(g_{\p}) = \min(n_{1,\p}-2l(g_{\p}),-n_{1,\p}) = -n_{1,\p}. \nonumber
\end{equation}
It follows that
\begin{equation}
	v_{\p}(\theta_i q)+v_{\p}(\mathfrak{d}) \geq -n_{0,\p}-\max(0,l(g_{\p})-n_{\p}+m_{\p}) \red{=} -n_{0,\p}-m_{1,\p}(g_{\p}). \label{eq:helpful_eq1}
\end{equation}

\red{On the other hand if $n_{1,\p}=n_{0,\p}$, we might encounter the situation $l_{\p}(g_{\p}) > n_{0,\p}$. In this case $t(g_{\p})=n_{0,\p}-2l(g_{\p})$, $n_{0,\p}(g_{\p})= n_{\p}-l(g_{\p})$, and $m_{1,\p}(g_{\p}) = \max(0,m_{\p}-l(g_{\p}))$. We estimate
\begin{align}
	v_{\p}(\theta_i q)+v_{\p}(\mathfrak{d}) &\geq -t(g_{\p})-\max(2l(g_{\p}),l(g_{\p})+m_{\p},n_{\p}) \nonumber \\ 
	&= -n_{0,\p}+2l(g_{\p})-\max(2l(g_{\p}),l(g_{\p})+m_{p}) \nonumber \\
	&= -n_{0,\p}-\max(0,m_{\p}-l(g_{\p})) = -n_{0,\p}-m_{1,\p}(g_{\p}). \label{eq:helpful_eq2}
\end{align}
Otherwise we can argue as above.
}
\end{proof}

\begin{rem} \label{rem:constants_for_eis}
Since the central character of $\pi(s)$ is unitary (and independent of $s$) the statement of \cite[Lemma~3.4]{As17_2} holds also for $W_{\pi(s),\p}$. The proof remains the same.
\end{rem}

The Whittaker expansion of $F$ is given by
\begin{eqnarray} 
	F(s,a(\theta_i\red{)} g n(x)a(y)) &=& \frac{\red{2^{r_1+2r_2}}\No(\mathfrak{d})^{\red{-s}}\chi_2^{-1}(\mathfrak{d})}{L(2s+1,\chi_1\chi_2^{\red{-1}})}b_r(s) \label{eq:whitt_eis_detail} \\
	&&\qquad \cdot\sum_{0\neq q\in \imath^{-1}} \abs{q \imath}_{ur}^{\frac{1}{2}} \eta_{\chi_1,\chi_2,s}\left( \frac{(q)\imath}{[(q)\imath]_{\n}} \right) \lambda_{\n,s}(q) W_{\infty,s}(a(q)n(x)a(y)).  \nonumber
\end{eqnarray}
In analogy to the notation in \cite{As17_2} we write
\begin{equation}
	\lambda_{\text{ur},s}(q) = \abs{q \imath}_{ur}^{\frac{1}{2}} \eta_{\chi_1,\chi_2,s}\left( \frac{(q)\imath}{[(q)\imath]_{\n}} \right).\nonumber
\end{equation}

\begin{rem} \label{rem:rel_forEisenstein}
Since $E(s,g)$ is the Eisenstein series associated to a new vector in an induced representation it is an eigenfunction of the Hecke operators $T(\mathfrak{a})$ for $(\mathfrak{a,\n})=1$. Similar to \cite[Lemma~3.2]{As17_2} we obtain
\begin{equation}
	\lambda_{\text{ur},s}(q) = \frac{\lambda_{s}\left( \frac{(q)\imath}{[(q)\imath]_{\n}} \right) }{\No\left( \frac{(q)\imath}{[(q)\imath]_{\n}} \right)}. \nonumber
\end{equation}
Where $\lambda_s(\mathfrak{a})$ is the corresponding Hecke eigenvalue. In particular we can express the Hecke eigenvalues in terms of generalized divisor sums.
\end{rem}

We obtain the following proposition which establishes good control on $F$ high up in the cusp.
 
\begin{prop} \label{pr:whit_exp_est_bad}
For $g\in \red{\mathcal{J}}_{\n}$  we have
\begin{eqnarray}
	&&\abs{F(iT_0,a(\theta_i)gn(x)a(y))} \nonumber \\
	&&\ll_{F,\epsilon} (\abs{T}_{\infty}\abs{y}_{\infty}^{-1}\No(\n))^{\epsilon} \bigg( \abs{T}_{\infty}^{\frac{1}{6}}\No(\n_0)^{\frac{1}{2}}+\abs{T}_{\infty}^{\frac{1}{6}}\abs{\frac{T}{y}}_{\infty}^{\frac{1}{4}}\No(\n_0\mathfrak{m}_1(g))^{\frac{1}{4}}+\abs{\frac{T}{y}}_{\infty}^{\frac{1}{2}}\No(\n_0\mathfrak{m}_1(g))^{\frac{1}{2}} \bigg). \nonumber
\end{eqnarray}
\end{prop}
\begin{proof}
We have
\begin{equation}
	c_{F}\red{(s)} \colon = \frac{\red{2^{r_1+2r_2}}\No(\mathfrak{d})^{\red{-s}}\chi_2^{-1}(\mathfrak{d})}{L(\red{2s}+1,\chi_1\chi_2^{\red{-1}})}b_r(\red{s}) \ll_{F\red{,\Re(s)}} \frac{\abs{b_r(\red{s})}}{\abs{L(\red{2s}+1,\chi_1\chi_2^{-1})}}. \label{eq:def_cFs}
\end{equation}
One checks by hand that
\begin{equation}
	b_{r}(s)  \ll\red{ \zeta_{[\n]_{\mathfrak{l}}}(1+2\Re(s)) \prod_{\p\mid \n } q_{\p}^{\redd{-\frac{a(\chi_{1,\p})}{2}-(\Re(s)+\frac{1}{2})a(\chi_{2,\p})}}, \text{ for }\Re(s)>-\frac{1}{2}}. \label{eq:rough_br}
\end{equation}
In particular we have
\begin{equation}
	b_r(s) \ll_{F,\epsilon\red{,\Re(s)}} \No(\n)^{\epsilon}, \red{\text{ for } s>-\frac{1}{2}}. \nonumber
\end{equation}
\red{Thus, in order to bound $c_F(s)$ we still need a suitable lower bound for $L(1+2s, \redd{\chi_1\chi_2^{-1}})$. Indeed, the bound
\begin{equation}
	L(1+2s, \chi_1\chi_2^{-1}) \gg_{F} \log(\No(\mathfrak{l})(2\abs{\Im(s)}+3))^{\redd{-2}}, \text{ for }-\frac{c}{\log(\No(\mathfrak{l})(2\abs{\Im(s)}+3))^{\redd{2}}} <\Re(s) \text{ and } \abs{\Im(s)}\gg 1, \nonumber
\end{equation}
for some small constant $c$ is folklore. However, since we could not locate a suitable reference, let us sketch the proof. First, one deduces from \cite[(5.27), (5.28)]{IK04} and the standard zero-free region for $L(s,\chi_1\redd{\chi_2^{-1}})$, \cite[Theorem~5.35]{IK04}, that
\begin{equation}
	\frac{L'(1+2s, \chi_1\chi_2^{-1})}{L(1+2s, \chi_1\chi_2^{-1})} \ll_F \log(\No(\mathfrak{l})(2\abs{\Im(s)}+3))^2,\nonumber
\end{equation}
when $s$ satisfies $-\frac{c}{\log(\No(\mathfrak{l})(2\abs{\Im(s)}+3))^{\redd{2}}} <\Re(s)$ and $\abs{\Im(s)}\gg 1$. We conclude by observing that
\begin{multline*}
	\log\left(\frac{1}{\abs{L(1+2s, \chi_1\chi_2^{-1})}}\right) =-\Re(\log(L(1+2s, \chi_1\chi_2^{-1})))  \\
	= -\Re\left(\log\left(L\left(1+\frac{1}{\log(\No(\mathfrak{l})(2\abs{\Im(s)}+3))^2}+\Im(s), \chi_1\chi_2^{-1}\right)\right)\right) \\
	+\int_{1+2\Re(s)}^{1+\frac{1}{\log(\No(\mathfrak{l})(2\abs{\Im(s)}+3))^2}}\Re\left(\frac{L'(u+\Im(s),\chi_1\chi_2^{-1})}{L(u+\Im(s),\chi_1\chi_2^{-1})}\right)du\nonumber
\end{multline*}
and estimating trivially.

We deduce
\begin{equation}
	\abs{c_F(\sigma+iT_0)} \ll_{F, \sigma, \epsilon} (\No(\n)\abs{T}_{\infty})^{\epsilon} \text{ for }-\frac{c}{\log(\No(\n)\abs{T}_{\infty})^{\redd{2}}} < \sigma. \label{eq:bounding_cFs}
\end{equation}}

If we use this estimate to replace \cite[Lemma~3.5]{As17_2}, we can follow step by step the proof of \cite[Proposition~3.1]{As17_2}. This yields the desired bound.
\end{proof}

\subsection{Preliminary estimates for Eisenstein series}

As a result of the previous section we can control $F(\red{s,g})$ for $s\in i \R$ via its Whittaker expansion. However, later on we will need estimates for $F$ when $s$ is not purely imaginary. The goal of this section is to establish such estimates following closely the arguments in \cite{Yo15} and \cite{HX16}.

Let us write  $s= \sigma+it$. For convenience we restrict ourselves to $g=a(\theta_i)g'n(x)a(y)$. Let us recall some properties of $g$. First, note that $g_{\p} = a(\theta_{i,\p})$ for all $\p\nmid \n$. For $\p\mid\n$ we claim that
\begin{equation}
	H_{\p}(g_{\p})= \red{\begin{cases}
		\abs{\theta_i}_{\p}q_{\p}^{-n_{0,\p}}&\text{ if } n_{1,\p}=n_{0,\p} \text{ and } l(g_{\p})>n_{0,\p}, \\
		\abs{\theta_i}_{\p}q_{\p}^{n_{1,\p}-2l(g_{\red{\p}})} &\text{ else.}
	\end{cases}} \label{eq:special_desc_of_H}
\end{equation}
This can be seen as follows. The definition of $\mathcal{J}_{\n}$ implies $g_{\p}=a(\theta_{i,\p})ka(\varpi_{\p}^{n_{1,\p}})$ with \red{$n_{1,\p}(g_{\p}) = n_{0,\p}$}. \red{Thus, if $n_{1,\p}>n_{0,\p}$ or $n_{1,\p}+n_{0,\p}$ and $l(ka(\varpi_{\p}^{n_{1,\p}}))\leq n_{0,\p}$, then \cite[Lemma~2.4]{Sa15} implies that $t(ka(\varpi_{\p}^{n_{1,\p}})) = \min(n_{1,\p}-2l(g_{\p}),-n_{1,\p}) = -n_{1,\p}$. In this case the claim is easily deduced from \cite[Remark~2.1]{Sa15}. If $n_{1,\p}=n_{0,\p}$ and $l(ka(\varpi_{\p}^{n_{1,\p}}))> n_{0,\p}$, then $t(ka(\varpi_{\p}^{n_{1,\p}}))=n_{0,\p}-2l(g_{\p})$. One also concludes using  \cite[Remark~2.1]{Sa15}.}
The archimedean part is controlled by \cite[Lemma~\redd{6}]{BHMM16}. Indeed,
\begin{equation}
	H_{\infty}(g) = \abs{y}_{\infty} \gg \No(\n_2)^{-1}. \nonumber
\end{equation}

Further, we recall some standard estimates for the $K$-Bessel function. The basic bound
\begin{equation}
	K_{\sigma+it}(x) \ll_{\red{\sigma,\epsilon}} \begin{cases} 
		(1+\abs{t})^{\red{\abs{\sigma}}+\epsilon}x^{-\red{\abs{\sigma}}-\epsilon}e^{-\pi\frac{\abs{t}}{2}} &\text{ for }0<x\leq 1+\pi\frac{\abs{t}}{2}, \\
		 x^{-\frac{1}{2}}e^{-x} &\text{ for } x> 1+\pi\frac{\abs{t}}{2}
	\end{cases} \label{eq:HM_bessel_est}
\end{equation}
\red{follows from \cite[Proposition~\redd{9}]{HM06} together with  $K_{-\sigma+it}(x)=\overline{K_{-\sigma-it}(x)}= \overline{K_{\sigma+it}(x)}$. Furthermore, for $\sigma>\redd{-}\frac{1}{2}$ Stirling's approximation yields
\begin{equation}
	\Gamma(\frac{1}{2}+\sigma+it)\gg_{\sigma} (1+\abs{t})^{\sigma}e^{-\frac{\pi\abs{t}}{2}}. \nonumber
\end{equation}
Thus, we deduce that 
\begin{equation}
	\frac{K_{\sigma+it}(y)}{\Gamma(\frac{1}{2}+\sigma+it)} \ll_{\sigma,\lambda} y^{-\abs{\sigma}-\lambda}(1+\abs{t})^{\abs{\sigma}-\sigma+\lambda} \red{\text{ for }\lambda>0 \text{ and } \sigma>-\frac{1}{2}}.\nonumber
\end{equation}
Observe that when $-\frac{c}{\log(1+\abs{t})}<\sigma<0$ we have $(1+\abs{t})^{2\abs{\sigma}} \ll 1$. Thus, we obtain
\begin{equation}
	\frac{K_{\sigma+it}(y)}{\Gamma(\frac{1}{2}+\sigma+it)} \ll_{\sigma,\lambda} y^{-\abs{\sigma}-\lambda}(1+\abs{t})^{\lambda} \red{\text{ for }\lambda>0 \text{ and } \sigma\gg-\frac{1}{\log(1+\abs{t})}}.\label{eq:yo_bound_bessel_real}
\end{equation}
This is weaker than \eqref{eq:HM_bessel_est}, but sufficient for our purposes. A similar estimate has been used in \cite{Yo15}.} From this we can derive a bound which will be suitable for the complex places. Indeed, using Stirling's approximation to estimate $\Gamma( \frac{1}{2}+\sigma+it)/\Gamma(1+\sigma+it)$ one obtains
\begin{equation}
	\frac{K_{2(\sigma+it)}(y)}{\Gamma(1+2(\sigma+it))} \ll_{\sigma,\lambda} y^{-2\red{\abs{\sigma}}-\lambda}(1+\abs{t})^{\lambda-\frac{1}{2}}  \red{\text{ for }\lambda>0 \text{ and } \sigma\gg-\frac{1}{\log(1+\abs{t})}}. \label{eq:yo_bound_bessel_complex}
\end{equation}

 From this we deduce the bound
\begin{equation}
	\abs{W_{\infty,s}(a(q)n(x)a(y))} \ll_{F,\sigma,\lambda} \abs{qy}_{\infty}^{\frac{1}{2}-\abs{\sigma}-\lambda-\epsilon}\abs{T}_{\infty}^{\lambda+\epsilon}\abs{T}_{\C}^{-\frac{1}{4}}\red{\text{ for }\lambda>0 \text{ and } \sigma\gg-\frac{1}{\log(1+\abs{t})}}. \nonumber
\end{equation}
Where we write $T=(T_{\nu})_{\nu}$ with $T_{\nu} = \red{\max}(\frac{1}{2},\abs{t_{\nu}+t})$.

Finally, we record the trivial bound
\begin{equation}
	\lambda_{\text{ur},s}(q) = \abs{q\imath}_{\text{ur}}^{\frac{1}{2}}\eta_{\chi_1,\chi_2,s}\left( \frac{(q)\imath}{[(q)\imath]_{\n}}\right) \ll_{\epsilon,F} \abs{q\imath}_{\text{ur}}^{\frac{1}{2}-\abs{\sigma}-\epsilon}. \nonumber
\end{equation}

We are now ready for our first estimate.

\begin{lemma} \label{lm:prelim_est_F}
Let $g=a(\theta_i) g' n(x)a(y)$ with $g\red{'}\in \mathcal{J}_{\n}$ and $n(x)a(y)\in \mathcal{F}_{\n_2}$ as usual. If $s = \sigma+it$ for $\sigma>-\frac{c}{\log(\No(\n)\abs{T}_{\infty})^{\redd{2}}}$ and $t\neq 0$, then
\begin{multline}
	F(s,g) \ll_{F,\sigma,\red{\epsilon}} \No(\n)^{\abs{\sigma}+\epsilon}\No(\imath)^{\abs{\sigma}\red{+\epsilon}}\abs{T}_{\infty}^{1+\epsilon}\abs{y}_{\infty}^{-\frac{1}{2}-\abs{\sigma}}\\ \cdot \left( \No(\imath)^{\frac{1}{2}}\red{+\No(\imath)^{\frac{1}{4}}\No(\n_0(g))^{\frac{1}{2}}\abs{y}_{\infty}^{\frac{1}{4}}}+ \No(\n_0(g))^{\frac{1}{2}}\abs{y}_{\infty}^{\red{\frac{1}{2}}} \right). \nonumber
\end{multline}
\end{lemma}
\begin{proof}
Applying the H\"older inequality yields
\begin{multline}
	F(s,g) \ll_{F,\sigma,\red{\epsilon}}	\abs{c_F(s)}\left(\sum_{0\neq q \in \imath^{-1}} \abs{q}_{\infty}^{4\epsilon+4\abs{\sigma}-1}\abs{W_{\infty}(\red{a(qy)})}^4  \prod_{\nu}(1+\abs{q_{\nu}y_{\nu}}_{\nu})^{2+4\epsilon} \right)^{\frac{1}{4}} \nonumber \\
	\cdot  \left( \sum_{0\neq q\in \imath^{-1}}\abs{q}_{\text{fin}}^{\frac{4\epsilon}{3}+\frac{4\abs{\sigma}}{3}-\frac{1}{3}} \abs{\lambda_{\text{ur},s}(q)}^{\frac{4}{3}}\abs{\lambda_{\n,s}(q)}^{\frac{4}{3}} \prod_{\nu} (1+\abs{q_{\nu}y_{\nu}}_{\nu})^{-\frac{2}{3}-\frac{4\epsilon}{3}}  \right)^{\frac{3}{4}}, \nonumber
\end{multline}
\red{where $c_F(s)$ was defined in \eqref{eq:def_cFs}.}

We put
\begin{eqnarray}
	S_1 &=& \sum_{0\neq q \in \imath^{-1}} \abs{q}_{\infty}^{4\epsilon+4\abs{\sigma}-1}\abs{W_{\infty}(\red{a(qy)})}^4  \prod_{\nu}(1+\abs{q_{\nu}y_{\nu}}_{\nu})^{2+4\epsilon} \text{ and } \nonumber \\
	S_2 &=& \sum_{0\neq q\in \imath^{-1}}\abs{q}_{\text{fin}}^{\frac{4\epsilon}{3}+\frac{4\abs{\sigma}}{3}-\frac{1}{3}} \abs{\lambda_{\text{ur},s}(q)}^{\frac{4}{3}}\abs{\lambda_{\n,s}(q)}^{\frac{4}{3}} \prod_{\nu} (1+\abs{q_{\nu}y_{\nu}}_{\nu})^{-\frac{2}{3}-\frac{4\epsilon}{3}} \nonumber
\end{eqnarray}
and deal with each one of these sums on its own.

We start by evaluating $S_1$. To do so we define the boxes
\begin{equation}
	J(\du{n}) = \{ q\in a\imath^{-1} \colon n_{\nu} <\abs{a^{-1}q_{\nu}y_{\nu}}_{\nu} \leq n_{\nu}+1  \text{ for all }\nu\} \text{ for } \du{n} \in \N_0^{r_1+r_2}. \nonumber
\end{equation}
Here $a\in \imath$ is chosen as in \cite[(3.4)]{As17_2}. It is an easy exercise to modify \cite[Lemma~3.8]{As17_2} to get  
\begin{equation}
	\sharp J(\du{n}) \ll_F \prod_{\nu} f_{\nu}(n_{\nu}) \nonumber
\end{equation}
with
\begin{equation}
	f_{\nu}(n_{\nu}) = \begin{cases}
		\frac{\abs{a}}{\abs{y_{\nu}}}+1 &\text{ if $\nu$ is real,} \\
		(n_{\nu}+1)\left(\frac{\abs{a}}{\abs{y_{\nu}}}+1\right)^2 &\text{ if $\nu$ is complex.} 
	\end{cases} \nonumber
\end{equation}
This leads to the estimate
\begin{equation}
	S_1 \ll \abs{y}_{\infty}^{1-4\abs{\sigma}-4\epsilon} \prod_{\nu} 	S_{1,\nu}, \nonumber
\end{equation}
where
\begin{equation}
	S_{1,\nu} = \sum_{n_{\nu} \geq 0} f_{\nu}(n_{\nu}) \abs{y_{\nu}q_{\nu}a^{-1}}_{\nu}^{4\epsilon+4\abs{\sigma}-1} (1+\abs{y_{\nu}q_{\nu} a^{-1}}_{\nu})^{2+4\epsilon} \abs{W_{\nu,s}(y_{\nu}q_{\nu}a^{-1})}^4. \nonumber
\end{equation}
If $\nu$ is real we use \eqref{eq:yo_bound_bessel_real} with $\lambda = \epsilon$ for $n_{\nu}=0$ and $\lambda = 1+3\epsilon$ otherwise to get
\begin{equation}
	S_{1,\nu} \ll_{\sigma,\epsilon} \abs{T_{\nu}}_{\nu}^{\epsilon}\left(\frac{\abs{a}_{\nu}}{\abs{y_{\nu}}_{\nu}}+1\right) + \abs{T_{\nu}}_{\nu}^{4+\epsilon}\left(\frac{\abs{a}_{\nu}}{\abs{y_{\nu}}_{\nu}}+1\right)\sum_{n_{\nu}>0} n_{\nu}^{-1-\epsilon} \ll_{\epsilon}  \abs{T_{\nu}}_{\nu}^{4+\epsilon}\left(\frac{\abs{a}_{\nu}}{\abs{y_{\nu}}_{\nu}}+1\right).\nonumber
\end{equation}
For $\nu$ complex we choose $\lambda = \epsilon$ for $n_{\nu}=0$ and $\lambda = \frac{5}{\red{4}}+\red{3}\epsilon$ in \eqref{eq:yo_bound_bessel_complex} and  get
\begin{equation}
	S_{1,\nu} \ll_{\sigma,\epsilon} \abs{T_{\nu}}_{\nu}^{-1+\epsilon}\left(\frac{\abs{a}}{\abs{y_{\nu}}} + 1 \right)^2+ \abs{T_{\nu}}_{\nu}^{4+\epsilon}\left(\frac{\abs{a}}{\abs{y_{\nu}}} + 1 \right)^2 \sum_{n_{\nu}>0} n_{\nu}^{-1-\epsilon} \ll _{\epsilon} \abs{T_{\nu}}_{\nu}^{4+\epsilon}\left(\frac{\abs{a}_{\nu}}{\abs{y_{\nu}}_{\nu}} + 1 \right). \nonumber
\end{equation}
We conclude
\begin{equation}
	S_1 \ll_{\sigma,\epsilon} \abs{y}_{\infty}^{1-4\abs{\sigma}-4\epsilon}\abs{T}_{\infty}^{4+\epsilon} \prod_{\nu} \left(\frac{\abs{a}_{\nu}}{\abs{y_{\nu}}_{\nu}} + 1 \right). \nonumber
\end{equation}
Balancing $\abs{a}_{\nu}$ and $\abs{y_{\nu}}_{\nu}$ similarly to \cite[Corollary~3.3]{As17_2} yields
\begin{equation}
	S_1 \ll_{\sigma,\epsilon} \abs{y}_{\infty}^{1-4\abs{\sigma}-4\epsilon}\abs{T}_{\infty}^{4+\epsilon} \left(\frac{\abs{a}_{\infty}}{\abs{y}_{\infty}} + 1 \right). \label{eq:est_S_1_also_neq} 
\end{equation}

Let us turn to $S_2$.  We define the boxes
\begin{equation}
	I(\du{k},\red{\mathfrak{a}}) = \{ q\in F^{\times} \colon 2^{k_{\nu}} \leq 1+\abs{q_{\nu}y_{\nu}}_{\nu} <2^{k_{\nu}+1}  \text{ for all } \nu, (q)\imath = \red{\mathfrak{a}}\}. \nonumber
\end{equation}
By \cite[Corollary~1]{BHMM16} one has $\sharp I(\du{k},\red{\mathfrak{a}})\ll \frac{\red{\prod_{\nu}}2^{\epsilon k_{\nu}}}{\abs{y}_{\infty}^{\epsilon}\No(\red{\mathfrak{a}}\red{\imath^{-1}})^{\epsilon}}$. \red{Applying H\"older's inequality twice results in} 
\begin{eqnarray}
	S_2 &\ll_{F,\epsilon,\sigma}&  \sum_{\du{k}} \prod_{\nu} 2^{-(\frac{2}{3}+\frac{4\epsilon}{3})k_{\nu}} \sum_{\red{\mathfrak{a}}\mid \n^{\infty}} \No(\red{\mathfrak{a}}\imath^{-1})^{\frac{1}{3}-\frac{4\abs{\sigma}}{3}-\frac{4\epsilon}{3}} \nonumber \\
	&&\qquad\qquad\qquad\qquad \cdot\sum_{ (\red{\mathfrak{b}},\n)=1} \No(\red{\mathfrak{b}})^{-\frac{1}{3}-\epsilon} \sum_{0\neq q\in I(\du{k},\red{\mathfrak{a}\mathfrak{b}})} \abs{\lambda_{\n,s}(q)}^{\frac{4}{3}}  \nonumber \\
	&\ll& \abs{y}_{\infty}^{-\frac{\epsilon}{3}}  \sum_{\du{k}}  \prod_{\nu} 2^{-(\frac{2}{3}+\epsilon)k_{\nu}} \sum_{\red{\mathfrak{a}}\mid \n^{\infty}} \No(\red{\mathfrak{a}}\imath^{-1})^{\frac{1}{3}-\frac{4\abs{\sigma}}{3}-\frac{5\epsilon}{3}} \underbrace{\left( \sum_{(\red{\mathfrak{b}},\n)=1} \No(\red{\mathfrak{b}})^{-1-3\epsilon} \right)^{\frac{1}{3}}}_{\ll_{F,\epsilon} 1} \nonumber \\
	&&\qquad\qquad\qquad\qquad\qquad\qquad\qquad \qquad\qquad\qquad\qquad\cdot \left(  \sum_{\substack{(\red{\mathfrak{b}},\n)=1,\\  q\in I(\du{k},\red{\mathfrak{a}\mathfrak{b}})}} \abs{\lambda_{\n,s}(q)}^2\right)^{\frac{2}{3}}. \nonumber
\end{eqnarray}
The crucial part is to estimate the $q$-sum
\begin{equation}
	S_{\text{ram}} = \sum_{\substack{(\red{\mathfrak{b}},\n)=1,\\  q\in I(\du{k},\red{\mathfrak{a}\mathfrak{b}})}} \abs{\lambda_{\n,s}(q)}^2. \nonumber
\end{equation}

As in \red{\cite[(3.14)]{As17_2}} we obtain
\begin{equation}
	S_{\text{ram}} \ll \frac{\No(\red{\mathfrak{a}})\No(\n_0(g))}{\No(\imath)}  F_{R}(\du{k}(\red{\mathfrak{a}}\imath^{-1}))\int_{C^{\imath}(\du{k}(\red{\mathfrak{a}}\imath^{-1}))} \abs{\lambda_{\n,s}(q)}^2d\mu_{\text{fin}}(q)  \nonumber
\end{equation}
with $R=R(\du{k})=((2^{k_{\nu}+1}+1)/y_{\nu})_{\nu}$. \red{The computation of the remaining integral will be reduced to local integrals as follows.}
\begin{eqnarray}
	\int_{C^{\imath}(\du{l})} \abs{\lambda_{\n,s}(q)}^2 d\mu_{\text{fin}}(q) &=& \prod_{\p\nmid\n} \int_{\varpi_{\p}^{-v_{\p}(\imath)}\op_{\p}} 1 d\mu_{\p} \prod_{\p\mid\n} \int_{\varpi_{\p}^{l_{\p}}\op_{\p}^{\times}} \abs{W_{\pi(s),\p}(a(\varpi_{\p}^{v_{\p}(\mathfrak{d})}\theta_i q)g_{\p})}^2 d\mu_{\p}(q) \nonumber\\ 
	&=& \frac{\No(\imath)}{\No([\imath]_{\n})}\zeta_{\n}(1)^{-1} \prod_{\p\vert\n} q_{\p}^{-l_{\p}} \int_{\op_{\p}^{\times}}\abs{W_{\p}(a(\varpi_{\p}^{v_{\p}(\mathfrak{d}\theta_i)+l_{\p}}q)g_{\p})}^2d\mu_{\p}^{\times}(q). \nonumber
\end{eqnarray}
\red{To deal with the remaining integrals we use Proposition~\ref{pr:supp_non_unitary_W}.  In the setting of this proposition we have
\begin{eqnarray}
	r&=& t(a(\varpi^{v_{\p}(\mathfrak{d}\theta_i)+l_{\p}})g_{\p})+\max(2l(g_{\p}),l(g_{\p})+m_{\p},n_{\p}) \nonumber \\
	&=& v_{\p}(\mathfrak{d}\theta_i)+l_{\p}+t(g_{\p})+\max(2l(g_{\p}),l(g_{\p})+m_{\p},n_{\p})= l_{\p}+v_{\p}(\imath).\nonumber
\end{eqnarray}
The last equality follows from \eqref{eq:helpful_eq1} and \eqref{eq:helpful_eq2} as well as from the definition of $\imath$ \redd{given in \eqref{eq:def_of_several_ideals}}. With this at hand one gets
\begin{equation}
	 \int_{\op_{\p}^{\times}}\abs{W_{\p}(a(\varpi_{\p}^{v_{\p}(\mathfrak{d}\theta_i)+l_{\p}}q)g_{\p})}^2d\mu_{\p}^{\times}(q) \ll_{\epsilon} \zeta_F(1)q_{\p}^{(2\abs{\sigma}-\frac{1}{2}+\epsilon)(l_{\p}+v_{\p}(\imath))+(2\abs{\sigma}+\epsilon)n_{\p}}. \nonumber
\end{equation}
Thus the full integral is bounded by}
\begin{eqnarray}
	\red{\int_{C^{\imath}(\du{l})} \abs{\lambda_{\n,s}(q)}^2 d\mu_{\text{fin}}(q) }&\ll&  \No(\n)^{\epsilon+2\abs{\sigma}}\frac{\No(\imath)}{\No([\imath]_{\n})}  \prod_{\p\vert \n} q_{\p}^{(2\abs{\sigma}+\epsilon -\frac{1}{2})(v_{p}(\imath)+l_{\p})-l_{\p}} \nonumber \\
	&=& \No(\n)^{\epsilon+2\abs{\sigma}}\frac{\No(\imath)}{\No([\imath]_{\n})^{\frac{3}{2}-2\abs{\sigma}-\epsilon}}  \prod_{\p\vert \n} q_{\p}^{(2\abs{\sigma}+\epsilon-\frac{3}{2})l_{\p}}. \nonumber
\end{eqnarray} 
We obtain
\begin{equation}
	S_{\text{ram}} \ll \No(\n)^{2\abs{\sigma}+\epsilon}\No(\n_0(g)) \No(\red{\mathfrak{a}})^{2\abs{\sigma}+\epsilon-\frac{1}{2}}F_R(\du{k}(\red{\mathfrak{a}}\imath^{-1}). \nonumber
\end{equation}
Inserting the \red{equation} \cite[Lemma~3.10]{As17_2} for $F_R$ yields
\begin{equation}
	S_{\text{ram}} \ll \No(\n)^{2\abs{\sigma}+\epsilon}\bigg( \No(\n_0(g)) \No(\red{\mathfrak{a}})^{2\abs{\sigma}+\epsilon-\frac{1}{2}} +  \abs{R}_{\infty} \No(\imath)\No(\red{\mathfrak{a}})^{2\abs{\sigma}+\epsilon-\frac{3}{2}} \bigg)\nonumber
\end{equation}
Since $I(\du{k},\red{\mathfrak{a}\mathfrak{b}})$ is empty for all  $\red{\mathfrak{a}\mathfrak{b}}$ if $\prod_{\nu}2^{k_{\nu}} \ll \No(\imath^{-1})\abs{y}_{\infty}$ we can  impose a condition in the $\du{k}$-sum. Therefore, we can estimate
\begin{eqnarray}
	S_2 &\ll& \abs{y}_{\infty}^{-\epsilon} \No(\n\imath)^{\frac{4}{3}\abs{\sigma}+\epsilon} \bigg( \No(\n_0(g))^{\frac{2}{3}} \No(\imath)^{-\frac{1}{3}}\sum_{\prod_{\nu}2^{k_{\nu}}\gg \No(\imath^{-1})\abs{y}_{\infty}} \red{\prod_{\nu}}2^{-(\frac{2}{3}+\epsilon)k_{\nu}} \nonumber \\
	&&\qquad\qquad\qquad\qquad\qquad\qquad\qquad + \No(\imath)^{\frac{1}{3}} \sum_{\du{k}} \abs{R(\du{k})}_{\infty}^{\frac{2}{3}}\prod_{\nu} 2^{-(\frac{2}{3}+\epsilon)k_{\nu}} \bigg) \nonumber \\
	&\ll&\abs{y}_{\infty}^{-\epsilon} \No(\n\imath)^{\frac{4}{3}\abs{\sigma}+\epsilon}  \left( \frac{\No(\n_0(g))^{\frac{2}{3}}}{\abs{y}_{\infty}^{\frac{1}{3}}} + \frac{\No(\imath)^{\frac{1}{3}}}{\abs{y}_{\infty}^{\frac{2}{3}}} \right). \label{eq:est_S_2_also_neq}  
\end{eqnarray}

Combining \red{the estimates \eqref{eq:est_S_1_also_neq} and \eqref{eq:est_S_2_also_neq} for $S_1$ and $S_2$ respectively} yields
\begin{multline}
	F(s,g)\red{ \ll_{F,\sigma} \abs{c_F(s)}S_1^{\frac{1}{4}}S_2^{\frac{3}{4}}  }\\ \ll_{F,\sigma}  \red{\abs{c_F(s)}}\No(\n\imath)^{\abs{\sigma}+\epsilon}\abs{T}_{\infty}^{1+\epsilon}\abs{y}_{\infty}^{-\frac{1}{2}-\abs{\sigma}} \left( \No(\imath)^{\frac{1}{2}}\red{+\No(\imath)^{\frac{1}{4}}\No(\n_0(g))^{\frac{1}{2}}\abs{y}_{\infty}^{\frac{1}{4}}}+ \No(\n_0(g))^{\frac{1}{2}}\abs{y}_{\infty}^{\frac{1}{2}} \right). \nonumber
\end{multline}
\red{We conclude the proof by estimating $c_F(s)$ using \eqref{eq:bounding_cFs}.}
\end{proof}

We will need one more bound which is derived in a quite different fashion.

\begin{lemma} \label{lm:Bound_way_to_bad_in_n_0}
For $\Re(s)>\frac{1}{2}$ we have
\begin{equation}
	F(s,g) \ll_{F,\Re(s)} H(g)^{-\Re(s)}(\red{\No(\n)^{3(\Re(s)-\frac{1}{2})}}H(g)^{\frac{1}{2}}+\red{H(g)^{-\frac{1}{2}}}). \label{eq:first_easy_bound}
\end{equation}
In particular, if $g=a(\theta_i)g'n(x)a(y)$ with $g'\in \mathcal{J}_{\n}$, and $n(x) a(y) \in \mathcal{F}_{\n_2}$\red{, and $\Re(s)=\frac{1}{2}+\epsilon$}, then
\begin{equation}
	\red{F(s,g) \ll_{F,\epsilon} \No(\n)^{4\epsilon} \No(\n_0).} \nonumber
\end{equation}
\end{lemma}
The proof is a generalization of the proof given in \cite[Lemma~3.2]{Yo15} with some ideas from \cite{HX16}.
\begin{proof}
Write $s=\sigma+it$ for some $\sigma>\frac{1}{2}$. Note that for such $s$ the sum in \eqref{eq:averaging_def} absolutely convergent. Thus, we have
\begin{equation}
	F(s,g) = -[M(s)v^{\circ}(s)](g)+\sum_{1\neq\gamma\in B(F) \setminus G(F)} v^{\circ}(s)(\gamma g). \nonumber
\end{equation}
By the very definition of $v^{\circ}$ we observe 
\begin{equation}
	\abs{v^{\circ}(s)(g)}\leq H(g)^{\frac{1}{2}+\Re(s)}. \label{eq:from_ram_to_ev_ur}
\end{equation}
This inspires us to define the Eisenstein series $E_0(s,g)$ which is unramified everywhere and is induced from $\chi_1=\chi_2=1$. We obtain
\begin{eqnarray}
	F(s,g) &\leq& \abs{[M(s)v^{\circ}(s)](g)}+ \sum_{1\neq\gamma\in B(F) \setminus G(F)} H(\gamma g)^{\frac{1}{2}+\sigma} \nonumber \\
	&=&\abs{[M(s)v^{\circ}(s)](g)} - H(g)^{\frac{1}{2}+\sigma} + E_0(\sigma,g). \nonumber
\end{eqnarray}
The point of this is that \red{the} Whittaker expansion for $E_0(s,g)$ has a very nice shape which we will exploit for our estimate. \red{Indeed, $E_0$ is associated to the new vector $\tilde{v}^{\circ}(g)=H(g)^{\frac{1}{2}}\in \mathcal{B}(1,1)$, which is unramified everywhere. Thus, the expansion \eqref{eq:Whittaker_exp_E} used with $1$ instead of $\chi_1$ and $\chi_2$ reads}
\begin{eqnarray}
	E_0(\sigma,g)-H(g)^{\frac{1}{2}+\sigma} &=& \underbrace{\frac{\Lambda_F(2\sigma)}{\sqrt{\No(\mathfrak{d})}\Lambda_F(2\sigma+1)}}_{\ll_{F,\sigma}1}H(g)^{\frac{1}{2}-\sigma}  \nonumber \\
	&&\qquad + \underbrace{\frac{\red{2^{r_1+2r_2}}\No(\mathfrak{d})^{-\sigma}}{\Lambda_F(2\sigma+1)}}_{\ll_{F,\sigma}1}\sum_{q\in F^{\times}} \red{\tilde{W}}_{\text{ur},\sigma}(a(q)g_{ur})\red{\tilde{W}}_{\infty,\sigma}(a(q)g_{\infty}). \nonumber
\end{eqnarray}
where $\Lambda_F(s)$ denotes the completed Dedekind zeta function associated to $F$ \red{and $\tilde{W}_{\text{ur},s}$ and $\tilde{W}_{\infty,s}$ are the fitting unramified Whittaker functions}. To simplify the notation we define the fractional ideal $\mathfrak{a}_g$ by
\begin{equation}
	\nu_{\p}(\mathfrak{a}_g) = -\frac{\log(H_{\p}(g_{\p}))}{\log(q_{\p})} \text{ for all $\p$}. \nonumber
\end{equation}
In particular we have $\No(\mathfrak{a}_g)^{-1}=H_{\text{fin}}(g)$. Observe that
\begin{eqnarray}
	\abs{W_{\text{ur},\sigma}(a(q)g_{ur})W_{\infty,\sigma}(a(q)g_{\infty})} &=& \delta_{q\in\mathfrak{a}_g^{-1}\mathfrak{d}^{-1}} H(g)^{\frac{1}{2}}\eta_{1,1,\sigma}((q)\mathfrak{a}_g\mathfrak{d})\nonumber \\
	&&\qquad\cdot\prod_{\nu}K_{[F_{\nu}:\R]\sigma}([F_{\nu}:\R]2\pi \abs{q_{\red{\nu}}}H_{\nu}(g_{\nu})^{\frac{1}{[F_{\nu}:\R]}}).\nonumber
\end{eqnarray}
We estimate
\begin{equation}
	\eta_{1,1,\sigma}((q)\mathfrak{a}_g\mathfrak{d}) = \No((q)\mathfrak{a}_g\mathfrak{d})^{\sigma}\sum_{\mathfrak{b}\mid (q)\mathfrak{a}_g\mathfrak{d}} \No(\mathfrak{b})^{-2\sigma} \ll_{F,\sigma} H_{\text{fin}}(g)^{-\sigma}\abs{q}_{\text{fin}}^{-\sigma}. \nonumber
\end{equation}

Let us introduce the box 
\begin{eqnarray}
	J(\du{n}) = \bigg\{ q\in \mathfrak{a}_g^{-1}\mathfrak{d}^{-1} \colon \abs{\frac{n_{\nu}H_{\nu}(g_{\nu})^{-\frac{1}{[F_{\nu}:\R]}}}{[F_{\nu}:\R]2\pi}}_{\nu} \leq \abs{q}_{\nu} \leq \abs{\frac{(n_{\nu}+1)H_{\nu}(g_{\nu})^{-\frac{1}{[F_{\nu}:\R]}}}{[F_{\nu}:\R]2\pi}}_{\nu} \bigg\}. \nonumber
\end{eqnarray}
By \cite[Corollary~1]{BHMM16} we have
\begin{equation}
	\sharp J(\du{n}) \ll_F \No(\mathfrak{a}_g))\prod_{\nu} \abs{(n_{\nu}+1)H_{\nu}(g_{\nu})^{-\frac{1}{[F_{\nu}:\R]}}}_{\nu} = H(g)^{-1} \prod_{\nu}\abs{n_{\nu}+1}_{\nu}. \nonumber
\end{equation}

Using \eqref{eq:HM_bessel_est} for $g\in J(\du{n})$ we obtain
\begin{eqnarray}
	\abs{W_{\text{ur},\sigma}(a(q)g_{\text{ur}})W_{\infty,\sigma}(a(q)g_{\infty})} &\ll_{F,\sigma}& 
	H(g)^{\frac{1}{2}}H_{\text{fin}}(g)^{-\sigma} \prod_{\nu} f_{\nu}(n_{\nu}), \nonumber
\end{eqnarray}
for
\begin{equation}
	f_{\nu}(n_{\nu}) = \begin{cases}
		H_{\nu}(g_{\nu})^{-\sigma}\abs{n_{\nu}}_{\nu}^{\sigma-\frac{1}{2[F_{\nu}:\R]}}e^{-n_{\nu}} &\text{ if $n_{\nu}\geq 1$, }\\
		H_{\nu}(g_{\nu})^{-\sigma-\epsilon}H_{\text{fin}}(g)^{-\frac{\epsilon}{\red{r_1+r_2}}} &\text{ if $n_{\nu}=0$.}
	\end{cases} \nonumber
\end{equation}

We derive the estimate
\begin{eqnarray}
	&&\sum_{q\in F^{\times}} W_{\text{ur},\sigma}(a(q)g_{\text{ur}})W_{\infty,\sigma}(a(q)g_{\infty}) \ll_{F,\sigma,\epsilon}  H(g)^{\frac{1}{2}}H_{\text{fin}}(g)^{-\sigma}\sum_{\du{n}} \sharp J(\du{n}) \prod_{\nu}f_{\nu}(n_{\nu})  \nonumber \\
	&\ll&  H(g)^{-\frac{1}{2}-\sigma}\prod_{\nu} \left( H_{\nu}(g_{\nu})^{-\epsilon}H_{\text{fin}}(g)^{-\frac{\epsilon}{\red{r_1+r_2}}}+\underbrace{\sum_{n\geq 1}n^{1+\sigma-\frac{1}{2[F_{\nu}:\R]}}e^{-n}}_{\ll_{\sigma} 1} \right). \nonumber
\end{eqnarray}
\red{By exploiting $Z(F_{\nu})K_{\nu}$ invariance of $\tilde{v}^{\circ}(s)$ one can argue similarly as above \cite[(5.\redd{9})]{BHMM16} and assume that $H_{\nu}(g_{\nu}) \asymp H_{\infty}(g)^{\frac{1}{r_1+r_2}}$.} We conclude
\begin{equation}
	E_0(\sigma,g) -H(g)^{\frac{1}{2}+\sigma} \ll_{F,\sigma,\epsilon} H(g)^{\frac{1}{2}-\sigma}+\red{H(g)^{-\frac{1}{2}-\sigma}}. \nonumber
\end{equation}

We still have to deal with
\begin{equation}
	\abs{[M(s)v^{\circ}(s)](g)} \ll_{F} \abs{c_r(s)\frac{\Lambda(2s,\chi_1\chi_2^{-1})}{\Lambda(2s+1,\chi_1\chi_2^{-1})}}H(g)^{\frac{1}{2}-\sigma}.\nonumber
\end{equation}
We start by estimating the archimedean parts of the completed $L$-function.

For complex places $\nu$ we have
\begin{equation}
	\frac{L_{\nu}(2s,\chi_1\chi_2^{-1})}{L_{\nu}(2s+1,\chi_1\chi_2^{-1})} = \frac{4\pi}{4s+4it_{\nu}} \ll_{\sigma} (1+\abs{t+t_{\nu}}_{\nu})^{-\frac{1}{2}}\ll\abs{T_{s,\nu}}_{\nu}^{-\frac{1}{2}}. \nonumber
\end{equation}
For real places $\nu$ we use Stirling's formula to observe
\begin{equation}
	\frac{L_{\nu}(2s,\chi_1\chi_2^{-1})}{L_{\nu}(2s+1,\chi_1\chi_2^{-1})} = \sqrt{\pi} \frac{\Gamma(s+it_{\nu})}{\Gamma(s+it_{\nu}+\frac{1}{2})} \ll_{\sigma} (1+\abs{t+t_{\nu}}_{\nu})^{-\frac{1}{2}}\ll\abs{T_{s,\nu}}_{\nu}^{-\frac{1}{2}}. \nonumber
\end{equation}
It is clear that
\begin{equation}
	\frac{L(2s,\chi_1\chi_2^{-1})}{L(2s+1,\chi_1\chi_2^{-1})}  \ll_{F,\sigma} 1\red{.} \nonumber
\end{equation}
\red{The contribution of $c_r(s)$ can be estimated place by place using \eqref{eq:def_of_c_ram}. We get
\begin{equation}
	c_r(s)\ll_{F,\sigma} \No(\mathfrak{l})^{2\sigma-\frac{1}{2}}\No(\n)^{\redd{\sigma}-1}\zeta_{[n]_{\mathfrak{l}}}(1+2\sigma)\ll_{F,\sigma} \No(\n)^{3\redd{\sigma}-\frac{3}{2}}, \text{ for } \frac{1}{2}< \sigma. \nonumber
\end{equation}}

Gathering all the estimates together concludes the proof of \eqref{eq:first_easy_bound}. If $g$ is of the special form $a(\theta_i)g'n(x)a(y)$ with $g'\in \mathcal{J}_{\n}$ and $n(x) a(y) \in \mathcal{F}_{\n_2}$,  we use \eqref{eq:special_desc_of_H} \red{to derive the bound 
\begin{equation}
	H(g) \geq \abs{y}_{\infty}\No(\n_0^{-1}\n_2)\gg \No(\n_0)^{-1}. \nonumber
\end{equation}
Using this the claimed bound for $F(s,g)$ follows easily.}
\end{proof}

\section{On averages of generalized divisor sums} \label{sec:averages_div_f}

At this point we prove an asymptotic formula for averages of generalized divisor sums. Due to the generality of the Eisenstein series under consideration we need to consider divisor sums twisted by Gr\"o\ss encharakteren and supported on prime ideals. We will extend the results \cite[Lemma~5.1]{Yo15} and \cite[Lemma~5.1]{HX16} to this setting. Before we continue we will fix some notation.

For two Hecke characters $\chi_1,\chi_2\colon \A_F^{\times}/F^{\times}\to \C^{\times}$ we defined the generalized divisor sum $\eta_{\chi_1,\chi_2,s}$ in \eqref{eq:def_of_gen_div_sum}. Note that $\eta_{\chi_1,\chi_2,s}$ is a function on ideals. Indeed, behind the scenes we used the 1-1 correspondence between Hecke characters and \red{primitive} Gr\"o\ss encharakteren to make this definition. This correspondence is  given in \cite[Corollary~6.14]{Ne13}.

Recall that a so called Gr\"o\ss encharakter modulo $\mathfrak{q}$ is a character $\chi\colon \mathcal{J}(\mathfrak{q}) \to S^1$ such that, restricted to principal integral ideals, it factors through two characters
\begin{eqnarray}
	\chi_f\colon \left(\mathcal{O}_F/\mathfrak{q}\right)^{\red{\times}} \to S^1 \text{ and } \chi_{\infty}\colon F_{\infty}^{\times} \to S^1. \nonumber
\end{eqnarray} 
The characters of the multiplicative group $F_{\infty}^{\times}$ are well\red{-}understood. See for example \cite[Proposition~6.7]{Ne13}. Each $\chi_{\infty}$ is uniquely determined by its type $(\du{p},\du{q})$. More precisely, we have
\begin{equation}
	\chi_{\infty}(\alpha) = \prod_{\nu} \left(\frac{\alpha_{\nu}}{\abs{\alpha_{\nu}}}\right)^{p_{\nu}} \abs{\alpha_{\nu}}_{\nu}^{iq_{\nu}}, \nonumber
\end{equation}
for some $q_{\nu}\in \R^{\times}$ and $p_{\nu} \in \Z$.

The notion of Gr\"o\ss encharakter includes the classical Dirichlet characters $\xi$. Following \cite[Definition~6.8]{Ne13} these are characters 
\begin{equation}
	\xi\colon Cl_F^{\mathfrak{q}} \to S^1, \nonumber
\end{equation}
of the narrow ray class group. We will usually consider them as functions $\xi\colon \mathcal{J}(\mathfrak{q})\to S^1$ such that $\mathcal{P}_{\mathfrak{q}}\subset \text{ker}(\xi)$. \red{ Recall that $\mathcal{P}_{\mathfrak{q}}$ was defined below \eqref{eq:rough_statement_div_s}.} It is well\red{-}known that a Dirichlet character (defined as above) corresponds to a Gr\"o\ss encharkter with infinity-type $(\du{p},0)$, where $p_{\nu} =0$ for all complex $\nu$.

Let us consider the character $\chi_1\chi_2^{-1}$ in some detail. Recall that we denoted the conductor of $\chi_1\chi_2^{-1}$ by $\mathfrak{l}$. Note that, by the assumptions made in Section~\ref{sec:set_uP}, we have
\begin{equation}
	\lambda_{\infty} = \chi_1\chi_2^{-1}\big\vert_{F_{\infty}^{\times}} = \prod_{\nu} \abs{\cdot}_{\nu}^{2t_{\nu}} \nonumber
\end{equation}
where $\sum_{\nu} t_{\nu} = 0$. Thus, $\chi_1\chi_2^{-1}$ corresponds to a Gr\"o\ss encharakter modulo $\mathfrak{l}$ with infinity-type $(0,(2t_{\nu})_{\nu})$.

The goal is to find the correct size of the sum
\begin{equation}
	A_L(r_1,r_2) = \sum_{\substack{\alpha\in\mathcal{P}_{\mathfrak{q}},\\ (\alpha,\n)=1}} w\left(\frac{\No(\alpha)}{L}\right) \log(\No(\alpha))\eta_{\chi_1,\chi_2,ir_1}((\alpha))\eta_{\chi_1^{-1},\chi_2^{-1},\red{-}ir_2}((\alpha)), \label{eq:def_of_AL}
\end{equation}
for a smooth function $w$ with support in the interval $[1,2]$ satisfying $0\leq w(r)\leq 1$ and
\begin{equation}
	\tilde{w}(1) = \int_{-\infty}^{\infty} w(r) dr > 0. \nonumber
\end{equation}
Here and throughout the rest of this section $\tilde{w}$ will stand for the Mellin transform of $w$.

We will pro\red{ve} the following estimate.

\begin{lemma} \label{lm:low_amp_eis}
Let $\mathfrak{q}$ be an ideal such that $(\mathfrak{q},\n)=1$ and  $\No(\mathfrak{q})\ll \log(L)^{B_2}$. Further, we assume that 
\begin{eqnarray}
	r_1,r_2 &=& T_0+O(\log(L)^{-1-\delta}) \text{ for some } T_0\gg_F 1, \nonumber \\
	\log(\sum_{\nu}\abs{t_{\nu}}+T_0)^{\frac{2}{3}+\delta}&\ll& \log(L) \ll \log(\sum_{\nu}\abs{t_{\nu}}+T_0), \nonumber \\
	\log(\No(\mathfrak{l})) &\ll& \log(L)^{1-\frac{\delta}{2}}\text{ and } \No(\n)^{\epsilon}\ll L^{\frac{1}{2}}, \nonumber
\end{eqnarray}
for $\delta>0$. Then we have
\begin{equation}
	A_L(r_1,r_2) =   \frac{2\tilde{w}(1)}{\sharp Cl_F^{\mathfrak{q} }} L(1+o_{B,\delta}(1)) \nonumber
\end{equation}
\red{as $L$ tends to infinity.}
\end{lemma}

\begin{proof}
The definition of $\eta_{\chi_1,\chi_2,s}$ implies
\begin{eqnarray}
	\xi(\mathfrak{m}) \eta_{\chi_1,\chi_2,ir}(\mathfrak{m}) &=& \eta_{\xi\chi_1,\xi\chi_2,ir}(\mathfrak{m})  \text{ and} \nonumber \\
	\sum_{r=0}^{\infty} \eta_{\xi_1,\xi_2,s}(\p^r) X^r &=& (1-\xi_1(\p)\No(\p)^{\red{-}s} X)^{-1} (1-\xi_2(\p)\No(\p)^{s} X)^{-1}.  \label{eq:prereq_for_local_fac}
\end{eqnarray}
For a Dirichlet character $\xi$ we define the Dirichlet series
\begin{equation}
	D(s,\xi) = \sum_{\substack{ (\mathfrak{a},\red{\mathfrak{q}}\n)=1}} \frac{\eta_{\chi_1,\chi_2,ir_1}(\mathfrak{a})\eta_{\xi\chi_1^{-1},\xi\chi_2^{-1},\red{-}ir_2}(\mathfrak{a})}{\No(\mathfrak{a})^s}. \nonumber
\end{equation}
\red{On the other hand, $D(s,\xi)$ has a nice factorization in well\red{-}studied $L$-functions. Indeed, \eqref{eq:prereq_for_local_fac} together with \cite[Lemma~1.6.1]{Bu96} implies
\begin{equation}
	D(s,\xi) = \frac{ L(s-ir_1+ir_2,\xi)L(s+ir_1-ir_2,\xi)L(s+ir_1+ir_2,\xi\chi_1^{-1}\chi_2)L(s-ir_1-ir_2,\xi\chi_1\chi_2^{-1})}{L(2s,\xi^2)E^{-1}(s,\xi)} \label{eq:factorisation_of_D}
\end{equation}
for some well\red{-}behaved correction factor $E$.}

The logarithmic derivative of $D$ is given by
\begin{multline}
	-\frac{D'(s,\xi)}{D(s,\xi)} = \red{-\frac{ L(s-ir_1+ir_2,\xi)}{ L'(s-ir_1+ir_2,\xi)}-\frac{L(s+ir_1-ir_2,\xi)}{L'(s+ir_1-ir_2,\xi)}-\frac{L(s+ir_1+ir_2,\xi\chi_1^{-1}\chi_2)}{L'(s+ir_1+ir_2,\xi\chi_1^{-1}\chi_2)}} \\
	\red{-\frac{L(s-ir_1-ir_2,\xi\chi_1\chi_2^{-1})}{L'(s-ir_1-ir_2,\xi\chi_1\chi_2^{-1})}+\frac{L(2s,\xi^2)}{L'(2s,\xi^2)}-\frac{E(s,\xi)}{E'(s,\xi)}}= \sum_{\substack{(\mathfrak{a},\red{\mathfrak{q}}\n)=1}} \frac{b_{\xi}(\mathfrak{a})}{\No(\mathfrak{a})^s}, \nonumber
\end{multline}
for
\begin{equation}
	b_{\xi}(\mathfrak{a}) = \begin{cases} 
		\log(\No(\p))\eta_{\chi_1,\chi_2,ir_1}(\p)\eta_{\xi\chi_1^{-1},\xi\chi_2^{-1},\red{-}ir_2}(\p) &\text{ if } \mathfrak{a} = \p, \\
		b_{\p,r} &\text{ if } \mathfrak{a} = \p^r \text{ for $r\geq 2$},\\
		0 &\text{ else.}
	\end{cases} \nonumber
\end{equation}
Here $b_{\p,r}$ are some coefficients satisfying the bound
\begin{equation}
	b_{\p,r} \red{\leq  5\log(\No(\p))}, \nonumber
\end{equation}
\red{which follows from the corresponding bounds for the coeficients of logarithmic derivatives of Hecke $L$-functions.} Using character orthogonality to detect the congruence condition in $A_L(\red{r_1,r_2})$ yields
\begin{equation}
	A_L(r_1,r_2) =  \frac{1}{\sharp Cl_F^{\mathfrak{q}}}\sum_{\xi\in\widehat{Cl}^{\mathfrak{q}}_F}\sum_{\substack{ (\mathfrak{a},\red{\mathfrak{q}}\n)=1}} w\left(\frac{\No(\mathfrak{a})}{L}\right)b_{\xi}(\mathfrak{a})+\red{O(L^{\frac{1}{2}})}. \nonumber
\end{equation}
By Mellin inversion we get
\begin{equation}
	A_L(r_1,r_2) = \frac{-1}{2\pi i \sharp Cl_F^{\mathfrak{q}}}\sum_{\xi\in \red{\widehat{Cl}}_F^{\mathfrak{q}}}  \int_{(5)}L^s\tilde{w}(s)\frac{D'(s,\xi)}{D(s,\xi)}ds + \red{O(L^{\frac{1}{2}})}. \nonumber
\end{equation}
\red{Thus, in view of \eqref{eq:factorisation_of_D}, we have to consider six contour integrals.}

By using trivial estimates we observe that
\begin{equation}
	\frac{1}{2\pi i} \int_{(5)} L^s\tilde{w}(s)\frac{L'(2s,\xi^2)}{L(2s,\xi^2)}ds \ll_{\epsilon} L^{\frac{1}{2}+\epsilon}, \nonumber
\end{equation}
for any $\xi\in \widehat{Cl}_F^{\mathfrak{q}}$. Furthermore, one can compute the correction factors and estimate
\begin{equation}
	\frac{1}{2\pi i} \int_{(5)} L^s\tilde{w}(s)\frac{E'(s,\xi)}{E(s,\xi)}ds \ll_{\sigma,\epsilon} L^{\sigma}\No(\mathfrak{q}\n)^{\epsilon}, \nonumber
\end{equation}
for $0<\sigma<1$. Thus, so far we have seen that
\begin{eqnarray}
	A_L(r_1,r_2) &=& \frac{-1}{2\pi i \sharp Cl_F^{\mathfrak{q}}}\sum_{\pm}\sum_{\xi\in Cl_F^{\mathfrak{q}}}\bigg( \int_{(5)}L^s\tilde{w}(s)\frac{L'(s\pm ir_1\pm ir_2,\xi\chi_1^{\mp 1}\chi_2^{\pm 1})}{L(s\pm ir_1\pm ir_2,\xi\chi_1^{\mp 1}\chi_2^{\pm 1})}ds \nonumber \\
	&&\qquad \qquad\qquad +  \int_{(5)}L^s\tilde{w}(s)\frac{L'(s\pm ir_1\mp ir_2,\xi)}{L(s\pm ir_1\mp ir_2,\xi)}ds\bigg) \nonumber \\
	&&\qquad\qquad\qquad\qquad + O(\No(\mathfrak{q}\n)^{\epsilon}L^{\frac{1}{2}+\epsilon}). \nonumber
\end{eqnarray}

We will now estimate the contribution coming from the $L$-function attached to the Gr\"o\ss encharakter $\xi\chi_1\chi_2\red{^{-1}}$. To do so we define
\begin{equation}
	V(t) =  \abs{t}+ \sum_{\nu}\abs{t_{\nu}} \text{ and } M(\mathfrak{q},t)=\max\left(\log(\red{\mathcal{N}(\mathfrak{ql}))},\log(V(t))^{\frac{2}{3}+\frac{\delta}{2}}\right). \nonumber
\end{equation}
\red{As on \cite[p.302]{Co90} we derive from \cite[Theorem~1]{Co90} that
\begin{multline} \label{eq:coleman_up_bound}
	L(s, \xi\chi_1\chi_2^{-1}) \ll \exp\left(c_1\frac{\log\log(V(t))^{\frac{2}{3}}}{\log(V(t))^{\frac{2}{3}}}M(\mathfrak{q},t)\right), \\ \text{ for } 1-\sigma_{0,t}= 1-c_2\frac{\log\log(V(t))^{\frac{2}{3}}}{\log(V(t))^{\frac{2}{3}}}\leq \Re(s), V(\Im(s))>V_0,
\end{multline}
\red{where $V_0$ is an absolute constant depending on $F$.} Furthermore, \cite[Theorem~2]{Co90}  implies that $L(s,\xi\chi_1\chi_2^{-1}) \neq 0$ for $1-\frac{4c_3}{M(\mathfrak{q},t)}\leq \Re(s)$ and $V(\Im(s)) \red{>} V_0$. If $t_{\nu}=0$ for all $\nu$, then there might exist a real simple exceptional zero. However, this one is automatically excluded due to the assumption $V(t)>V_0$. Furthermore, the statement of \cite[Theorem~3.11]{Ti86} carries over to our setting with the same proof. Applying it with $\phi(t) = c_1\frac{\log\log(V(t))^{\frac{2}{3}}}{\log(V(t))^{\frac{2}{3}}}M(\mathfrak{q},t)$ and $\theta(t) = c_2\frac{\log\log(V(t))^{\frac{2}{3}}}{\log(V(t))^{\frac{2}{3}}}$  reveals
\begin{equation}
	\frac{L'(s, \xi\chi_1\chi_2^{-1}) }{L(s, \xi\chi_1\chi_2^{-1}) }\ll M(\mathfrak{q},\Im(s)) \text{ for }1-\frac{c_3}{M(\mathfrak{q},t)}\leq \Re(s) \text{ and } V(\Im(s)) \red{>} V_0. \label{eq:bound_for_logder_we_need1}
\end{equation}}

We have to consider two cases. 

First, we assume that $V(0) \leq V_0$, in particular we have $T_0\asymp V(T_0)$. In this case we assume without loss of generality that 
\begin{equation}
	r_1+r_2-2V_0 = 2T_0-2V_0+O(\log(L)^{-1-\delta}) \gg T_0. \nonumber
\end{equation}
Consider the contour
\begin{eqnarray}
	\mathcal{C}_1 &=& \bigg\{ 1+\epsilon \pm it\colon t\geq r_1+r_2-2V_0  \bigg\}  \nonumber \\
		&&  \cup \bigg\{\sigma\pm \red{i( r_1+ r_2 - 2V_0)}\colon 1-\frac{c_{\red{3}}}{M(\mathfrak{q},V(T_0))} \leq \sigma \leq 1+\epsilon \bigg\} \nonumber \\
		&&  \cup \bigg\{1-\frac{c_{\red{3}}}{M(\mathfrak{q},V(T_0))}+it\colon t\in [-(r_1+r_2-2V_0),r_1+r_2-2V_0] \bigg\}. \nonumber
\end{eqnarray}
\red{Note that all poles and zeros, including the possible exceptional zero, of $L(s-ir_1-ir_2,\xi\chi_1\chi_2^{-1})$ lie to the left of the contour $\mathcal{C}_1$. Thus, we can shift the line of integration and get}
\begin{equation}
	\frac{1}{2\pi i} \int_{(5)} L^s\tilde{w}(s)\frac{L'(s-ir_1-ir_2,\xi\chi_1\chi_2^{-1})}{L(s-ir_1-ir_2,\xi\chi_1\chi_2^{-1})}ds = \frac{1}{2\pi i} \int_{\mathcal{C}_1} L^s\tilde{w}(s)\frac{L'(s-ir_1-ir_2,\xi\chi_1\chi_2^{-1})}{L(s-ir_1-ir_2,\xi\chi_1\chi_2^{-1})}ds. \nonumber
\end{equation}
We continue to estimate the integral along each piece of the contour $\mathcal{C}_1$.

\red{Exploiting the rapid decay of $\tilde{w}$ we estimate trivially to get}
\begin{equation}
	\int_{\abs{t} \geq r_1+r_2-2V_0} L^{1+\epsilon+it}\tilde{w}(1+\epsilon+it)\frac{L'(s-ir_1-ir_2,\xi\chi_1\chi_2^{-1})}{L(s-ir_1-ir_2,\xi\chi_1\chi_2^{-1})} dt \ll_{\epsilon,F, A} L^{1+\epsilon} V(T_0)^{-A}. \nonumber
\end{equation}

If $\log(\No(\mathfrak{ql}))\leq \log(V(T_0))^{\frac{2}{3}+\frac{\delta}{2}}$, we see that $M(\mathfrak{q},T_0) = \log(V(T_0))^{\frac{2}{3}+\frac{\delta}{2}}$. In this case the estimation of the integral on the remaining parts of the integral is \red{straightforward} and left to the reader. Therefore, we assume 
\begin{equation}
	M(\mathfrak{q},T_0) = \log(\No(\mathfrak{ql})). \nonumber
\end{equation}
\red{In this situation \eqref{eq:bound_for_logder_we_need1} reduces to}
\begin{equation}
	\frac{L'(s\red{-ir_1-ir_2}, \xi\chi_1\chi_2^{-1})}{L(s\red{-ir_1-ir_2},\xi\chi_1\chi_2^{-1})} \ll \log(\No(\mathfrak{ql})) \text{ for } s\in \mathcal{C}_1, \abs{\Im(s)} \leq r_1+r_{\red{2}}-2V_0. \nonumber
\end{equation}
With this at hand we estimate
\begin{eqnarray}
	&& \int_{1-\frac{c_{\red{3}}}{\log(\No(\mathfrak{ql}))}}^{1+\epsilon} L^{\sigma\pm ir_1\pm ir_2\mp i2V_0} \tilde{w}(\sigma\pm ir_1\pm ir_2 \mp i2V_0) \nonumber \\
	 &&\qquad\qquad\qquad\cdot\frac{L'(\sigma-ir_1-ir_2\pm ir_1\pm ir_2\mp i2V_0, \xi\chi_1\chi_2^{-1})}{L(\sigma-ir_1-ir_2\pm ir_1\pm ir_2\mp i2V_0,\xi\chi_1\chi_2^{-1})} d\sigma \ll_A L^{1+\epsilon} V(T_0)^{-A}. \nonumber
\end{eqnarray}
The last piece can be bounded as follows:
\begin{align}
	&\int_{\abs{t} \red{\leq} r_1+r_2-2V_0} L^{1\red{-\frac{c_3}{\log(\No(\mathfrak{ql}))}}+it}\tilde{w}(1\red{-\frac{c_3}{\log(\No(\mathfrak{ql}))}}+it)\frac{L'(\red{1-\frac{c_3}{\log(\No(\mathfrak{ql}))}+it}-ir_1-ir_2,\xi\chi_1\chi_2^{-1})}{L(\red{1-\frac{c_3}{\log(\No(\mathfrak{ql}))}+it}-ir_1-ir_2,\xi\chi_1\chi_2^{-1})} dt \nonumber \\
	& \ll L \exp\left(-c'\frac{\log(L)}{\log(\No(\mathfrak{lq}))} \right) \log(\No(\mathfrak{ql}))  \ll L\log(L) \exp(-c''\log(L)^{\delta}) \nonumber\\
	& \ll L\log(L)^{-A}. \nonumber 
\end{align}

Second, we consider the case $V(0) \red{>} V_0$. In this case there must be at least one $t_{\nu}\neq 0$ so that there can not be a pole at 1 or an exceptional zero for the character $\xi\chi_1\chi_2^{-1}$. \red{Furthermore, the bounds \eqref{eq:coleman_up_bound} and \eqref{eq:bound_for_logder_we_need1} as well as the corresponding zero-free region} hold for all $t$. With this in mind we define the contour
\begin{eqnarray}
	\mathcal{C}_2 &=& \bigg\{ 1+\epsilon\pm it\colon t\geq 100V(T_0)  \bigg\}  \nonumber \\
		&&  \cup \bigg\{\sigma\pm i100V(T_0)\colon 1-\frac{c_{\red{3}}}{M(\mathfrak{q},T_0)} \leq \sigma \leq 1+\epsilon \bigg\} \nonumber \\
		&&  \cup \bigg\{1-\frac{c_{\red{3}}}{M(\mathfrak{q},T_0)}+it\colon t\in [-100V(T_0),100V(T_0)] \bigg\}. \nonumber
\end{eqnarray}
Since there are no poles or zeros of $L(s\red{-ir_1-ir_2},\xi\chi_1\chi_2^{-1})$ to the right of $\mathcal{C}_2$ we obtain
\begin{equation}
	\frac{-1}{2\pi i} \int_{(5)} L^s\tilde{w}(s)\frac{L'(s-ir_1-ir_2,\xi\chi_1\chi_2^{-1})}{L(s-ir_1-ir_2,\xi\chi_1\chi_2^{-1})}ds = \frac{-1}{2\pi i} \int_{\mathcal{C}_2} L^s\tilde{w}(s)\frac{L'(s-ir_1-ir_2,\xi\chi_1\chi_2^{-1})}{L(s-ir_1-ir_2,\xi\chi_1\chi_2^{-1})}ds. \nonumber
\end{equation}
The remaining contour integral is estimated similarly to the previous case. \red{Further, the contribution of $L(s+ir_1+ir_2,\xi\chi_1^{-1}\chi_2)$ can be estimated analogously.}

By choosing $A$ sufficiently large, and exploiting the assumption $\log(L)\ll\log(V(T_0))$ we conclude that
\begin{eqnarray}
	A_L(r_1,r_2) &=& \frac{-1}{2\pi i \sharp Cl_F^{\mathfrak{q}}}\sum_{\pm}\sum_{\xi\in Cl_F^{\mathfrak{q}}}  \int_{(5)}L^s\tilde{w}(s)\frac{L'(s\pm ir_1\mp ir_2,\xi)}{L(s\pm ir_1\mp ir_2,\xi)}ds +O(L\log(L)^{-A'}). \nonumber
\end{eqnarray}

To deal with the final contribution we recall \cite[Satz~2.1]{Hi80}. It says that for $s=\sigma+it$ satisfying 
\begin{equation}
	\sigma\geq 1-c_1\left(\frac{\log\log\abs{t}}{\log\abs{t}}\right)^{\frac{2}{3}} \text{ and }\abs{t}\geq 3 \nonumber
\end{equation} 
we have
\begin{equation}
	L(s,\xi) \ll \exp\left(c_1\left(\frac{\log\log\abs{t}}{\log\abs{t}}\right)^{\frac{2}{3}}\log\No(\mathfrak{q})+c_2\log\log\abs{t}\right).\label{eq:bound_for_L_incrit}
\end{equation}
Furthermore, by \cite[Satz~1.1]{Hi80}, there is at most one real simple zero $\beta$ in the region
\begin{equation}
	\sigma\geq 1-\frac{4c_3}{M\red{'}(\mathfrak{q},t)}, \text{ for } M'(\mathfrak{q},t) = \max\bigg(\log(\No(\mathfrak{q})\red{)},(\log(\abs{t}+3))^{\frac{2}{3}}(\log\log(\abs{t}+3))^{\frac{1}{3}}\bigg). \nonumber
\end{equation}
Even more, there is at most one character modulo $\mathfrak{q}$ featuring an exceptional zero which we will denote by $\xi_e$. Let us remark that since we are dealing with an arbitrary number field the case $\xi_0=\xi_e$ is not excluded. Here and in the following $\xi_0$ denotes the principal character.

Guided by this zero-free region we define the contour
\begin{eqnarray}
	\mathcal{C}_{\red{3}} &=& \bigg\{ 1\pm it\colon t\geq100V(T_0)  \bigg\} \cup \bigg\{\sigma\red{\pm}100iV(T_0)\colon 1-\frac{c'}{\log(V(T_0))^{\frac{2}{3}+\frac{\delta}{2}}} \leq \sigma \leq 1\bigg\} \nonumber \\
	&& \cup \bigg\{1-\frac{c'}{\log(V(T_0))^{\frac{2}{3}+\frac{\delta}{2}}}+it\colon t\in [-100V(T_0),100V(T_0)] \bigg\}. \nonumber
\end{eqnarray}
Our assumption on the size of $\No(\mathfrak{q})$ implies 
\begin{equation}
	M'(\mathfrak{q},t) \ll_{\delta} \log(V(T_0))^{\frac{2}{3}+\frac{\delta}{2}} \text{ for } t\in [-100V(T_0),100V(T_0)]. \nonumber
\end{equation}
We conclude that for suitabl\red{y} taken $c'$ our contour $\mathcal{C}_{\red{3}}$ is contained in the extended zero-free region described above. Let us also show that the exceptional zero (if it exists) must be on the left of $\mathcal{C}_{\red{3}}$. By \cite{Fo63} the exceptional zero satisfies
\begin{equation}
	\red{1-}\beta \gg_{\epsilon,F} \No(\mathfrak{q})^{-\epsilon}. \nonumber
\end{equation}
Further, our assumptions on the size of $L$ and $\No(\mathfrak{q})$ imply
\begin{equation}
	\red{1-}\beta \geq c(\epsilon,B_1,B_2) \log(V(T_0))^{-\epsilon B_2 B_1}. \nonumber
\end{equation}
By choosing $\epsilon$ small enough this \red{yields}
\begin{equation}
	\red{1-}\beta> c(B_2,B_1) \log(V(T_0))^{-\frac{1}{2}}> \frac{8c'}{\log(V(T_0))^{\frac{2}{3}+\frac{\delta}{2}}}, \nonumber
\end{equation}
after making $c'$ smaller if necessary. In particular, we have seen that $\red{\beta}<1-\frac{c'}{\log(V(T_0))^{\frac{2}{3}+\frac{\delta}{2}}}$.

The \red{residue} theorem implies
\begin{eqnarray}
	\frac{1}{2\pi i} \int_{(5)} L^s\tilde{w}(s)\frac{L'(s+ir_1-ir_2,\xi)}{L(s+ir_1-ir_2,\xi)}ds &=& -\delta_{\xi=\xi_0}\tilde{w}(1-ir_1+ir_2)L^{1-ir_1+ir_2} \nonumber \\
	&& +\frac{1}{2\pi i}\int_{\mathcal{C}_{\red{3}}} L^{s}\tilde{w}(s)\frac{L'(s+ir_1-ir_2,\xi)}{L(s+ir_1-ir_2,\xi)}ds. \nonumber 
\end{eqnarray}

Writing $-ir_1+ir_2 = i\eta = O(\log(L)^{-1-\delta})$ and \red{continuity of the exponential function} shows that the pole contributes
\begin{equation}
	\tilde{w}(1-ir_1+ir_2)L^{1-ir_1+ir_2} = \tilde{w}(1)L(1+o(1)). \nonumber
\end{equation}
This gives us the expected main term.

Finally, we estimate the integral along the contour $\mathcal{C}_{\red{3}}$. To do so we need bounds for $\frac{L'(s,\xi)}{L(s,\xi)}$ on $\mathcal{C}_{\red{3}}$. \red{Again we use \cite[Theorem~3.11]{Ti86}, which easily generalizes to our situation, together with \eqref{eq:bound_for_L_incrit} to obtain
\begin{equation}
	\frac{L'(s,\xi)}{L(s,\xi)} \ll M'(\mathfrak{q},t) \text{ for } \sigma\geq 1-\frac{c_3}{M\red{'}(\mathfrak{q},t)}\text{ and }\abs{t}\geq 3. \nonumber
\end{equation}
From this we deduce the weaker bounds
\begin{equation}
	\frac{L'(s+ir_1-ir_2,\xi)}{L(s+ir_1-ir_2,\xi)} \ll \begin{cases} 
		\log(V(T_0)) &\text{ if } \abs{t} \leq 100V(T_0), \\
		\log(\abs{t})&\text{ else,}	
	\end{cases} \nonumber
\end{equation}	
for $s\in \mathcal{C}$. Note that these bounds are very rough, but due to the region on which they are valid still highly non-trivial.} Using these bounds together with the rapid decay of $\tilde{w}$ yields
\begin{equation}
	\frac{1}{2\pi i}\int_{\mathcal{C}_{\red{3}}} L^{s}\tilde{w}(s)\frac{L'(s+ir_1-ir_2,\xi)}{L(s+ir_1-ir_2,\xi)}ds = O(L\log(V(T_0))^{-A}). \nonumber
\end{equation}
\noindent\red{The same argument deals with the contribution of $L(s-ir_1+ir_2,\xi)$.}

By patching all the pieces together using the technical assumptions on $\No(\mathfrak{q})$ and $\log(L)$ we obtain the required asymptotic formula for $A_L(r_1,r_2)$.
\end{proof}

\section{The amplification of Eisenstein series}

In this section we construct an amplifier for the Eisenstein series. The argument is similar to the one in \cite[Section~4.1]{As17_1} and we will leave out some details. Throughout this section we fix a parameter $T_0\redd{\geq }1$. Lemma~\ref{lm:Eisenstein_reduction} implies that, after possibly replacing $E$ by $E^{\mathfrak{L}}$ for some $\mathfrak{L}\mid \n$, we can restrict ourselves to $g= a(\theta_i)g'n(x)a(y)$ with $g' = kh_{\n}\in \mathcal{J}_{\n}$ and $n(x)a(y) \in \mathcal{F}_{\n_2}$. Define $E'(s,g) = E(s,gh_{\n})$. 

We choose an ideal $\mathfrak{q}$ such that every quadratic residue modulo $\mathfrak{q}$ is indeed a square in $\mathcal{O}_F^{\times}$.  By \cite[Lemma~5.1]{As17_2} there exists such an ideal which further satisfies $(\mathfrak{q},\n)=1$ as well as $\No(\mathfrak{q})\ll \log(\No(\n))^{B_2}$ for some large $B_2\geq 0$.

Even if the Eisenstein series $E'$ itself is not an element of $L^2(X)$, we will use the spectral expansion to obtain average bounds. To do so we proceed by choosing a test function $f$, very similar to the ones used in \cite{As17_2, BHMM16, Sa15} for cusp forms, \red{and associating the integral operator
\begin{equation}
	R(f)w = \int_{Z(\A_F)\setminus G(\A_F)} f(g)[\pi(s)(g)w]dg. \nonumber
\end{equation}
The desired bounds will follow by analyzing both the geometric and the spectral expansion of this operator. Exploiting the product structure of $\A_F$ we construct $f$ place by place as follows.}

At the archimedean places $\nu$ we choose 
\begin{equation}
	f_{\nu}(g_{\nu}) = k_{\nu}(u_{\nu}(g_{\nu}.i_{\nu},i_{\nu})), \nonumber
\end{equation}
for $k_{\nu}$ as in \cite[Lemma~\redd{10}]{BHMM16} \red{with spectral parameter $t_{\nu}+T_0$ of $\pi_{\nu}(iT_0)$}. Note that by uniqueness of the spherical vector  \red{in $\mathcal{B}(\chi_{1,\nu}(s),\chi_{2,\nu}(-s))$} we have
\begin{equation}
	R(f_{\nu})v_{\nu}^{\circ}(it) = c_{\nu}(\pi_{\nu}(it))\red{v_{\nu}^{\circ}(it)}. \nonumber
\end{equation}
\red{The eigenvalue $c_{\nu}(\pi_{\nu}(\cdot))$ can be obtained from $k_{\nu}$ via the inverse Selberg/Harish-Chandra transform. Thus, by \cite[(9.8)]{BHMM16}} we   get $c_{\nu}(\pi_{\nu}(\red{i}T_0))\gg 1$. By continuity \red{of $t\redd{\mapsto} c_{\nu}(\pi(it))$} there is $\eta>0$ such that
\begin{equation}
	c_{\nu}(\pi_{\nu}(\red{i}t))\gg 1 \text{ for all } t\in [T_0-\eta,T_0+\eta]. \nonumber
\end{equation}
\redd{Even more, according to  \cite[(9.5)]{BHMM16}, the constant $\eta$ can be chosen independently of $T_0$.}

If $\p\mid\mathfrak{q}$, we take
\begin{equation}
	f_{\p}(g_{\p}) = \begin{cases} 
		\text{vol}(Z(\op_{\p})\setminus \tilde{K}_{0,\p}(1))^{-1}\varpi_{\pi}^{-1}(z) &\text{ if } g_{\p}=zk\in Z(F_{\p})\tilde{K}_{0,\p}(1),\\
		0 &\text {else.}
	\end{cases} \nonumber
\end{equation} 
\red{Note that as in \cite[p. 24]{As17_2} we have} $\abs{f_{\p}}\ll q_{\p}^{2+\epsilon}$ and that
\begin{equation}
	R(f_{\p})v_{\p}^{\circ}(it) = v_{\p}^{\circ}(it). \nonumber
\end{equation}
\red{Since the support of $f_{\p}$ is contained in $Z(F_{\p})K_{\p}$ this is true for all $t$.}

For $\p\vert \n$ we choose
\begin{equation}
	f_{\p}(g_{\p}) = \abs{\det(g_{\p})}^{ia_{\p}} \Phi_{\pi_{\p}'(0)}'(g_{\p}) \nonumber
\end{equation}
as defined on \cite[Section~2.6]{Sa15}. \red{It follows from \cite[Proposition~2.13]{Sa15} that}
\begin{equation}
	R(f_{\p})v_{\p}^{\circ}(it) = \delta_{\pi_{\p}'(0)}v_{\p}^{\circ}(it), \text{ for all } t,
\end{equation}
\red{where $\delta_{\pi_{\p}'(0)}$ constant independent of $t$ satisfying
\begin{equation}
	 \delta_{\pi_{\p}'(0)} \gg q_{\p}^{-n_{1,\p}-m_{1,\p}}. \nonumber
\end{equation} }
\red{By construction $f_{\p}$ satisfies the two important properties}
\begin{eqnarray}
	\abs{f_{\p}(g_{\p})} &\leq& 1 \text{ for all } g_{\p}, \nonumber\\
	\text{supp}(f_{\p}) &\subset& \begin{cases}
		Z(F_{\p})K_{\p} &\text{ if $n_{\p}$ is even, } \nonumber\\
		Z(F_{\p})K_{\p}^{\circ}(1) &\text{ else.}
	\end{cases} 
\end{eqnarray} 

The remaining places are treated at once. At this stage we diverge from the treatment in \cite{As17_2}. Indeed, we will exploit the explicit form of the Hecke eigenvalues of Eisenstein series to construct a shorter amplifier. \red{For $(\mathfrak{a},\n)=1$ let 
\begin{equation}
	\kappa_{\mathfrak{a}}\colon \prod_{\p\nmid \n}G(F_{\p}) \to \C \nonumber
\end{equation}
be the integral kernel for the Hecke operator $T(\mathfrak{a})$. Thus, according to Remark~\ref{rem:rel_forEisenstein}, we have
}
\begin{equation}
	R(\kappa_{\mathfrak{a}})v_{\text{ur}}^{\circ}(it) = \sqrt{\No(\mathfrak{a})}\eta_{\chi_1,\chi_2,it}(\mathfrak{a})v_{\text{ur}}^{\circ}(it), \label{eq:rel_to+hecke_important}
\end{equation}
\red{at least for every principal ideal $\mathfrak{a}$ \redd{coprime} to $\n$. Further we fix a large parameter $L$, a test function $w$ as in Section~\ref{sec:averages_div_f}, and recall the definition of the set $\mathcal{P}(L)\subset \mathcal{O}_F$ from \cite[p. 24]{As17_2}. With this at hand we construct the test function}
\begin{equation} 
	f_{\text{ur}} = \left( \sum_{\alpha\in \mathcal{P}(L)} \frac{\kappa_{(\alpha)}}{\sqrt{\No(\alpha)}}x_{\alpha} \right) \cdot \left( \sum_{\alpha\in \mathcal{P}(L)} \frac{\kappa_{(\alpha)}}{\sqrt{\No(\alpha)}}x_{\alpha} \right)^{*}, \nonumber
\end{equation}
\redd{where 
\begin{equation}
	x_{\alpha} = w\left(\frac{\No(\alpha)}{L}\right)\log(\No(\alpha)) \eta_{\chi_2^{-1},\chi_1^{-1},iT_0}((\alpha)),\quad \text{ for } \alpha\in \mathcal{P}(L).\nonumber
\end{equation}
Note that the integral operator with  kernel $f_{\text{ur}}$ is positive by construction.}

\red{In view of \eqref{eq:rel_to+hecke_important} and $\eta_{\chi_2^{-1},\chi_1^{-1},iT_0}((\alpha)) = \eta_{\chi_1^{-1},\chi_2^{-1},-iT_0}((\alpha))$ we have}
\begin{equation}
	R(f_{ur}) v_{\text{ur}}^{\circ}(ir) = \abs{A_L(r,T_0)}^2v_{\text{ur}}^{\circ}(ir), \nonumber
\end{equation}
where $A_{L}$ was defined in \eqref{eq:def_of_AL}. Opening the square yields
\begin{equation}
	f_{\text{ur}} = \sum_{\alpha\in \mathcal{O}_F} y_{\alpha} \frac{\kappa_{\alpha}}{\sqrt{\No(\alpha)}}, \nonumber
\end{equation}
for
\begin{equation}
	y_{\alpha} = \begin{cases}
		\sum_{\alpha' \in \mathcal{P}(L)} \abs{x_{\alpha'}}^2 \omega^{-1}_{\pi_{(\alpha')}}(\varpi_{(\alpha')})+ \abs{x_{\alpha'^2}}^2 \omega^{-1}_{\pi_{(\alpha')}}(\varpi_{(\alpha')}^2) &\text{ if }\alpha=1, \\ 
		x_{\alpha_1}\overline{x_{\alpha_2}} &\text{ if } \alpha=\alpha_1\alpha_2 \text{ for } \alpha_1,\alpha_2\in \mathcal{P}(L), \\ 
		 0 &\text{ else}.
	\end{cases} \nonumber
\end{equation}
It is clear that
\begin{equation}
	y_{\alpha} \ll L^{\epsilon}\cdot \begin{cases}
		L &\text{ if $\alpha= 1$,}\\
		1 &\text{if $\alpha = \alpha_1\alpha_2$ for $\alpha_1,\alpha_2 \in \mathcal{P}(L)$,}\\
		0&\text{ else.}  
	\end{cases} \nonumber
\end{equation}
\red{Comparing this to \cite[(4.4)]{As17_2} or \cite[(9.16)]{BHMM16} is showing once more the advantage of this shorter amplifier.}

Globally we define the test function
\begin{equation}
	f=f_{\text{ur}}\prod_{\nu} f_{\nu}\prod_{\p\mid\mathfrak{q}\n} f_{\p}. \nonumber
\end{equation}
By exploiting the positivity of $f$ \redd{and recalling the definitions in \eqref{eq:def_of_several_ideals}} one obtains
\begin{equation}
	\frac{1}{\No(\n_1\mathfrak{m}_1)} \int_{T_0-\eta}^{T_0+\eta}\abs{A_L(t,T_0)}^2\abs{E'(it,g)}^2 dt \ll \sum_{\gamma\in Z(F)\setminus G(F)} f(g^{-1}\gamma g) \red{=\colon  K_f(g,g).} \label{eq:spec_side_for_eis}
\end{equation}
\red{This is parallel to \cite[(4.4)]{As17_2} and \cite[Section~9.4]{BHMM16}.} We continue by estimating the geometric side of this pre-trace inequality.

 We start by expanding
\begin{eqnarray}
	K_f(g,g) &\ll& \sum_{0\neq \alpha\in \mathcal{O}_F} \frac{\abs{y_{\alpha}}}{\sqrt{\No(\alpha)}} \sum_{\gamma\in Z(F)\setminus G(F)} \abs{\kappa_{(\alpha)}\prod_{\p\mid \n\mathfrak{q}}f_{\p}}(g'^{-1}a(\theta_i^{-1})\gamma a(\theta_i)g')\abs{k(u(\gamma .P,P))}. \nonumber
\end{eqnarray}
Since we are using the same test function as in \cite{As17_2} we can exploit the \redd{support properties} analogously. \red{By following the first steps in the proof of \cite[Proposition~4.1]{As17_2}} we arrive at
\begin{equation}
	K_f(g,g) \ll \No(\mathfrak{q})^{2+\epsilon} L^{\epsilon} \sum_{0\neq \alpha\in \mathcal{O}_F} \frac{\redd{\abs{y_{\alpha}}}}{\sqrt{\No(\alpha)}} \sum_{\gamma\in \Gamma(i,\alpha)} \abs{k(u(\gamma .P,P))}, \nonumber
\end{equation}
\red{where $\Gamma(i,\alpha)$ was defined in \cite[(9.21)]{BHMM16} (or on \cite[p.28]{As17_2}).} We follow the argument from \cite[p. 26]{BHMM16}. In particular, to each $\gamma\in \Gamma(i,m)$ and each $\nu$ we associate the smallest $k_{\nu}(\gamma)\in \Z$ such that 
\begin{equation}
	\max(T_{\nu}^{-2},u_{\nu}(\gamma_{\nu}.P_{\nu},P_{\nu}) )\leq 2^{k_{\nu}(\gamma)}. \nonumber
\end{equation}
Recall that $T_{\nu} = \max(\frac{1}{2}, \abs{t_{\nu}+T_0})$. Further, define $\delta_{\nu}(\gamma)=2^{k_{\nu}(\gamma)}$. As in \cite[\redd{(9.22)}]{BHMM16} we observe that
\begin{equation}
	k(u(\gamma.P,P)) \ll \abs{T}_{\infty}^{\frac{1}{2}}\abs{\delta(\gamma)}_{\infty}^{-\frac{1}{4}}. \nonumber
\end{equation} 
In order to group the $\gamma$'s together appropriately we define the sets
\begin{equation}
	\mathfrak{M}(L,j,\delta)= \bigg\{ \gamma\in \Gamma(i,\alpha_1^j\alpha_2^j)\colon u_{\nu}(\gamma_{\nu}.P_{\nu},P_{\nu})\leq \delta_{\nu}\text{ for all }\nu \text{ and } \alpha_1,\alpha_2\in \mathcal{P}(L)  \bigg\} \nonumber
\end{equation}
 and put 
\begin{equation}
	M(L,j,\delta) = \sharp \mathfrak{M}(L,j,\delta). \nonumber
\end{equation}
Using the support of $k_{\nu}$ and the shape of $\abs{y_{\alpha}}$ we conclude
\begin{equation}
	K_f(g,g) \ll \No(\mathfrak{q})^{2+\epsilon}L^{\epsilon}\sum_{\substack{\du{k}\in \Z^{\sharp\{\nu\}},\\ T_{\nu}^{-2} \leq \delta_{\nu} = 2^{k_{\nu}}\leq 4}} \frac{\abs{T}_{\infty}^{\frac{1}{2}}}{\abs{\delta}_{\infty}^{\frac{1}{4}}} \left(\underbrace{LM(L,0,\delta)}_{\text{ contribution from $\alpha=1$.}}+L^{-1}M(L,1,\delta)  \right).\nonumber
\end{equation}
\red{In contrast to \cite[(4.9)]{As17_2} and \cite[\redd{(9.24)}]{BHMM16} we do not have to deal with the contribution of $M(L,2,\delta)$. This is due to the modified amplifier.}

Inserting the  counting results from \cite[\redd{Section~11}]{BHMM16} summarized in the first list and using
\begin{equation}
	\abs{T}_{\R}^{-2} \leq \abs{\delta}_{\R} \redd{\ll_F}1 \text{ and }\abs{T}_{\C}^{-2} \leq \abs{\delta}_{\C} \redd{\ll_F} 1 \nonumber
\end{equation}
yields
\begin{equation}
	K_f(g,g) \ll \No(\mathfrak{q})^{2+\epsilon}(\abs{T}_{\infty} L)^{\epsilon} \left( \abs{T}_{\infty}L+\abs{T}_{\infty}^{\frac{1}{2}}\abs{y}_{\infty} L^2+ \frac{\abs{T}_{\infty}^{\frac{1}{2}}\abs{T}_{\C}^{\frac{1}{2}}}{\No(\n_2)^{\frac{1}{4}}}L^{\frac{3}{2}} +\frac{\abs{T}_{\infty}^{\frac{1}{2}}}{\No(\n_2)}L^3 \right). \nonumber
\end{equation}

We choose
\begin{equation}
	L=(\No(\n)t_0)^{\epsilon}\abs{T}_{\infty}^{\frac{1}{4}} \No(\n_2)^{\frac{1}{2}} \nonumber
\end{equation}
for $t_0 = \sum_{\nu} \abs{t_{\nu}}+\abs{T_0}$. Note that $1\redd{\leq} t_0$. \redd{Furthermore, since we are assuming $\chi_1\chi_2^{-1}\vert_{F_{\infty}^{+}}=1$, we have $\sum_{\nu}t_{\nu}=0$ and therefore $t_0\ll \abs{T}_{\infty}$.} This leads to
\begin{equation}
	K_f(g,g) \ll \No(\mathfrak{q})^{2+\epsilon}(\abs{T}_{\infty}\No(\n))^{\epsilon} \left( \abs{T}_{\infty}^{\redd{\frac{5}{4}}} \No(\n_2)^{\frac{1}{2}}+\abs{T}_{\infty}^{\frac{7}{8}}\abs{T}_{\C}^{\frac{1}{2}}\No(\n_2)^{\frac{1}{2}}+\abs{T}_{\infty}\No(\n_2)\abs{y}_{\infty} \right). \label{eq:geo_bound_for_eis}
\end{equation}

In order to use Lemma~\ref{lm:low_amp_eis} we must ensure that the appropriate growth conditions are satisfied. Let us make the following assumptions
\begin{eqnarray}
	t &=& T_0 + O(\log(\abs{T}_{\infty})^{-1-\delta}), \label{eq:assumptions_b}\\
	\No(\mathfrak{q}) &\ll&  \log(\No(\n))^{C_1}, \\
	\No(\n) &\ll& \abs{T}_{\infty}^{C_2}, \\
	\log(\No(\mathfrak{l})) &\ll& \log(\abs{T}_{\infty})^{1-\frac{\delta}{2}}. \label{eq:assumptions_e}
\end{eqnarray}

\begin{lemma} \label{lm:verifying_some_conditions}
Assuming \eqref{eq:assumptions_b}-\eqref{eq:assumptions_e} and $T_0\gg 1$ one obtains
\begin{equation}
	\abs{A_L(t,T_0)}^2 \gg_{C_2,\epsilon} \red{\No(\n)^{-\epsilon}}L^2. \nonumber
\end{equation}
\end{lemma}
\begin{proof}
The statement follows directly from Lemma~\ref{lm:low_amp_eis}. Thus, we need to verify that the technical assumptions of this lemma are satisfied. 

First, we note that obviously $\No(\n)^{\epsilon} \ll L^{\frac{1}{2}}$ and $\No(\mathfrak{q})\ll \log(L)^{B_2}$ for some big enough $B_2$. These estimates are both due to the factor $\No(\n)^{\epsilon}$ in $L$.

Second, observe that $t_0 = V(T_0)$. One can check that
\begin{equation}
	\log(\abs{T}_{\infty}) = \sum_{\nu} [F_{\nu}:\R]\log(\abs{t_{\nu}+T_0}) \ll \log(t_0), \nonumber
\end{equation}
for $T_0\redd{\geq} 1$. We compute
\begin{eqnarray}
	\log(L) = \epsilon\log(\No(\n)t_0)+\frac{1}{2}\log(\No(\n_2)) + \frac{1}{4}\log(\abs{T}_{\infty}) \ll_{C_2} \log(\abs{T}_{\infty}) \ll \log(t_0). \nonumber
\end{eqnarray}
On the other hand\redd{, because $t_0\geq 1$,} we have
\begin{equation}
	\log(t_0)^{\frac{2}{3}+\delta} \ll_{\epsilon} \log(t_0^{\epsilon}) \leq \log(L). \nonumber
\end{equation}

Third, we see that 
\begin{equation}
	t-T_0 \ll \log(\abs{T}_{\infty})^{-1-\delta}\ll \log(L^{\frac{1}{C_2}})^{-1-\delta}\ll_{C_2} \log(L)^{-1-\delta}. \nonumber
\end{equation}

Finally, we have to confirm that $\log(\No(\mathfrak{l}))\ll\log(L)^{1-\frac{\delta}{2}}$. This follows from
\begin{equation}
	\log(L)^{1-\frac{\delta}{2}} \redd{\gg} \log(\abs{T}_{\infty})^{1-\frac{\delta}{2}} \gg \log(\No(\mathfrak{l})). \nonumber
\end{equation}
\end{proof}

Combining \eqref{eq:spec_side_for_eis}, \eqref{eq:geo_bound_for_eis} and Lemma~\ref{lm:verifying_some_conditions} yields the following result.

\begin{prop} \label{pr:after_amp_for_eis}
If $T_0$, $\n$ and $\mathfrak{l}$ satisfy \eqref{eq:assumptions_b}-\eqref{eq:assumptions_e}, then we have
\begin{eqnarray}
	&&\int_{T_0-c\log(\abs{T}_{\infty})^{-1-\delta}}^{T_0+c\log(\abs{T}_{\infty})^{-1-\delta}} \abs{E'(it,g)}^2dt \nonumber \\
	&&\qquad\qquad \ll (\No(\n)\abs{T}_{\infty})^{\epsilon} \No(\n_0\mathfrak{m}_1)\bigg( \abs{T}_{\infty}^{\frac{3}{4}}\No(\n_2)^{\frac{1}{2}}+\abs{T}_{\R}^{\frac{3}{8}}\abs{T}_{\C}^{\frac{7}{8}}\No(\n_2)^{\frac{1}{2}}+\abs{T}_{\infty}^{\frac{1}{2}}\No(\n_2)\abs{y}_{\infty} \bigg), \nonumber
\end{eqnarray}
for $g= a(\theta_i)g'n(x)a(y)$ with $g' = kh_{\n}\in \mathcal{J}_{\n}$ and $n(x)a(y) \in \mathcal{F}_{\n_2}$.
\end{prop}

Using the same amplifier as in \cite{As17_2} and dropping all $T$-dependence  leads
\begin{equation}
	\int_{T_0-\delta}^{T_0+\delta}\abs{E'(it,g)}^2dt \ll_{T_0} \No(\n)^{\epsilon}\No(\n_0\mathfrak{m}_1)\bigg(\No(\n_2)^{\frac{2}{3}}+\No(\n_2)\abs{y}_{\infty} \bigg) \nonumber 
\end{equation}
for some small $\delta>0$ depending on $T_0$. This holds without any restriction on $\mathfrak{l}$. However, we loose $\No(\n_2)^{\frac{1}{6}}$ compared to the bound above.

\section{An individual bound via an average bound}

In the previous section we used the amplification method to prove average bounds for $E(s,g)$. Therefore, we need to convert these average bounds into \redd{pointwise} bounds. It turns out that this can be achieved using the functional equation satisfied by Eisenstein series. In this section we will adapt the argument from \cite[Section~4]{Yo15} to our situation and derive the corresponding result.

\begin{lemma} \label{lm:ind_via_av_bounds}
For $g=a(\theta_i)g'n(x)a(y)$ with $g'\in\mathcal{J}_{\n}$ and $n(x) a(y)\in \mathcal{F}_{\n_2}$ we have
\begin{eqnarray}
	 F(iT_0,g)^2 &\ll& \No(\n)^{\epsilon} \abs{T}_{\infty}^{\epsilon}\int_{\abs{v}\leq 4\log(\abs{T}_{\infty}) } \bigg[\abs{F(i(T_0+v\red{)},g)}^2 + \abs{\hat{F}(i(-T_0+v),g)}^2 \bigg]dv  \nonumber \\
	&&\qquad + \No(\n)^{\epsilon} \abs{T}_{\infty}^{\epsilon}\No(\n_0)^2 + \mathcal{E}(y),\nonumber
\end{eqnarray}
where
\begin{equation}
	\mathcal{E}(y) =\exp(-\log(\abs{T}_{\infty})^2) \No(\n)^{\epsilon} \Big( \No(\imath)\abs{y}_{\infty}^{-1}\red{+\No(\imath)^{\frac{1}{2}}\No(\n_0(g))\abs{y}_{\infty}^{-\frac{1}{2}}}+\No(\n_0(g)) \Big). \nonumber
\end{equation}
If we assume that $\log(\No(\n)) \ll \log(\abs{T}_{\infty})$, then we even have $\mathcal{E}(y) \ll_{F,\epsilon} 1$.
\end{lemma}

\begin{proof}
Define the integral
\begin{equation}
	I=\frac{1}{2\pi i} \int_{(\frac{1}{2}+\red{\frac{\epsilon}{8}})} F(s+w,g)^2\frac{\exp(w^2)}{w}dw. \nonumber
\end{equation}
\red{If $0\leq \Re(s)\leq \frac{\epsilon}{8}$, Lemma~\ref{lm:Bound_way_to_bad_in_n_0} yields} the trivial bound $I = O_{F,\epsilon}(\No(\n)^{\epsilon}\No(\n_0)^2)$. Let 
\begin{equation}
	0<\delta = \min\left(\red{\frac{\epsilon}{8}}, \frac{c}{\log(\No(\n)\abs{T}_{\infty})}\right)<1. \nonumber
\end{equation}
The \red{residue} theorem implies
\begin{equation}
	I = F(\delta+iT,g)^2 + \frac{1}{2\pi i}\int_{(-\delta)}F(\underbrace{s+w}_{\in i\R},g)^2 \frac{\exp(w^2)}{w}dw\redd{, \quad \text{for } s=\delta+iT.} \nonumber
\end{equation}
In particular,
\begin{equation}
	\abs{F(\delta+iT)}^2 \ll \No(\n)^{\epsilon}\No(\n_0)^2 + \int_{-\infty}^{\infty} \abs{F(iT+iv,g)}^2 \frac{\exp(-v^2)}{\abs{-\delta+iv}}dv. \label{eq:first_bound_for_avi}
\end{equation}

On the other hand, the \red{residue} theorem also implies that
\begin{equation}
	F(iT_0,g)^2 = \frac{1}{2\pi i} \int_{\mathcal{C}} F(iT_0+u,g)^2\frac{\exp(u^2)}{u}du, \nonumber
\end{equation}
where $\mathcal{C}$ is the rectangle with corners $(\pm \delta,\pm i 2\log(\abs{T}_{\infty}))$. 

Using Lemma~\ref{lm:prelim_est_F} we estimate the integral over the small sides of the rectangle by
\begin{align}
	& \frac{\pm 1}{2\pi i} \int_{-\delta}^{\delta} F(i(T_0\pm i 2 \log(\abs{T}_{\infty})+u,g)^2 \frac{\exp([\pm i 2\log(\abs{T}_{\infty})+u]^2)}{\pm i 2\log(\abs{T}_{\infty})+u}du \nonumber \\
	&\ll \delta \frac{\exp(-4\log(\abs{T}_{\infty})^2)}{\log(\abs{T}_{\infty})}\sup_{u\in [-\delta,\delta]}\abs{F(i(T_0 \pm 2\log (\abs{T}_{\infty}))+u,g)}^2 \nonumber \\
	&\ll \exp(-\log(\abs{T}_{\infty})^2) \No(\n)^{8\delta+\epsilon}\abs{y}_{\infty}^{-\delta} \Big( \No(\imath)\abs{y}_{\infty}^{-1}\red{+\No(\imath)^{\frac{1}{2}}\No(\n_0(g))\abs{y}_{\infty}^{-\frac{1}{2}}}+\No(\n_0(g)) \Big) \nonumber \\
	&=  \mathcal{E}(y). \nonumber
\end{align} 
The error can be simplified to
\begin{equation}
	\mathcal{E}(y) = \exp(-\log(\abs{T}_{\infty})^2) \No(\n)^{\epsilon} \Big( \No(\imath)\abs{y}_{\infty}^{-1}\red{+\No(\imath)^{\frac{1}{2}}\No(\n_0(g))\abs{y}_{\infty}^{-\frac{1}{2}}}+\No(\n_0(g)) \Big). \nonumber
\end{equation}

The integral over the left side of the rectangle is estimated using the functional equation of $E$. Recall that \red{according to \cite[(5.15)]{GJ79} and \eqref{eq:everything_about_cs}} the functional equation \red{of $E$} reads $E(s,g) = c(s)\hat{E}(-s,g)$. Since the constant term satisfies the same functional equation, we conclude that
\begin{equation}
	F(s,g) = c(s) \hat{F}(-s,g). \nonumber
\end{equation}
We compute
\begin{eqnarray}
	&& \frac{-1}{2\pi i}\int_{-2\log(\abs{T}_{\infty})}^{2\log(\abs{T}_{\infty})} F(i(T_0+u)-\delta,g)^2\frac{\exp([iu-\delta]^2)}{iu-\delta}du \nonumber \\
	&=& \frac{1}{2\pi i}\int_{-2\log(\abs{T}_{\infty})}^{2\log(\abs{T}_{\infty})} F(i(T_0-u)-\delta,g)^2\frac{\exp([iu+\delta]^2)}{iu+\delta}du \nonumber \\
	&=& \frac{1}{2\pi i}\int_{-2\log(\abs{T}_{\infty})}^{2\log(\abs{T}_{\infty})}c(i(T_0-u)-\delta)^2 \hat{F}(i(-T_0+u)+\delta,g)^2\frac{\exp([iu+\delta]^2)}{iu+\delta}du. \nonumber
\end{eqnarray}

This implies 
\begin{eqnarray}
	F(iT_0,g)^2 \ll \int_{-2\log(\abs{T}_{\infty})}^{2\log(\abs{T}_{\infty})}&& \bigg[ \abs{F(i(T_0+u)+\delta,g)}^2 \nonumber \\
		&&\quad +\abs{c(i(T_0-u)-\delta)}^2\abs{\hat{F}(i(-T_0+u)+\delta,g)}^2\bigg]\frac{\exp(-u^2)}{\abs{iu+\delta}}du  +\mathcal{E}(y). \nonumber
\end{eqnarray}
One checks that
\begin{equation}
	\abs{c(i(T_0-u)-\delta)}^2 \ll_{F,\epsilon} \abs{T}_{\infty}^{\epsilon} \No(\n)^{\epsilon}. \nonumber
\end{equation}
The bound \eqref{eq:first_bound_for_avi} holds for $F$ as well as $\hat{F}$ and the integral appearing in it can be truncated. This \red{leads} to
\begin{eqnarray}
	&& F(iT_0,g)^2  \nonumber \\
	&\ll& \No(\n)^{\epsilon} \abs{T}_{\infty}^{\epsilon}\int_{\abs{u}\leq 2\log(\abs{T}_{\infty})} \int_{\abs{v}\leq 2\log(\abs{T}_{\infty}) } \bigg[\abs{F(i(T_0+u+v,g)}^2  \nonumber \\
	&& \qquad\qquad\qquad\qquad\qquad\qquad+\abs{\hat{F}(i(-T_0+u+v),g)}^2 \bigg] \frac{\exp(-v^2)}{\abs{-\delta+iv}}dv\frac{\exp(-u^2)}{\abs{\delta+iu}}du \nonumber \\
	&&\qquad\qquad\qquad\qquad\qquad\qquad+ \No(\n)^{\epsilon} \abs{T}_{\infty}^{\epsilon}\No(\n_0)^2 + \mathcal{E}(y)\nonumber  \\
	&\ll& \No(\n)^{\epsilon} \abs{T}_{\infty}^{\epsilon}\int_{\abs{v}\leq 4\log(\abs{T}_{\infty}) } \bigg[\abs{F(i(T_0+v,g)}^2 + \abs{\hat{F}(i(-T_0+v),g)}^2 \bigg] \nonumber \\
	&&\qquad\qquad\qquad \cdot \int_{\abs{u}\leq 2\log(\abs{T}_{\infty})} \frac{\exp(-(v-u)^2)\exp(-u^2)}{\abs{-\delta+i(v-u)}\abs{\delta+iu}}du dv + \No(\n)^{\epsilon} \abs{T}_{\infty}^{\epsilon}\No(\n_0)^2 + \mathcal{E}(y).\nonumber
\end{eqnarray} 
Elementary estimates reveal that the $u$-integral can be bounded by $\abs{T}_{\infty}^{\epsilon}$. This concludes the proof.
\end{proof}

\begin{cor} \label{cor:int_to_ind_for_eis}
Let $g=a(\theta_i)g'n(x)a(y)$ with $g'\in\mathcal{J}_{\n}$ and $n(x) a(y)\in \mathcal{F}_{\n_2}$. If $\log(\No(\n)) \ll \log(\abs{T}_{\infty})$, then
\begin{eqnarray}
	E(iT_0,g)^2 &\ll& \No(\n)^{\epsilon}\abs{T}_{\infty}^{\epsilon}\int_{\abs{v}\leq 4\log(\abs{T}_{\infty}) } \bigg[\abs{E(i(T_0+v),g)}^2 + \abs{\hat{E}(i(-T_0+v),g)}^2 \bigg]dv \nonumber \\
	&& +\No(\n)^{\epsilon}\abs{T}_{\infty}^{\epsilon}\No(\n_0)^2+   \No(\n)^{\epsilon}\abs{T}_{\infty}^{\epsilon}\No(\n_1)\abs{y}_{\infty}. \nonumber
\end{eqnarray}
\end{cor}
\red{\begin{proof}
This follows directly from Lemma~\ref{lm:ind_via_av_bounds} after observing that
\begin{equation}
	\abs{[v^{\circ}(it)](g)} \leq H(g)^{\frac{1}{2}} \ll_F \No(\n_1)^{\frac{1}{2}} \abs{y}_{\infty}^{\frac{1}{2}}. \label{eq:rough_bound_const_term}
\end{equation} 
Here we used \eqref{eq:special_desc_of_H} to obtain the last inequality.
\end{proof}}

\section{The proof of the main theorem} \label{sec:endgame}

Fix two Hecke characters $\chi_1$ and $\chi_2$. We are ready to prove upper bounds for the Eisenstein series $E(iT_0,g)$ associated to a new vector in $\chi_1\boxplus \chi_2$. Let $\mathfrak{l}$ to be the conductor of $\chi_1\chi_2^{-1}$ and define $(T)_{\nu}= (\max(\frac{1}{2},\abs{t_{\nu}+T_0}))$. Further, assume 
\begin{equation}
	\log(\No(\n)) \ll \log(\abs{T}_{\infty}) \text{ and } \log(\No(\mathfrak{l}))\ll \log(\abs{T}_{\infty})^{1-\delta}, \nonumber
\end{equation}
where $\mathfrak{n}$ is the conductor of $\chi_1\boxplus \chi_2$. Thus, we are exactly in the setting of Theorem~\ref{th:main_th_3}, which we will now prove.

\begin{proof}[Proof of Theorem~\ref{th:main_th_3}]
\red{Note that Lemma~\ref{lm:Eisenstein_reduction} provides the first part of  Theorem~\ref{th:main_th_3}.} Thus, we assume that $g=a(\theta_i)g'h_{\n}n(x)a(y)$ with $n(x)a(y) \in \mathcal{F}_{\n_2}$ and $g'h_{\n}\in \mathcal{J}_{\n}$.

First, let us assume that $\abs{y}_{\infty} \leq \abs{T}_{\infty}^{\frac{1}{4}}\No(\n_2)^{-\frac{1}{2}}$. Then we put $\eta\asymp \log(\abs{T}_{\infty})^{-1-2\delta}$ and find a covering 
\begin{equation}
	\bigcup_{i\in I} U_i = (-4\log(\abs{T}_{\infty}),4\log(\abs{T}_{\infty})) \nonumber
\end{equation}
with open intervals $U_i$ of length $\eta$. It is clear that we can do so with $\sharp I \ll \abs{T}_{\infty}^{\epsilon}$. Then we can use Corollary~\ref{cor:int_to_ind_for_eis} to establish
\begin{eqnarray}
	E(iT_0,g)^2 &\ll& \No(\n)^{\epsilon}\abs{T}_{\infty}^{\epsilon}\int_{\abs{v}\leq 4\log(\abs{T}_{\infty}) } \bigg[\abs{E(i(T_0+v),g)}^2 + \abs{\hat{E}(i(-T_0+v),g)}^2 \bigg]dv \nonumber \\
	&& +\No(\n)^{\epsilon}\abs{T}_{\infty}^{\epsilon}\No(\n_0)^2+\abs{T}_{\infty}^{\frac{1}{4}+\epsilon}\No(\n_0^2\n_2)^{\frac{1}{2}+\epsilon}. \nonumber
\end{eqnarray}
Further, we use the covering $\{U_i\}_{i\in I}$ to cut the integral into pieces. To each piece we can apply Proposition~\ref{pr:after_amp_for_eis}. This leads to
\begin{eqnarray}
	E(iT_0,g)^2 &\ll&  \No(\n)^{\epsilon}\abs{T}_{\infty}^{\epsilon} \No(\n_0\mathfrak{m}_1)\bigg[ \abs{T}_{\infty}^{\frac{3}{4}}\No(\n_2)^{\frac{1}{2}}+\abs{T}_{\R}^{\frac{3}{8}}\abs{T}_{\C}^{\frac{7}{8}}\No(\n_2)^{\frac{1}{2}} \bigg] \nonumber \\
	&&  +\No(\n)^{\epsilon}\abs{T}_{\infty}^{\epsilon}\No(\n_0)^2. \nonumber
\end{eqnarray}
\red{Note that in this case \eqref{eq:rough_bound_const_term} implies that the constant term can be bounded by $\No(\n_0)^{\frac{1}{2}+\epsilon}\No(\n_2)^{\frac{1}{4}+\epsilon}\abs{T}_{\infty}^{\frac{1}{4}+\epsilon}$. Thus, we can absorb it in the error term. In particular, the asymptic formula stated in Theorem~\ref{th:main_th_3} degenerates to an upper bound.}

Second, let $\abs{y}_{\infty} > \abs{T}_{\infty}^{\frac{1}{4}}\No(\n_2)^{-\frac{1}{2}}$. Then Proposition~\ref{pr:whit_exp_est_bad} implies
\begin{equation}
	E(iT_0,g) = [v^{\circ}(iT_0)](g) + c(iT_0)[\hat{v}^{\circ}(-iT_0)](g) +O\bigg(\No(\n)^{\epsilon}\abs{T}_{\infty}^{\epsilon} \No(\n_0\mathfrak{m}_1)^{\frac{1}{2}}\No(\n_2)^{\frac{1}{4}}\abs{T}_{\infty}^{\frac{3}{8}} \bigg). \nonumber
\end{equation}
This completes the proof.
\end{proof}

\appendix

\section{Averaging non-unitary Whittaker new vectors}

In this appendix we extend \cite[Proposition~2.9]{Sa15} to allow non-unitary principal series representations. This is needed to deal with the Whittaker expansion of Eisenstein series for general $s$. The computations in this appendix rely heavily on the explicit expressions for the constants $c_{t,l}(\mu)$ (defined in \cite[(1.6)]{As17_1}) given in \cite[Lemma~2.2,2.3]{As17_1}.  \red{We will mostly stick to the notation of this paper, with some additions from \cite{As17_1}. Recall for example the matrices $g_{t,l,v}\in G(F_{\p})$ defined below \cite[(1.4)]{As17_1} and the set $\mathfrak{X}_k = \{\xi\colon F^{\times}\to S^1\colon \xi(\varpi_{\p})=1 \text{ and } a(\xi)\leq k\}$. All the computations in this appendix are done in a fixed non-archimedean field $F_{\p}$.} 

For the sake of exposition we consider three cases.

\begin{lemma}
Let $\pi_{\p}=\chi_1\boxplus\chi_2$ be a principal series representation of $G(F_{\p})$ with $a(\chi_1)>a(\chi_2)=0$. In this case we have $a(\omega_{\pi}) = a(\chi_1) = n_{\p} = m_{\p}$. For $0\leq l\leq n_{\p}$ and $t=-(l+n_{\p})+r$ we have
\begin{equation}
	\int_{v\in \op_{\p}^{\times}} \abs{W_{\pi_{\p}}(g_{t,l,v})}^2 d^{\times}v \ll \begin{cases}
	\zeta_{F_{\p}}(1)q_{\p}^{-r}(\abs{\chi_1(\varpi_{\p}^{r+n_{\p}})}^2+\abs{\chi_1(\varpi_{\p}^{-r-n_{\p}})}^2) &\text{ if $r\geq 0$}, \\
	0 &\text{ if $r<0$.} \end{cases} \nonumber
\end{equation}
\end{lemma}
\begin{proof}
By \cite[(1.6)]{As17_1} and character orthogonality it is clear that
\begin{equation}
	S_{t,l} = \int_{v\in \op_{\p}^{\times}} \abs{W_{\pi_{\p}}(g_{t,l,v})}^2 d^{\times}v = \sum_{\mu\in\mathfrak{X}_l}\abs{c_{t,l}(\mu)}^2. \label{eq:starting_point}
\end{equation}
We insert the expressions for $c_{t,l}(\mu)$ given in \cite[Lemma~2.3]{As17_1} and consider several cases.

The easiest situation occurs when $l=0$. In this case we have
\begin{equation}
	S_{t,0} = \begin{cases}
		q^{-r}\abs{\chi_1(\varpi^{-r-n_{\p}})}^2 &\text{ if $r\geq 0$,} \\
		0 &\text{ if $r<0$.}
	\end{cases}\nonumber
\end{equation}

For $0<l<n_{\p}$\red{,} we observe that $a(\mu\pi_{\p}) = n_{\p}+a(\mu)$ and obtain
\begin{eqnarray}
	S_{t,l} &=&\sum_{\substack{\mu\in \mathfrak{X}_l, \\ \red{t=-a(\mu\omega_{\pi,\p})-l}}}\zeta_{F_{\p}}(1)^2 q_{\p}^{-l}\abs{\chi_1(\varpi^{l-n_{\p}-a(\mu)})}^2 \nonumber \\
	&\leq&\red{ \begin{cases} \zeta_{F_{\p}}(1)^2\sharp\mathfrak{X}_l q_{\p}^{-l}\max(1,\abs{\chi_1(\varpi^{-\red{n_{\p}}})}^2) &\text{ if $r=0$,} \\ 0&\text {else.} \end{cases}} \nonumber
\end{eqnarray}
\red{Recalling $\sharp\mathfrak{X}_l = \zeta_F(1)^{-1}q_{\p}^l$ yields $S_{t,l} \leq \zeta_{F_{\p}}(1)\max(1,\abs{\chi_1(\varpi^{-n_{\p}})}^2)$.}

At last we look at $l=n$. One has
\begin{eqnarray}
	S_{t,n} &=& \sum_{\substack{\mu\in\mathfrak{X}_{n_{\p}}, \\ \mu \neq \omega_{\pi_{\p}}^{-1}, \\ t=-a(\mu\red{\omega_{\pi,\p}})-n_{\p}}} \zeta_{F_{\p}}(1)^{\red{2}}q_{\p}^{-n_{\p}} \abs{\chi_1(\varpi^{\red{r}})}^2 \nonumber \\
	&&\qquad\qquad + \underbrace{\delta_{t=-n_{\p}-1}\zeta_{F_{\p}}(1)^2 q_{\p}^t\abs{\chi_1(\varpi^{-t-2})}^2+ \delta_{t\geq -n_{\p}} q_{\p}^{-t-2n_{\p}}\abs{\chi_1(\varpi^{t+2n_{\p}})}^2}_{\leq q_{\p}^{-r}\abs{\chi_1(\varpi_{\p}^r)}^2}. \nonumber
\end{eqnarray}
To estimate the first sum we write $a(\mu\omega_{\pi_{\p}}) = n_{\p}-r$. This corresponds precisely to $t=-2n_{\p}+r$. The sum is empty for $r<0$. If $r\geq 0$, we use the trivial bound 
\begin{equation}
	\sharp \{ \mu\in \mathfrak{X}_{n_{\p}} \setminus \{\omega_{\pi}^{-1}\} \colon a(\mu\omega_{\pi_{\p}}) = n_{\p}-r \} \leq \zeta_{F_{\p}}(1)^{-1}q_{\p}^{n_{\p}-r}. \label{eq:trivial_char_count_bound}
\end{equation}
This yields
\begin{equation}
	S_{t,n} \leq \begin{cases}
		2\zeta_{F_{\p}}(1)q_{\p}^{-r}\abs{\chi_1(\varpi_{\p}^{r})}^2 &\text{ if $r\geq 0$,} \\
		0&\text{ if $r<0$.}
	\end{cases} \nonumber
\end{equation}

Combining \red{these three} estimates completes the proof.
\end{proof}

\begin{lemma}
Let $\pi_{\p}=\chi_1\boxplus\chi_2$ with $a(\chi_1) > a(\chi_2)>0$. For $0\leq l\leq n_{\p}$ and $t=-\max(2l,m_{\p}+l,n_{\p})+r$ we have
\begin{equation}
	\int_{v\in \op_{\p}^{\times}} \abs{W_{\pi_{\p}}(g_{t,l,v})}^2 d^{\times}v \ll \begin{cases}
	\zeta_{F_{\p}}(1)^{\red{2}}q_{\p}^{-r}\left(\abs{\chi_1(\varpi_{\p}^{r+a(\chi_1)})}^2+\abs{\chi_1(\varpi_{\p}^{-r-a(\chi_1)})}^2\right) &\text{ if $r\geq 0$}, \\
	0 &\text{ if $r<0$.} \end{cases} \nonumber
\end{equation}
\end{lemma}
Note that this covers also the analogous case $0<a(\chi_1)<a(\chi_2)$. We remark that the exponents appearing inside of $\chi_1$ were not optimized.
\begin{proof}
For convenience we write $a(\chi_i) = a_i$. Note that in this situation $n=a_1+a_2$ and $m=a_1$. The strategy is to start from \eqref{eq:starting_point} and insert the expressions from \cite[Lemma~2.2]{As17_1}. Let us first deal with some easy cases.

If $0\leq l<a_2$, we have $t=-n_{\p}+r$ and 
\begin{equation}
	S_{t,l} = \sum_{\substack{\mu\in \mathfrak{X}_l', \\ t=-n_{\p}}} \zeta_{F_{\p}}(1)^{\red{2}} q_{\p}^{-l}\abs{\chi_1(\varpi_{\p}^{a_2-a_1})}^2 \leq \begin{cases} \zeta_{F_{\p}}(1)^{\red{2}} \abs{\chi_1(\varpi_{\p}^{a_2-a_1})}^2 &\text{ if $r=0$,} \\ 0 &\text{ else.} \end{cases} \label{eq:l_smaller_a_2}
\end{equation}
\red{Here we used $\sharp \mathfrak{X}'_{l} = \sharp \{ \mu\in \mathfrak{X}_l \colon a(\mu) = l \} \leq q^l$.} 

Similarly, if $a_1<l\leq n$, we have $t=-2l+r$ and
\begin{equation}
	S_{t,l} = \sum_{\substack{\mu\in \mathfrak{X}_l', \\ t=-2l}} \zeta_{F_{\p}}(1)^{\red{2}} q_{\p}^{-l}\abs{\chi_1(\varpi_{\p}^{0})}^2 \leq \begin{cases} \zeta_{F_{\p}}(1)^{\red{2}} &\text{ if $r=0$,} \\ 0 &\text{ else.} \end{cases} \label{eq:l_bigger_a_1}
\end{equation}

Finally, if $a_2<l<a_1$, we have $t=-(m_{\p}+l)+r$ and
\begin{equation}
	S_{t,l} = \sum_{\substack{\mu\in \mathfrak{X}_l', \\ t=-a_1-l}} \zeta_{F_{\p}}(1)^{\red{2}} q_{\p}^{-l}\abs{\chi_1(\varpi_{\p}^{l-a_1})}^2 \leq \begin{cases} \zeta_F(1)^{\red{2}} \abs{\chi_1(\varpi_{\p}^{l-a_1})}^2 &\text{ if $r=0$,} \\ 0 &\text{ else.} \end{cases} \nonumber
\end{equation}

Next we consider $l=a_2$. Note that this implies $t=-n_{\p}+r$ and $a(\mu\chi_1)=a_1$ as well as $a(\mu\chi_2) = a_{\red{2}} -r$. The explicit expressions for $c_{t,l}(\mu)$ yield
\begin{eqnarray}
	S_{t,a_2} &=& \sum_{\substack{\mu\in\mathfrak{X}_{a_2}',\\ \mu\vert_{\op_{\p}^{\times}}\neq \chi_2\vert_{\op_{\p}^{\times}},\\ a(\mu\chi_2)=a_2-r}} \zeta_{F_{\p}}(1)^2 q_{\p}^{-a_2}\abs{\chi_1(\varpi_{\p}^{a_2-a_1-r})}^2\nonumber \\
	&&\qquad +\underbrace{\delta_{r=a_2-1}\zeta_{F_{\p}}(1)^2q_{\p}^{-r-2}\abs{\chi_1(\varpi_{\p}^{\red{a_2-a_1-r}})}^2+\delta_{r\geq a_2}q_{\p}^{-r}\abs{\chi_1(\varpi_{\p}^{\red{-t-a_1}})}^2}_{\leq q_{\p}^{-r}\red{(\abs{\chi_1(\varpi_{\p}^{a_2-a_1-r})}^2+\abs{\chi_1(\varpi_{\p}^{-t-a_1})}^2)}} \nonumber \\
	&\red{\leq}& \zeta_{F_{\p}}(1)^{\red{2}} q_{\p}^{-r}\red{(\abs{\chi_1(\varpi_{\p}^{a_2-a_1-r})}^2+\abs{\chi_1(\varpi_{\p}^{-t-a_1})}^2)}. \nonumber
\end{eqnarray}
\red{Here we used \eqref{eq:trivial_char_count_bound} to estimate the $\mu$-sum.}

The last case to look at is $l=a_1$. We have $t=-2l+r$, $a(\mu\chi_2) = a_1$ and $a(\mu\chi_1) = a_1-r$. Therefore,
\begin{eqnarray}
	S_{t,a_2} &=& \sum_{\substack{\mu\in\mathfrak{X}_{a_1}',\\ \mu\vert_{\op_{\p}^{\times}}\neq \chi_1\vert_{\op_{\p}^{\times}},\\ a(\mu\chi_1)=a_1-r}} \zeta_{F_{\p}}(1)^2 q_{\p}^{-a_1}\abs{\chi_1(\varpi_{\p}^{r})}^2 +\underbrace{\delta_{r=a_1-1}\zeta_{F_{\p}}(1)^2q_{\p}^{-r-2}\abs{\chi_1(\varpi_{\p}^{r})}^2+\delta_{r\geq a_1}q^{-r}\abs{\chi_1(\varpi_{\p}^{t\red{+a_1}})}^2}_{\leq q^{-r}(\abs{\chi_1(\varpi_{\p}^{r})}^2+\abs{\chi_1(\varpi_{\p}^{t\red{+a_1}})}^2)  } \nonumber \\
	&\red{\leq}& \zeta_{F_{\p}}(1)^{\red{2}} q_{\p}^{-r}\abs{\chi_1(\varpi_{\p}^{r})}^2 + q_{\p}^{-r}\abs{\chi_1(\varpi_{\p}^{t\red{+a_1}})}^2. \nonumber
\end{eqnarray}

This covers all cases and the proof is complete. 
\end{proof}

\begin{lemma}
Let $\pi_{\p}=\chi_1\boxplus\chi_2$ with $a(\chi_1) = a(\chi_2)>0$. For $0\leq l\leq n_{\p}$ and $t=-\max(2l,m+l,n_{\p})+r$ we have
\begin{equation}
	\int_{v\in \op_{\p}^{\times}} \abs{W_{\pi_{\p}}(g_{t,l,v})}^2 d^{\times}v \ll \begin{cases}
	\zeta_{F_{\p}}(1)^{\red{2}}rq_{\p}^{-\frac{r}{2}}\left(\abs{\chi_1(\varpi_{\p}^{r+n_{\p}})}^2+\abs{\chi_1(\varpi_{\p}^{-r-n_{\p}})}^2\right) &\text{ if $r\geq 0$}, \\
	0 &\text{ if $r<0$.} \end{cases} \nonumber
\end{equation}
\end{lemma}

\begin{proof}
For simplicity we write $a=a(\chi_1)$. In particular $n=2a$. Equations \eqref{eq:l_smaller_a_2} and \eqref{eq:l_bigger_a_1} remain true and cover $l\neq a$. So let us assume $l=a$. If there is a character $\mu$ such that $a(\mu\chi_1)=a(\mu\chi_1)=0$, we get the contribution
\begin{eqnarray}
	&&\delta_{r=n_{\p}-2}q^{-2-a}\zeta_{F_{\p}}(1)^{\red{2}}+\delta_{r=n_{\p}-1}q_{\p}^{-1-a}\abs{\chi_1(\varpi_{\p})+\chi_1(\varpi_{\p}^{-1})}^2 \nonumber \\
	&&\quad +\delta_{r>n_{\p}}q_{\p}^{-t-a}\zeta_{F_{\p}}(1)^{\red{2}}\abs{ -q_{\p}^{-1}\zeta_{F_{\p}}(1)^{-1}(\chi_1(\varpi_{\p}^{t+2})+\chi_2(\varpi_{\p}^{t+2}))+\zeta_{F_{\p}}(1)^{-2}\sum_{l=0}^t \chi_1(\varpi_{\p}^l)\chi_2(\varpi_{\p}^{t-l})}^{\red{2}} \nonumber  \\
	&& \ll \red{t}q_{\p}^{-\frac{r}{2}} (\abs{\chi_1(\varpi_{\p}^{-r+n_{\p}-\red{3}})}^2+\abs{\chi_1(\varpi_{\p}^{r-n_{\p}+\red{3}})}^2). \nonumber
\end{eqnarray}
The other exceptional contribution comes from a character $\mu$ satisfying $a(\mu\chi_j)\neq a(\mu\chi_i)=0$. In this case we write $a(\mu\chi_j) = a-r_0$. By assumption we have $r_0<a$ so that for $r\geq a-1+r_0$ we have $r_0\leq \frac{r}{2}$. This situation contributes
\begin{eqnarray}
	&&\delta_{r=n_{\p}-a(\mu\chi_j)-1}\zeta_{F_{\p}}(1)^2q_{\p}^{-1-a}\abs{\chi_i(\varpi^{\red{n_{\p}-r-2}})}^2 + \delta_{r\geq n_{\p}-a(\mu\chi_j)}q_{\p}^{-a-\red{t}-a(\mu\chi_j)}\abs{\chi_i(\varpi_{\p}^{\red{-1}})}^2 \nonumber \\
	&&\leq q^{-\frac{r}{2}}(\abs{\chi_i(\varpi^{\red{-1}})}^2+\abs{\chi_i(\varpi^{\red{a}})}^2).\nonumber
\end{eqnarray}

At last we deal with the contribution of generic $\mu$. Together with the cases above we have
\begin{eqnarray}
	S_{t,a} &=& \underbrace{\sum_{\substack{\mu\in\mathfrak{X}_a',\\ a(\mu\chi_i)\neq 0,\\t=-a(\mu\chi_1)-a(\mu\chi_2)}}\zeta_{F_{\p}}(1)^2q_{\p}^{-l}\abs{\chi_1(\varpi_{\p}^{a(\mu\chi_2)-a(\mu\chi_1)})}^2}_{=S_g(r)} \nonumber \\
	&&\qquad\qquad + O(\max(\red{t},1)q_{\p}^{-\frac{r}{2}}\left(\abs{\chi_1(\varpi_{\p}^{r+n_{\p}})}^2+\abs{\chi_1(\varpi_{\p}^{-r-n_{\p}})}^2)\right). \nonumber
\end{eqnarray}
First, we observe that the generic characters only contribute when $r<n-1$. We write $a(\mu\chi_i)=a-r_i$ for $0\leq r_1,r_2<a$ and define
\begin{equation}
	\mathfrak{X}_{r_1,r_2}= \{\mu\in\mathfrak{X}_a' \colon a(\mu\chi_i) = a-r_i  \}. \nonumber
\end{equation}
One has the trivial bound $\sharp\mathfrak{X}_{r_1,r_2} \leq \zeta_{F_{\p}}(1)^{-1}q_{\p}^{a-\max(r_1,r_2)}$. Therefore, by ordering the summation in $S_g(r)$ accordingly we obtain
\begin{eqnarray}
	S_g(r) &=& \sum_{\substack{0\leq r_1,r_2< a, \\ r_1+r_2=r}} \sharp\mathfrak{X}_{r_1,r_2} \zeta_{F_{\p}}(1)^2q_{\p}^{-l}\abs{\chi_1(\varpi_{\p}^{a(\mu\chi_2)-a(\mu\chi_1)})}^2 \nonumber \\
	&\leq& \zeta_{F_{\p}}(1)rq_{\p}^{-\frac{r}{2}}(\abs{\chi_1(\varpi_{\p}^{r})}^2+\abs{\chi_1(\varpi_{\p}^{-r})}^2). \nonumber
\end{eqnarray}
\red{The stated inequality follows after putting everything together.}
\end{proof}

The last three lemmata together imply an extension of \cite[Proposition~2.10]{Sa15}. 
\begin{prop} \label{pr:supp_non_unitary_W} 
Let $\pi_{\p}= \chi_1\boxplus\chi_2$ such that $\omega_{\pi_{\p}}(\varpi_{\p})=1$ and $g\in GL_2(F_{\p})$ such that $t(g) = -\max(2l(g),l(g)+m_{\p},n_{\p})+r(g)$. Then we have
\begin{enumerate}
\item If $W_{\pi_{\p}}(g)\neq 0$\red{,} then $r(g)\geq 0$.
\item If $r(g) \geq 0$\red{,} then
\begin{equation}
	\left( \int_{\op_{\p}^{\times}}\abs{W_{\pi_{\p}}(a(v)g)}^2d^{\times}v\right)^{\frac{1}{2}} \ll \zeta_{F_{\p}}(1)\max(1,\red{\sqrt{r}})q_{\p}^{-\red{\frac{r}{4}}}\left(\abs{\chi_1(\varpi_{\p}^{r+n_{\p}})}+\abs{\chi_1(\varpi_{\p}^{-r-n_{\p}})}\right). \nonumber
\end{equation}
\end{enumerate}
\end{prop}
\begin{proof}\red{
First we note that there is $w\in \mathfrak{o}_{\p}$, $z\in Z(F_{\p})$, $x\in F_{\p}$ and $k\in K_{1,\p}(n_{\p})$  such that
\begin{equation}
	g=zn(x) g_{t(g),l(g),w}k. \nonumber
\end{equation}
By the transformation behavior of $W_{\pi_{\p}}$ we get
\begin{equation}
	\abs{W_{\pi_{\p}}(a(v)g)}^2 = \abs{W_{\pi_{\p}}(g_{t(g),l(g),wv^{-1}})}^2. \nonumber
\end{equation}
With this at hand the proposition follows from the previous lemmata.
}\end{proof}

\bibliographystyle{plain}
\bibliography{bibliography} 

\end{document}